\PassOptionsToPackage{unicode}{hyperref}
\PassOptionsToPackage{naturalnames}{hyperref}
\documentclass[a4paper]{amsart}
\usepackage{latexsym,bm,stmaryrd}
\usepackage[a4paper,hmargin=26mm,vmargin=28mm]{geometry}
\usepackage{latexsym,bm}
\usepackage[svgnames]{xcolor}
\usepackage{xparse}
\usepackage{listings}
\usepackage{tikz}
\usetikzlibrary{matrix}
\usepackage[utf8]{inputenc}
\usepackage[english]{babel}
\usepackage[T1]{fontenc}
\usepackage{amsmath, amsfonts, amssymb, amsthm, mathrsfs, pb-diagram}
\usepackage{lmodern}
\usepackage{fullpage}
\usepackage{microtype}
\usepackage{enumerate}
\usepackage{mathtools}
\usepackage{enumitem}
\setlist[enumerate]{label=\alph*\upshape), nolistsep}

\usepackage{todonotes}

\NewDocumentCommand\Cmd{ sm }{\textsf{\textbackslash #2}\IfBooleanT{#1}{$\{\ldots\}$}}

\usepackage{etoolbox}
\usepackage{aliascnt}

\usepackage[pagebackref=false]{hyperref}

\newcommand\enumref[2]{\hyperref[#2]{\autoref*{#1}(\autoref*{#2})}}

\usepackage[misc]{ifsym}

\def\NewTheorem#1{%
  \newaliascnt{#1}{equation}%
  \newtheorem{#1}[#1]{#1}%
  \aliascntresetthe{#1}%
  \expandafter\def\csname #1autorefname\endcsname{#1}%
}

\newcommand\blam{{\boldsymbol\lambda}}

\newcommand\bmu{{\boldsymbol\mu}}

\DeclareMathOperator\Shape{Shape}

\let\<=\langle
\let\>=\rangle
\def\({\big(}
\def\){\big)}

\def\Z{\mathbb{Z}}

\def\C{\mathbb{C}}

\def\t{\mathfrak{t}}
\def\s{\mathfrak{s}}
\def\u{\mathfrak{u}}
\def\v{\mathfrak{v}}

\def\tlam{\t^{\blam}}
\def\O{\mathcal{O}}

\def\lam{\lambda}

\def\Sym{\mathfrak{S}}
\def\eps{\varepsilon}

\newcommand\HH{\mathscr{H}}

\def\P{\mathscr{P}}

\def\bQ{\mathbf{Q}}

\def\veps{\varepsilon}

\newcommand\minum{\text{min}}
\newcommand{\sft}[1]{\langle #1 \rangle}

\def\m{{\rm mid}}

\def\K{\mathscr{K}}

\DeclareMathOperator\Hom{Hom}

\DeclareMathOperator\cha{char}

\DeclareMathOperator\Tr{Tr}

\DeclareMathOperator\rank{rank}

\DeclareMathOperator\res{res}

\DeclareMathOperator\Std{Std}

\title[Seminormal bases for cyclotomic Hecke algebras]{Seminormal basis for the cyclotomic Hecke algebra of type $G(r,p,n)$}
\author{Jun Hu}
\address[Jun Hu]{Key Laboratory of Algebraic Lie Theory and Analysis of Ministry of Education\\
School of Mathematics and Statistics\\ Beijing Institute of Technology\\ Beijing, 100081, P.R.~China}
\email{junhu404@bit.edu.cn}

\author[Corresponding author]{Shixuan Wang\textsuperscript{\Letter}}\thanks{\Letter Shixuan Wang \qquad sxwang@zjhu.edu.cn}
\address[Shixuan Wang]{Department of Mathematics\\
Huzhou University\\
Zhejiang Huzhou, 313000, P.R. China}
\email{sxwang@zjhu.edu.cn}

\numberwithin{equation}{section}
\newtheorem{prop}[equation]{Proposition}
\newtheorem{thm}[equation]{Theorem}
\newtheorem{cor}[equation]{Corollary}
\newtheorem{conj}[equation]{Conjecture}

\newtheorem{lem}[equation]{Lemma}

\theoremstyle{definition}
\newtheorem{dfn}[equation]{Definition}
\theoremstyle{remark}
\newtheorem{rem}[equation]{Remark}

\begin{document}

\begin{abstract} The cyclotomic Hecke algebra $\HH_{r,p,n}$ of type $G(r,p,n)$ (where $r=pd$) can be realized as the $\sigma$-fixed point subalgebra of certain cyclotomic Hecke algebra $\HH_{r,n}$ of type $G(r,1,n)$ with
some special cyclotomic parameters, where $\sigma$ is an automorphism of $\HH_{r,n}$ of order $p$. In this paper we prove a number of rational properties on the $\gamma$-coefficients arising in the construction of the
seminormal basis for the semisimple Hecke algebra $\HH_{r,n}$. Using these properties, we construct a seminormal basis for the semisimple Hecke algebra $\HH_{r,p,n}$ in terms of the seminormal basis for the semisimple Hecke
algebra $\HH_{r,n}$. The proof relies on some careful and subtle study on some rational and symmetric properties of some quotients and/or products of $\gamma$-coefficients of $\HH_{r,n}$. As applications, we obtain an
explicit basis for the center $Z(\HH_{r,p,n})$ of $\HH_{r,p,n}$ and an explicit basis for the $\sigma$-twisted $k$-center $Z(\HH_{r,n})^{(k)}$ of $\HH_{r,n}$ for each $k\in\Z/p\Z$.
\end{abstract}

\maketitle
\setcounter{tocdepth}{1}
\tableofcontents

\section{Introduction}

Complex reflection groups are some generalizations of the real reflection groups (also known as Coxeter groups). Shephard and Todd \cite{ShTo} have given a classification of all the irreducible complex reflection groups. It
is given by an infinite series $\{G(r,p,n)|r,p,n\in\Z^{\geq 1},r=pd\}$ and $34$ exceptional groups. The complex reflection group of type $G(r,p,n)$ can be realized as the group consisting of all $n\times n$ monomial
matrices such
that each non-zero entry is a complex $r$th root of unity, and the product of all non-zero entries is a $d$th root of unity.

To each complex reflection group $W$, Brou\'e and Malle \cite{BM} introduced a deformation of the group algebra $\mathbb{\C}[W]$, called the cyclotomic Hecke algebra associated to $W$. These Hecke algebras generalized the
classical Iwahori-Hecke algebras associated to real reflection groups and played a key role in the study of modular representation theory of finite reductive groups over fields of non-defining characteristics. For the
complex reflection group of type $G(r,1,n)$, these Hecke algebras were also introduced by Ariki and Koike \cite{A1}. By \cite{A2} and \cite{R1}, the cyclotomic Hecke algebra $\HH_{r,p,n}$ of type $G(r,p,n)$ can be realized
as certain subalgebra of certain cyclotomic Hecke algebra $\HH_{r,n}$ of type $G(r,1,n)$ with some special cyclotomic parameters.

We now recall the definitions of these cyclotomic Hecke algebras. Let $R$ be a commutative domain and $1\neq q\in R^\times$. Let $({\rm Q}_{1},\cdots,{\rm Q}_{r})\in R^{r}$. The cyclotomic Hecke algebra $\HH_{r,n}(q;{\rm
Q}_{1},\cdots,{\rm Q}_{r})$ of type $G(r,1,n)$ with Hecke parameter $q$ and cyclotomic parameters ${\rm Q}_1,\cdots,{\rm Q}_r$ is the unital associative $R$-algebra with generators $T_0,T_1,\cdots,T_{n-1}$ and the following
defining relations:
$$\begin{aligned}
	&(T_{0}-{\rm Q}_{1})\cdots(T_{0}-{\rm Q}_{r})=0;\\
	&T_{0}T_{1}T_{0}T_{1}=T_{1}T_{0}T_{1}T_{0};\\
	&(T_{i}-q)(T_{i}+1)=0,\quad \forall\,1\leq i\leq n-1;\\
	&T_{i}T_{j}=T_{j}T_{i},\,\,\forall\,1\leq i<j-1<n-1;\\
	&T_{i}T_{i+1}T_{i}=T_{i+1}T_{i}T_{i+1},\,\,\forall\,1\leq i<n-1.
\end{aligned}$$
If $r=1$ or $r=2$, then the cyclotomic Hecke algebra $\HH_{r,n}(q;{\rm Q}_{1},\cdots,{\rm Q}_{r})$ becomes the Iwahori-Hecke algebras of type $A_{n-1}$ or type $B_n$ respectively.

Let $p\in\Z^{\geq 1}$. Assume that $r=pd$, where $d\in\Z^{\geq 1}$. Suppose that the commutative domain $R$ contains a primitive $p$-th root of unity $\varepsilon\in R$. In particular, this implies that $p1_{R}\in R^\times$
is invertible. Let $\bQ=(Q_1,\cdots,Q_d)\in R^{d}$ and set
$$\bQ^{\vee \varepsilon}:=\varepsilon\bQ\vee\varepsilon^{2}\bQ\vee\cdots\vee\varepsilon^{p}\bQ=(\varepsilon Q_1,\cdots,\varepsilon Q_d, \cdots, \varepsilon^{p} Q_1,\cdots,\varepsilon^{p} Q_d).$$
We denote by $\HH_{r,n}(\bQ^{\vee \varepsilon}):=\HH_{r,n}(q, \bQ^{\vee \varepsilon})$ the cyclotomic Hecke algebra of type $G(r,1,n)$ with Hecke parameter $q$ and cyclotomic parameters
$\bQ^{\vee \varepsilon}$. Henceforth, without causing ambiguity, we write $\HH_{r,n}:=\HH_{r,n}(\bQ^{\vee \varepsilon})$ for simplicity.

\begin{dfn}\text{\text (\cite{A2}, \cite[Appendix]{R1})} The cyclotomic Hecke algebra $\HH_{r,p,n}$ of type $G(r,p,n)$ is the subalgebra of  $\HH_{r,n}(\bQ^{\vee \varepsilon})$ generated by $T_{0}^{p},
T_{u}:=T_{0}^{-1}T_{1}T_{0}, T_{1}, T_{2}, \cdots, T_{n-1}$.
\end{dfn}
The special case $\HH_{2,2,n}$ is isomorphic to an Iwahori-Hecke algebra of type $D_n$.

In the semisimple case, Ariki and Koike \cite{A1} have constructed all the irreducible representations for the cyclotomic Hecke algebra $\HH_{r,n}$, generalizing Hoefsmit's  the seminormal construction \cite{Hoe1974} for
irreducible representation of the Iwahori-Hecke algebra associated to finite Weyl groups. In \cite{A2}, such construction was generalized to the cyclotomic Hecke algebra $\HH_{r,p,n}$ of type $G(r,p,n)$.

In the non-semisimple case, less is well-understood than in semisimple case. There have been some extensive study on the non-semisimple representation theory of the cyclotomic Hecke algebra $\HH_{r,n}$, see
\cite{Ariki:can,BK:GradedDecomp,HM10,HW,LM}, partly because these algebras are cellular algebras in the sense of Graham and Lehrer \cite{GL96}, see also \cite{DJM}. More recently, Brundan and Kleshchev \cite{BK:GradedKL}
have proved that each block of $\HH_{r,n}$ is isomorphic to a cyclotomic KLR algebra of type $A$, and thus was endowed with a non-trivial $\Z$-grading. The first author of this paper and Mathas \cite{HM10} have shown that
these algebras are $\Z$-graded cellular algebras, extending earlier result of Graham and Lehrer.

In \cite{Ma04}, Mathas constructed the seminormal basis for the entire semisimple cyclotomic Hecke algebra $\HH_{r,n}$, from which one can get seminormal bases for each irreducible representation of $\HH_{r,n}$. The
seminormal bases for the entire Hecke algebra is very useful in that it can be used to study the non-semisimple representation theory of $\HH_{r,n}$. For example, many important elements such as the KLR idempotent
$e(\mathbf{i})$ can be lifted to a local ring and defined via seminormal bases of the semisimple cyclotomic Hecke algebra $\HH_{r,n}$, see \cite[Lemma 4.2]{Ma08}. Thus one can make use of the powerful and much
better-understood semisimple representation theory for $\HH_{r,n}$ to study the non-semisimple representation theory of the cyclotomic Hecke algebra of type $G(r,1,n)$. In \cite{HuMathas:SeminormalQuiver}, the first author
of this paper and Mathas have used the seminormal bases for the entire Hecke algebra to construct an $x$-deformed version of the KLR presentation for the integral cyclotomic Hecke algebras and generalized
Brundan-Kleshchev's isomorphism to the integral setting, see \cite{EM} for a more recent development. Roughly speaking, the philosophy underlying \cite{HuMathas:SeminormalQuiver} is to rebuild the $\Z$-grading structure
using the seminormal bases for the entire Hecke algebras.

The modular representations of the Iwahori-Hecke algebra of type $D_n$ and of the cyclotomic Hecke algebras of type $G(r,p,n)$ have been studied in a number of references, see
\cite{Ge00,Ge07,GJ,Hu02,Hu03,Hu04,Hu07,Hu08,Hu09,HM09,HM12,P,W}. More recently, Rostam \cite{R1} have introduced the quiver Hecke algebra of type $G(r,p,n)$ as a fixed point subalgebra of the quiver Hecke algebra of type
$G(r,1,n)$, and showed that it is isomorphic to the cyclotomic Hecke algebra $\HH_{r,p,n}$ of type $G(r,p,n)$. In particular, this endows the cyclotomic Hecke algebra $\HH_{r,p,n}$ of type $G(r,p,n)$ with a non-trivial
$\Z$-grading. Using his result, the first author of this paper, Mathas and Rostam \cite{HMR} have shown that the cyclotomic Hecke algebra $\HH_{r,p,n}$ of type $G(r,p,n)$ are graded skew cellular. However, this result
doesn't allow us to directly connect the graded skew cellular structure with the semisimple representation theory of the cyclotomic Hecke algebra of type $G(r,p,n)$ due to the lack of a suitable $x$-deformed KLR
presentation for the integral cyclotomic Hecke algebra of type $G(r,p,n)$ and semisimple bases theory for $\HH_{r,p,n}$, which motivates our current study of the seminormal bases for $\HH_{r,p,n}$ in this work.

Another motivation of the current work comes from the paper \cite{HS} on the center of the cyclotomic Hecke algebra of type $G(r,1,n)$. The first author of this paper and Lei Shi have shown in \cite{HS} that the center
$Z(\HH_{r,n})$ of the cyclotomic Hecke algebra of type $G(r,1,n)$ is equal to the number of multipartitions (with $r$-components) of $n$, which is independent of the characteristic of the ground field $K$, Hecke parameters
as well as cyclotomic parameters. The ideal of the proof is first to use seminormal bases to give the dimension of the center in the semisimple case, which yields a lower bound of the dimension of the center in the
non-semisimple case, and then to construct a $K$-linear spanning set with cardinality equal to this lower bound for the cocenter of the cyclotomic Hecke algebra $\HH_{r,n}$. As a result, this spanning set must be a basis
and the stability of the dimension of the center is proved. It is natural to ask if one can use this result to study the center for the cyclotomic Hecke algebra $\HH_{r,p,n}$. In this paper we give some preparation and
attempt and outline an approach towards this direction. It turns out that to study the center for the cyclotomic Hecke algebra $\HH_{r,p,n}$ of type $G(r,p,n)$, one has to characterize the $\sigma$-twisted $k$-center for
the cyclotomic Hecke algebra $\HH_{r,n}$ for each $k\in\Z/p\Z$. In this paper we shall give an explicit basis for the center $Z(\HH_{r,p,n})$ of $\HH_{r,p,n}$ and an explicit basis for the $\sigma$-twisted $k$-center
$Z(\HH_{r,n})^{(k)}$ of $\HH_{r,n}$ for each $k\in\Z/p\Z$ in the semisimple case.

Let $\P_{r,n}$ be the set of multipartitions (with $r$-components) of $n$. For each $\blam\in\P_{r,n}$, let $\Std(\blam)$ be the set of standard $\blam$-tableaux. For each $\t\in\Std(\blam)$ and $k\in\Z/p\Z$, we define
$r_{\t,k}:=\gamma_\t/\gamma_{\m_k(\t)}$, where $\gamma_\t$ and $\m_k(\t)$ are defined in Lemma \ref{gammacoeff} and Definition \ref{midkt} respectively. The following theorem is the first main result of this paper, which
gives a symmetric property for the $r_{\t,k}$ constants of the semisimple cyclotomic Hecke algebra $\HH_{r,n}$.

\begin{thm}\label{mainthm1}
	Suppose $\HH_{r,n}$ is semisimple. Let $\blam\in\P_{r,n}$ and $\t\in\Std(\blam)$. Let $0\leq k,l\leq p_{\blam}$. Then we have
	$$r_{\t,lo_{\blam}}r_{\t\sft{lo_{\blam}},ko_{\blam}}=r_{\t,ko_{\blam}}r_{\t\sft{ko_{\blam}},lo_{\blam}},$$
where $o_\blam, p_\blam$ are defined in (\ref{olamb}), $\t\<j\>$ is defined in (\ref{langlerangle}).
\end{thm}

The next two theorems is the second main result of this paper, which gives explicit square roots of $\frac{\gamma_{\t^{\blam}\<o_\blam\>}}{\gamma_{\t^{\blam}}}$ and $\frac{\gamma_{\t\<lo_\blam\>}}{\gamma_\t}$ in $K$, see Proposition \ref{squareProp}, (\ref{odddefinition1}), (\ref{odddefinition1}), Definition \ref{sqrtt} and Lemma \ref{squareht}.

\begin{thm}\label{mainthm2a}
Suppose $\HH_{r,n}$ is semisimple. Let $\blam\in\P_{r,n}$. Then there exists an explicitly defined invertible element $h_{\blam}\in K^{\times}$ such that
	$$\frac{\gamma_{\t^{\blam}\<o_\blam\>}}{\gamma_{\t^{\blam}}}=\bigl(h_{\blam}\bigr)^{2}.$$
\end{thm}

\begin{thm}\label{mainthm2b}
Suppose $\HH_{r,n}$ is semisimple. Let $\blam\in\P_{r,n}$, $\t\in\Std(\blam)$ and $l\in\Z/p_{\blam}\Z$. Then there exists an explicitly defined invertible element $h_\t^{\<l\>}\in K^{\times}$ such that
	$$\frac{\gamma_{\t\<lo_\blam\>}}{\gamma_{\t}}=\bigl(h_\t^{\<l\>}\bigr)^{2}.$$
\end{thm}

For any $\s,\t\in\Std(\blam)$, let $f_{\s\t}$ be the corresponding seminormal basis element of the semisimple cyclotomic Hecke algebra $\HH_{r,n}$, see \cite{Ma04}. The following theorem is the third main result of this
paper, which gives a seminormal basis for the cyclotomic Hecke algebra of type $G(r,p,n)$.

\begin{thm}\label{mainthm3} Suppose $\HH_{r,n}$ is semisimple. For each $\blam\in\P_{r,n}$, $0\leq k<p_{\blam}$, and $\s, \t\in\Std(\blam)$, we define $$
f_{\s\t}^{[k]}:=\sum_{i\in\Z/p_{\blam}\Z}\sum_{j\in\Z/p\Z}(\varepsilon^{o_{\blam}})^{ki}A_{i,j}^{\s\t}f_{\s\sft{io_{\blam}+j}\t\sft{j}}, $$
where $A_{i,j}^{\s\t}$ is defined in Definition \ref{fstk}. Then $f_{\s\t}^{[k]}\in\HH_{r,p,n}$ and the set $$
\bigl\{f_{\s\t}^{[k]}\bigm|0\leq k<p_{\blam},\s, \t\in\Std(\blam),\blam\in\P_{r,n},\s^{-1}(1)\in\lam^{(j)}, \t^{-1}(1)\in\lam^{(m)}, 1\leq j,m\leq do_\blam\bigr\}
$$
gives a $K$-basis of $\HH_{r,p,n}$. Moreover, for any $\bmu\in\P_{r,n}$ and $\u, \v\in\Std(\bmu)$ such that $\u^{-1}(1)\in\bmu^{(j)}, \v^{-1}(1)\in\bmu^{(l)}, 1\leq j,l\leq do_\blam$, we have
	$$f_{\s\t}^{[k]}f_{\u\v}^{[l]}=\delta_{\t, \u}\delta_{k,l}p_{\blam}\gamma_{\t}f_{\s\v}^{[k]},$$
where $\gamma_\t$ is defined in Lemma \ref{gammacoeff}.
\end{thm}

Let $\P_{r,n}^{\sigma}$ be the set of equivalence classes as defined in the paragraph below (\ref{olamb}). The following theorem is the fourth main result of this paper. .

\begin{thm}\label{mainthm4} Suppose $\HH_{r,n}$ is semisimple. For each $\blam\in \P_{r,n}^{\sigma}$, $\t\in\Std(\blam)$ and $0\leq k<p_\blam$, we define
$$F_{\t}^{[k]}:=f_{\t\t}^{[k]}/(p_{\blam}\gamma_{\t}). $$
Set $F_{\blam}^{[k]}:=\sum\limits_{\substack{\t\in\Std(\blam)\\ 1\in(\t^{[1]},\cdots,\t^{[o_{\blam}]})}}F_{\t}^{[k]}$. Then
$\{F_{\blam}^{[k]}|\blam\in \P_{r,n}^{\sigma}, 0\leq k<p_{\blam}\}$ is a complete set of central primitive idempotents of $\HH_{r,p,n}$. In particular, it gives a $K$-basis of $Z(\HH_{r,p,n})$.
\end{thm}

For any $0\leq k<p$, we define $$
Z(\HH_{r,n})^{(k)}=\bigl\{z\in\HH_{r,n}\bigm|zT_{0}=\varepsilon^{k}T_{0}z, zT_{i}=T_{i}z, \forall\, 1\leq i<n\bigr\},
$$
and call it the $\sigma$-twisted $k$-center of $\HH_{r,n}$. The following theorem is the fifth main result of this paper.

\begin{thm}\label{mainthm5} Suppose $\HH_{r,n}$ is semisimple. Then for any $0\leq k<p$, the $\sigma$-twisted $k$-center $Z(\HH_{r,n})^{(k)}$
has a basis
	$$\biggl\{ F_{\blam,k}:=\sum_{\t\in\Std(\blam)}\gamma_{\m_{k}(\t)}^{-1}f_{\t\sft{k}\t}\biggm| \blam\in\P_{r,n}, \blam\<k\>=\blam\biggr\} .$$
In particular, $F_{\blam,0}$ coincides with the central primitive idempotent $F_{\blam}$ of $\HH_{r,n}$.
\end{thm}

The content of the paper is organised as follows. In Section 2 we recall the definition of cyclotomic Hecke algebras of type $G(r,p,n)$, fix some notations and introduce some basic combinatorial notions which will be used
in later sections. In Section 3 we first recall the definitions of seminormal bases and the $\gamma$-coefficients for $\HH_{r,n}$ which will be the main objects studied in this paper. Then we present a number of symmetric
and rational properties of these $\gamma$-coefficients. In particular, we give the proof of the first and the second main results Theorems \ref{mainthm1}, \ref{mainthm2a}, \ref{mainthm2b} in this section. Using the results
obtained in Section 3 we give the construction of seminormal bases for the Hecke algebra $\HH_{r,p,n}$ in Section 4, which gives a {\it non-trivial} generalization of the seminormal bases for the Iwahori-Hecke algebra of
type $D_n$ in \cite[Section 5]{HWS}. In particular, we give the proof the third and the fourth main results Theorems \ref{mainthm3}, \ref{mainthm4} in this section. In Section 5 we initiate the study of the center of the
cyclotomic Hecke algebra $\HH_{r,p,n}$ in the general (non-semisimple) case. We give the proof of the fifth main result Theorem \ref{mainthm5} in this section, which gives an explicit basis for the $\sigma$-twisted
$k$-center $Z(\HH_{r,n})^{(k)}$ of $\HH_{r,n}$ for each $k\in\Z/p\Z$ in the semisimple case. Finally, we outline an approach in Proposition \ref{KeyProp} and Remark \ref{finalrem} on how to study the stability of the
dimension of the center of the cyclotomic Hecke algebra $\HH_{r,p,n}$ in the general (non-semisimple) case.

\bigskip\bigskip
\centerline{Acknowledgements}
\bigskip

The research was support by the National Natural Science Foundation of China (No. 12431002). Part of the main results of this paper is contained in the second author's PhD thesis \cite{W24}.
\bigskip

\section{Preliminaries}

In this section, we shall introduce some basic combinatorial notion and fix some notations which will be used in this paper.

Let $r,p,n$ be positive integers with $n\geq 2$. Assume that $r=pd$, where $d\in\Z^{\geq 1}$. Let $R$ be a commutative domain and $1\neq q\in R^\times$. Suppose that $R$ contains a primitive $p$-th root of unity
$\varepsilon\in R$. In particular, this implies that $p1_{R}\in R^\times$ is invertible. Let $\bQ=(Q_1,\cdots,Q_d)\in R^{d}$ and set
$$\bQ^{\vee \varepsilon}:=\varepsilon\bQ\vee\varepsilon^{2}\bQ\vee\cdots\vee\varepsilon^{p}\bQ=(\varepsilon Q_1,\cdots,\varepsilon Q_d, \cdots, \varepsilon^{p} Q_1,\cdots,\varepsilon^{p} Q_d).$$
The cyclotomic Hecke algebra $\HH_{r,n}(\bQ^{\vee \varepsilon}):=\HH_{r,n}(q, \bQ^{\vee \varepsilon})$ of type $G(r,1,n)$ with Hecke parameter $q$ and cyclotomic parameters $\bQ^{\vee \varepsilon}$ is the unital associative
$R$-algebra with generators $T_0,T_1,\cdots,T_{n-1}$ and the following defining relations:
$$\begin{aligned}
	&(T_{0}^p-Q_{1}^p)\cdots(T_{0}^p-{Q}_{d}^p)=0;\\
	&T_{0}T_{1}T_{0}T_{1}=T_{1}T_{0}T_{1}T_{0};\\
	&(T_{i}-q)(T_{i}+1)=0,\quad \forall\,1\leq i\leq n-1;\\
	&T_{i}T_{j}=T_{j}T_{i},\,\,\forall\,1\leq i<j-1<n-1,\\
	&T_{i}T_{i+1}T_{i}=T_{i+1}T_{i}T_{i+1},\,\,\forall\,1\leq i<n-1.
\end{aligned}$$
Henceforth, without causing ambiguity, we write $\HH_{r,n}:=\HH_{r,n}(\bQ^{\vee \varepsilon})$ for simplicity. The cyclotomic Hecke algebra $\HH_{r,p,n}:=\HH_{r,p,n}(\bQ^{\vee \varepsilon})$ of type $G(r,p,n)$ is the
subalgebra of  $\HH_{r,n}(\bQ^{\vee \varepsilon})$ generated by $T_{0}^{p}, T_{u}:=T_{0}^{-1}T_{1}T_{0}, T_{1}, T_{2}, \cdots, T_{n-1}$.

We set $$
L_{1}=T_{0},\,\, L_{k+1}=q^{-1}T_{k}L_{k}T_{k},\,\forall\,1\leq j\leq n-1 .
$$
These elements commutes with each other, and are called the Jucys-Murphy operators of $\HH_{r,n}(q;{\rm Q}_{1},\cdots,{\rm Q}_{r})$.

We recall some $R$-algebra automorphisms and anti-automorphisms $\sigma$ and $\tau$ of the cyclotomic Hecke algebra $\HH_{r,n}$. Let $\ast$ be the unique $R$-algebra anti-automorphism of $\HH_{r,n}$ which fixes all of its
Hecke generators $T_0,T_1,\cdots,T_{n-1}$. Let $\sigma$ be the unique $R$-algebra automorphism of $\HH_{r,n}$ which is defined on generators by
$$\sigma(T_{0})=\varepsilon T_{0}, \sigma(T_{i})=T_{i}, \forall\,1\leq i < n .$$
Let $\tau$ be the unique $R$-algebra automorphism of automorphism which is defined on generators by
$$\tau(h)=T_{0}^{-1}hT_{0}, \forall\, h\in \HH_{r,n}.$$
It is straightforward to check that $\HH_{r,p,n}=\{h\in \HH_{r,n} | \sigma(h)=h\}$, i.e., $\HH_{r,p,n}$ is the set of $\sigma$-fixed points in $\HH_{r,n}$. The restriction of $\tau$ to the subalgebra $\HH_{r,p,n}$ is also a
$K$-algebra automorphism of $\HH_{r,p,n}$.

We need some combinatorial notions and notations. Let $m$ be a positive integer. The symmetric group $\Sym_m$ on $\{1,2,\cdots,m\}$ is generated by $\{s_i:=(i,i+1)|1\leq i<m\}$. A word $w=s_{i_{1}}s_{i_{2}}\ldots
s_{i_{k}}$, where $1\leq i_1,\cdots,i_k\leq m-1$, is called a reduced expression of $w$ if $k$ is minimal; in this case we say that $w$ has length $k$ and we write $\ell(w)=k$. A composition of $m$ is a sequence
$\lam=(\lam_{1},\lam_{2},\cdots)$ of non-negative integers such that $|\lam|:=\Sigma_{i\geq 1}\lam_{i}=m$.
Given a composition $\lam=(\lam_{1},\lam_{2},\cdots,\lam_k)$ of $m$, the corresponding Young subgroup of $\Sym_m$ is defined to be $$
\Sym_\lam:=\Sym_{(1,2,\cdots,\lam_1)}\times\Sym_{(\lam_1+1,\lam_1+2,\cdots,\lam_1+\lam_2)}\times\cdots\times\Sym_{(\sum_{j=1}^{k-1}\lam_j+1,\cdots,m)} .
$$
We use $\mathcal{D}_\lam$ to denote the set of minimal length right coset representatives of $\Sym_\lam$ in $\Sym_m$.

A composition $\lam=(\lam_{1},\lam_{2},\cdots)$ is called a partition if $\lam_{1}\geq\lam_{2}\geq\cdots$. We write $\lam\vdash m$ if $\lam$ is a partition of $m$. Let $\P_{a}$ be the set of partitions of $a$. For
$\lam=(\lam_{1},\lam_{2},\ldots)\vdash a$, the conjugate partition of $\blam$ is the partition $\lam'=(\lam'_{1},\lam'_{2},\ldots)$, where $\lam'_{i}:=\#\{\lam_{j}|\lam_{j} \geq i\}$, $\forall\,i\geq 1$.

A multipartition (with $r$ components) of $n$ is a sequence $\blam=(\lam^{(1)},\cdots,\lam^{(r)})$ of partitions such that $|\lam^{(1)}|+\cdots+|\lam^{(r)}|=n$. We write $\blam\vdash n$. Let $\P_{r,n}$ be the set of
multipartitions of $n$.

The Young diagram of a partition $\lam$ is the set $[\lam]:=\{(i,j)\mid 1\leq j\leq \lam_{i}\}$. The Young diagram of a multipartition $\blam\in\P_{r,n}$ is the set
$$[\blam]:=\bigl\{(i,j,c)\bigm|1\leq j\leq \lam_{i}^{(c)}, 1\leq c\leq r\bigr\}.$$
The Young diagram $[\blam]$ of a multipartition can be viewed as a sequence of Young diagrams $[\lam^{(c)}], 1\leq c\leq r$ of partitions.

A $\blam$-tableau is a bijection $\t: [\blam]\rightarrow\{1,2,\cdots,n\}$. If $\t$ is a $\blam$-tableau write $\Shape(\t):=\blam$. A $\blam$-tableau $\t$ is standard, if $\t(i,j,l)\leq\t(a,b,l)$ whenever $i\leq a$, $j\leq
b$ and $1\leq l\leq r$. Let $\Std(\blam)$ be the set of standard $\blam$-tableaux. A standard tableau $\t\in\Std(\blam)$ can be viewed as a sequence $\t=(\t^{(1)},\cdots,\t^{(r)})$ of tableaux, where each $\t^{(i)}$ is
standard $\lam^{(i)}$-tableau.

Given $\blam\in\P_{r,n}$, the conjugate partition of $\blam$ is the multipartition
$$\blam':=\bigl(\lam^{(r)'},\cdots,\lam^{(1)'}\bigr).$$
For each $\t\in\Std(\blam)$, we define $\t'=\bigl(\t^{(r)'},\cdots,\t^{(1)'}\bigr)\in \Std(\blam')$.

For any multipartition $\blam\in\P_{r,n}$, Let $\t^{\blam}$ be the standard $\blam$-tableau with the numbers $1,2,\cdots,n$ entered in order first along the row of $\lam^{(1)}$ and then the row of $\lam^{(2)}$ and so on.
Set $\t_{\blam}:=(\t^{\blam'})'$, which is again a standard $\blam$-tableau. For any $\t\in\Std(\blam)$, let $d(\t)\in\Sym_{n}$ be the unique element such that $\t^\blam d(\t)=\t$. We set $w_{\blam}:=d(\t_{\blam})$.

For any $\blam,\bmu\in\P_{r,n}$, we write $\blam\unrhd\bmu$ if $\forall\,1\leq s\leq r$ and $\forall\,i\geq 1$ we have
$$\sum_{t=1}^{s-1}|\lam^{(t)}|+\sum_{j=1}^{i}\lam^{(s)}_{j}\geq \sum_{t=1}^{s-1}|\mu^{(t)}|+\sum_{j=1}^{i}\mu^{(s)}_{j}.$$
We write $\blam\rhd\bmu$ if $\blam\unrhd\bmu$ and $\blam\neq\bmu$. Then $\P_{r,n}$ become a poset under the dominance order ``$\unrhd$''.

The dominance order ``$\unrhd$'' on $\P_{r,n}$ can be extended to the set of standard tableaux as follows. For any $\s\in\Std(\blam), \t\in\Std(\bmu)$, we write $\s\unrhd\t$ if $\forall\,1\leq k\leq n$ we have
${\rm{Shape}}(\s\!\downarrow_{\{1,2,\cdots,k\}})\unrhd{\rm{Shape}}(\t\!\downarrow_{\{1,2,\cdots,k\}})$. We write $\s\rhd\t$ if $\s\unrhd\t$ and $\s\neq\t$. For any standard tableau $\s\in\Std(\blam)$, it is clear that
$\t^\blam\unrhd\s\unrhd\t_\blam$.

Given multipartition $\blam=(\lam^{(1)},\cdots,\lam^{(r)})\in\P_{r,n}$, let $\Sym_{\blam}$ be the corresponding standard Young subgroup of $\Sym_{n}$ which is defined as follows: $$\begin{aligned}
	\Sym_\blam:&=\Sym_{\{1,\cdots,\lam_{1}^{(1)}\}}\times\Sym_{\{\lam_{1}^{(1)}+1,\cdots,\lam_{2}^{(1)}\}}\times\cdots\times\Sym_{\{|\lam^{(1)}|-\lam^{(1)}_{b_{1}}+1,\cdots,|\lam^{(1)}|\}}\times\cdots\\
	&\qquad\times\Sym_{\{n-|\lam^{(r)}|+1,\cdots,n-|\lam^{(r)}|+\lam_{1}^{(r)}\}}\times\cdots\times \Sym_{\{n-|\lam^{(r)}_{b_r}|+1,\cdots,n\}},
\end{aligned}$$
where $b_i:=(\lam^{(i)'})_1$, $i=1,2,\cdots,r$. For any non-negative integers $a,b$, we consider the permutation $w_{a,b}\in\Sym_{a+b}$, which is defined in two-lines notation as follows: $$\begin{aligned}
w_{(a,b)} &= \begin{pmatrix}
			1 & \cdots & a & a+1 & \cdots & a+b \\
b+1 & \cdots & a+b & 1 & \cdots & b
\end{pmatrix}\\
&=\underbrace{s_a\cdots s_2 s_1}\underbrace{s_{a+1}\cdots s_3 s_2}\cdots \underbrace{s_{a+b-1}\cdots s_{a+1}s_a} ,
\end{aligned},
$$
and for any non-negative integer $k$, we set \begin{equation}\label{waaka}
\begin{aligned}
w_{a,b}^{\sft{k}} &= \begin{pmatrix}
			k+1 & \cdots & k+a & k+a+1 & \cdots & k+a+b \\
k+b+1 & \cdots & k+b+a & k+1 & \cdots & k+b
\end{pmatrix}\\
&=\underbrace{s_{k+a}\cdots s_{k+2}s_{k+1}}\underbrace{s_{k+a+1}\cdots s_{k+3}s_{k+2}}\cdots \underbrace{s_{k+a+b-1}\cdots s_{k+a+1}s_{k+a}} \end{aligned}. \end{equation}

Let $\blam\in\P_{r,n}$ be a multipartition. Recall that $r=pd$. We can rewrite $\blam=(\blam^{[1]}, \cdots, \blam^{[p]})$, where $$\blam^{[t]}=(\blam^{(dt-d+1)}, \blam^{(dt-d+2)}, \cdots, \blam^{(dt)}), \forall\, 1\leq
t\leq p.$$
Since $\blam^{[t+kp]}=\blam^{[t]}$, $\forall\, k\in \Z$, we can assume without loss of generality that the superscript $t\in\Z/p\Z$. Similarly, we write $\s=(s^{[1]}, \cdots, s^{[p]})$ for standard tableau
$\s\in\Std(\blam)$, where $\s^{[t]}=(\s^{(dt-d+1)}, \s^{(dt-d+2)}, \cdots, \s^{(dt)})$. For any $\blam\in\P_{r,n}$, define the integers
\begin{equation}\label{olamb} o_{\blam}=\minum\{k\geq 1 \, | \, \blam^{[k+t]}=\blam^{[t]}, \forall\, t\in\Z\},\,\, p_{\blam}:=p/o_{\blam}.\end{equation}
Then $o_\blam$ divides $p$ and hence $p_{\blam}\in\Z^{\geq 1}$.

There is an equivalence relation $\sim_{\sigma}$ on the set $\P_{r,n}$ defined by  $\blam\sim_{\sigma}\bmu$ if there exists an integer $k\in\Z$ such that $\blam^{[t]}=\bmu^{[t+k]}$, $\forall\,t\in\Z/p\Z$. Let
$\P_{r,n}^{\sigma}$ be the set of $\sim_{\sigma}$-equivalence classes in $\P_{r,n}$.

For an arbitrary sequence $\bm{a}=(a_{1}, a_{2}, \cdots, a_{m})$, define
$$\bm{a}\sft{k}=(a_{k+1}, a_{k+2}, \cdots, a_{k+m}),$$
where $a_{i+jm}:=a_{i}, \forall j\in \Z$. In particular, for  $\blam=(\blam^{[1]}, \cdots, \blam^{[p]})\in\P_{r,n}$ and $\s\in\Std(\blam)$, define
\begin{equation}\label{langlerangle}\blam\sft{z}=(\blam^{[z+1]}, \cdots, \blam^{[z+p]})\in\P_{r,n}, \quad \s\sft{z}=(\s^{[z+1]}, \cdots, \s^{[z+p]})\in\Std(\blam\langle z\rangle), \forall\, z\in\Z.\end{equation}

If $1\leq k\leq n$ and $\s\in\Std(\blam)$, define the residue of $k$ in $\s$ to be
$$\res_{\s}(k)=\varepsilon^{t}q^{b-a}Q_{c},$$
if $k$ appears in row $a$ and column $b$ of $\s^{(c+(t-1)d)}$, where $1\leq c\leq d$, $1\leq t\leq p$.

\bigskip


\section{Some rational properties of the $\gamma$-coefficients of $\HH_{r,n}$}

In this section we shall first recall the seminormal bases for the semisimple cyclotomic Hecke algebra $\HH_{r,n}$ and the definitions of the $\gamma$-coefficients for $\HH_{r,n}$. Then we shall present a number of rational
properties
and symmetric properties of those $\gamma$-coefficients. We shall give the proof of our first and second main results Theorems \ref{mainthm1}, \ref{mainthm2a} and \ref{mainthm2b} in this paper.

Let $R=K$ be a field, $1\neq q\in K^\times$, $\bQ=(Q_1,\cdots,Q_d)\in K^{d}$. By \cite{A1}, $\HH_{r,n}(\bQ^{\vee \varepsilon})$ is split semi-simple over $K$  if and only if
\begin{equation}\label{ssrpn}
	\prod_{i=1}^n(1+q+q^2+\cdots+q^{i-1})\prod_{-n<k<n}\Biggl(\biggl(\prod_{1\leq i<j\leq d}\prod_{0\leq t<p}(Q_{i}-\varepsilon^{t}q^{k}Q_{j})\biggr)
\biggl(\prod_{1\leq i\leq d}\prod_{1\leq t<p}Q_i(1-\varepsilon^{t}q^{k})\biggr)\Biggr)\neq 0.
\end{equation}
In this case, $\HH_{r,p,n}$ is also split semisimple over $K$. Conversely, if $\HH_{r,p,n}$ is also split semisimple over $K$, then
(\ref{ssrpn}) holds by \cite[Theorem 5.9]{Hu04} and \cite[Lemma 2.1]{Hu08}.
Henceforth, we shall always assume that (\ref{ssrpn}) holds. That says, both $\HH_{r,n}$ and $\HH_{r,p,n}$ is split semisimple over $K$.

In \cite{Ma04}, Mathas constructed a $K$-basis $\{f_{\s\t}\ |\ \s,\t \in \Std(\blam), \blam\in\P_{r,n}\}$ of $\HH_{r,n}$. The action of the generators $T_0, T_{1}, T_{2}, \cdots, T_{n-1}$ of $\HH_{r,n}$ can be described as
follows.

\begin{prop}[{\cite[Proposition 2.7]{Ma04}}] \label{tiact}
	Suppose that $\blam\in\P_{r,n}$, $\s,\u\in\Std{\blam}$. Let $\t=\s(i,i+1)$, where $i$ is an integer with $1\leq i<n$. Then $f_{\u\s}T_0=\res_\s(1)f_{\u\s}$. If $\t$ is standard then
	$$ f_{\u\s}T_{i}=A_{i}(\s)f_{\u\s}+B_{i}(\s)f_{\u\t},\,\,\text{if either $\s\rhd \t$ or $s\lhd\t$.} $$
	If $\t$ is not standard then
	$$f_{\u\s}T_{i}=\begin{cases}
		qf_{\u\s}, &\text{if $i$ and $i+1$ are in the same row of $\s$},\\
		-f_{\u\s}, &\text{if $i$ and $i+1$ are in the same column of $\s$}.
	\end{cases}$$
	where \begin{equation}\label{Ais}
	 A_{i}(\s):=\frac{(q-1)\res_{\s}(i+1)}{\res_{\s}(i+1)-\res_{\s}(i)}=\frac{(q-1)\res_{\t}(i)}{\res_{\t}(i)-\res_{\s}(i)},\end{equation}
	and
	$$B_{i}(\s):=\begin{cases}
		1, &\text{if $\s\rhd \t$},\\
		\frac{(q\res_{\s}(i)-\res_{\t}(i))(\res_{\s}(i)-q\res_{\t}(i))}{(\res_{\t}(i)-\res_{\s}(i))^{2}}, &\text{if $\s\lhd \t$.}
	\end{cases}$$
\end{prop}
{We call $\{f_{\u\s}|\u,\s\in\Std(\blam),\blam\in\P_{r,n}\}$ a \it{seminormal basis of the Hecke algebra} $\HH_{r,n}$.}

\begin{lem}\text{(\cite[2.9, 2.11]{Ma04})}\label{GammaCoeffi} For any $\blam\in\P_{r,n}$, $\s\in\Std{\blam}$, there is an invertible scalar $\gamma_\s\in K^\times$ such that for any $\t\in\Std{\blam}$, and
$\u,\v\in\Std(\bmu), \bmu\in\P_{r,n}$, we have that $$
f_{\u\v}f_{\s\t}=\delta_{\v\s}\gamma_{\s}f_{\u\t} .
$$
\end{lem}

\begin{lem}\text{(\cite[2.6, 2.9]{Ma04})}\label{gammacoeff} For any $\blam\in\P_{r,n}$, $\s,\u\in\Std{\blam}$ and $1\leq k\leq n$, we have that $$
f_{\u\s}^*=f_{\s\u},\quad f_{\u\s}L_k=\res_{\s}(k)f_{\u\s},\quad L_kf_{\u\s}=\res_{\u}(k)f_{\u\s} .
$$
Moreover, $$
\gamma_{\t^\blam}=[\blam]_q^{!}\prod_{\substack{1\leq i\leq j\leq d\\ 1\leq s,t\leq p}}\prod_{\substack{(a,b)\in\lam^{(s)}\\ \text{$i<j$ or}\\
\text{$i=j$ and $s<t$}}}(q^{b-a}\veps^s Q_i-\veps^tQ_j),
$$
and if $\s=\t(i,i+1)\rhd\t$, then $\gamma_{\t}=B_i(\t)\gamma_\s$, where $$
[m]_q^!:=[m]_q[m-1]_q\cdots[1]_q,\,\, [m]_q:=(q^m-1)/(q-1),\,\forall\,\,m\in\Z^{\geq 0} .
$$
\end{lem}

\begin{dfn} {For each $\blam\in\P_{r,n}$, let $S(\blam):=\sum_{\t\in\Std(\blam)}Kf_{\t}$ be a $K$-dimensional linear space with a $K$-basis $\{f_\t|\t\in\Std(\blam)\}$. Then we can endow $S(\blam)$ an irreducible right
$\HH_{r,n}$-module structure by the same formulae given in Proposition \ref{tiact} with $f_{\u\s}, f_{\u\t}$ replaced with $f_\s, f_\t$ respectively. We call this a seminormal construction of the irreducible module
$S(\blam)$, and call the corresponding $K$-basis $\{f_\t|\t\in\Std(\blam)\}$ \it{a seminormal basis of the irreducible module} $S(\blam)$.}
\end{dfn}

\begin{thm}\text{(\cite{A2},\cite[Corollary 2.5]{Hu08})}\label{Classification}
	Suppose that (\ref{ssrpn}) holds. Let $\blam\in\P_{r,n}$ be a multipartition. Then there is a right $\HH_{r,p,n}$-module decomposition:
	$$S(\blam)\downarrow_{\HH_{r,p,n}}\cong S^{\blam}_{1}\oplus S^{\blam}_{2}\oplus\cdots\oplus S^{\blam}_{p_{\blam}},$$
where each $S^{\blam}_{j}$ is an irreducible right $\HH_{r,p,n}$-module satisfying $$
\bigl(S^{\blam}_{j}\bigr)^\tau\cong S^{\blam}_{j-1},\,\,\forall\,j\in\Z/p_\blam\Z .
$$
Moreover, $\{S^{\blam}_{i} | \blam\in\P_{r,n}^{\sigma}, 1\leq i\leq p_{\blam}\}$ give a complete set of pairwise non-isomorphic simple $\HH_{r,p,n}$-modules.
\end{thm}

\begin{lem}\text{(\cite[2.5]{Ma04})}\label{dist} Suppose that (\ref{ssrpn}) holds. Then for any two standard tableaux $\s,\t$, $\s=\t$ if and only if $\res_\s(m)=\res_\t(m)$ for any $1\leq m\leq n$.
\end{lem}

In other words, Lemma \ref{dist} implies that the eigenvalues of each seminormal basis with respect to the Jucys-Murphy operators are pairwise distinct. Set $R(k):=\{\res_\s(k)|\s\in\Std(\mu),\bmu\in\P_{r,n}\}$.

\begin{dfn}\text{(\cite[2.4]{Ma04})} Suppose that (\ref{ssrpn}) holds. For each $\blam\in\P_{r,n}$ and $\t\in\Std(\blam)$, we define $$
F_{\t}=\prod\limits^n\limits_{k=1}\prod\limits_{\substack{c\in R(k)\\c\neq \res_{\t}(k)}}\frac{L_{k}-c}{\res_{\t}(k)-c} .
$$
\end{dfn}

\begin{lem}\text{(\cite[2.15]{Ma04})}\label{Ft} Suppose that (\ref{ssrpn}) holds. For each $\blam\in\P_{r,n}$ and $\t\in\Std(\blam)$, we have that $$
F_\t=f_{\t\t}/\gamma_{\t},
$$
which is a primitive idempotent.
\end{lem}

Recall that $\sigma$ is an $K$-algebra automorphism of $\HH_{r,n}$ of order $p$.

\begin{lem}\label{sigmaFt}
	Let $\blam\in\P_{r,n}$ and $\t\in\Std(\blam)$. Then
	$$\sigma(F_{\t})=F_{\t\sft{1}}.$$
	In particular, $\sigma(F_{\t}+F_{\t\sft{1}}+\cdots+F_{\t\sft{p-1}})=F_{\t}+F_{\t\sft{1}}+\cdots+F_{\t\sft{p-1}}$.
\end{lem}

\begin{proof} By definition, if the integer $k$ lies in row $a$ and column $b$ of the component $\t^{(c+(s-1)d)}$, then $\res_{\t}(k)=\varepsilon^{s}q^{b-a}Q_{c}$, and hence
	$$
\varepsilon^{i}\res_{\t}(k)=\varepsilon^{i+s}q^{b-a}Q_{c}=\res_{\t\sft{-i}}(k), \forall\, i\in \Z, $$
where $1\leq k\leq n$. In particular, $\res_{\t\sft{1}}(k)=\varepsilon^{-1}\res_{\t}(k)$.
	
	Applying the automorphism $\sigma$, we can get that
	$$\begin{aligned}
		\sigma(F_{\t})&=\prod\limits^n\limits_{k=1}\prod\limits_{\substack{c\in R(k)\\c\neq \res_{\t}(k)}}\frac{\varepsilon L_{k}-c}{\res_{\t}(k)-c}
		=\prod\limits^n\limits_{k=1}\prod\limits_{\substack{c\in R(k)\\c\neq \res_{\t}(k)}}\frac{L_{k}-\varepsilon^{-1}c}{\varepsilon^{-1}\res_{\t}(k)-\varepsilon^{-1}c}\\
		&=\prod\limits^n\limits_{k=1}\prod\limits_{\substack{c\in R(k)\\c\neq \res_{\t\sft{1}}(k)}}\frac{L_{k}-c}{\res_{\t\sft{1}}(k)-c}\\
		&=F_{\t\sft{1}}.
	\end{aligned}$$
This proves the lemma.
\end{proof}

In \cite[\S4]{Ma04}, Mathas have constructed certain invertible elements $\Phi_{\t}, \Psi_{\t}\in\HH_q(\Sym_n)$ which can intertwine the primitive idempotents $F_{\t^{\blam}}$ and $F_{\t}$.

\begin{lem}[{\cite[Proposition 4.1]{Ma04}}]\label{recursiveA}
	Suppose that $\blam\in\P_{r,n}$, $\t\in\Std(\blam)$. There exists invertible element $\Phi_{\t}\in\HH_q(\Sym_n)$ such that
	\begin{enumerate}
		\item $f_{\s\t}=\Phi_{\s}^{*}f_{\t^{\blam}\t^{\blam}}\Phi_{\t}$, for any standard tableaux $\s,\t\in\Std(\blam)$;
		\item $\Phi_{\t^{\blam}}=1$; if $\s=\t(i,i+1)\lhd\t$ then $\Phi_{\s}=\Phi_{\t}(T_{i}-A_{i}(\t))$.
	\end{enumerate}
\end{lem}	
Note that the construction of these elements depends on the choice of reduced expression for $d(\t)$.

Let $\blam\in\P_{r,n}$ and $\s, \t\in\Std(\blam)$. By (\ref{Ft}), we can get that
\begin{equation}\label{sigmaaction}
	\sigma(f_{\s\t})=\sigma(\Phi_{\s}^{*}f_{\t^{\blam}\t^{\blam}}\Phi_{\t})=\gamma_{\t^{\blam}}\Phi_{\s}^{*}\sigma(F_{\t^{\blam}})\Phi_{\t}=\gamma_{\t^{\blam}}\Phi_{\s}^{*}F_{\t^{\blam}\sft{1}}\Phi_{\t}
=\frac{\gamma_{\t^{\blam}}}{\gamma_{\t^{\blam}\sft{1}}}\Phi_{\s}^{*}f_{\t^{\blam}\sft{1}\t^{\blam}\sft{1}}\Phi_{\t}.
\end{equation}

\begin{dfn}\label{midkt} Let $\blam\in\P_{r,n}$ and $\t\in\Std(\blam)$. For each $1\leq k\leq p$, we set $$
a_{k}=\sum_{i=1}^{k}|\blam^{[i]}|=\sum_{j=1}^{kd}|\lam^{(j)}|, $$ and we can uniquely decompose
$d(\t)=x_{k}d_{k}$, where $x_{k}\in\Sym_{(a_{k},n-a_{k})}$, $d_{k}\in\mathcal{D}_{(a_{k},n-a_{k})}$, $\mathcal{D}_{(a_{k},n-a_{k})}$ is the set of minimal length right coset representatives of $\Sym_{(a_{k},n-a_{k})}$ in
$\Sym_{n}$. We define $$
\m_{k}(\t):=\t^{\blam}x_{k}.
$$ 	
\end{dfn}

\begin{rem}\label{rmkmid}
	For convenience, for any standard tableau $\t$, we write $\m_{0}(\t):=\t$. Note that $\m_k(\m_k\t)=\m_k\t$ and $\m_{0}(\t)=\m_p(\t)=\t$. By convention, we define $\m_{pa+b}(\t):=\m_b(\t)$ for any $a\in\Z, 0\leq
b<p$.
Thus we can assume without loss of generality that the subscript $k\in\Z/p\Z$ in $\m_k(\t)$.
\end{rem}

\begin{lem}\label{gtsft}
	Suppose that $\blam\in\P_{r,n}$, $\t\in\Std(\blam)$ and $k\in\Z/p\Z$. Then $\m_{k}(\t)\trianglerighteq\t$ and
\begin{equation}\label{eqa10}
\frac{\gamma_{\t}}{\gamma_{\t\sft{k}}}=\frac{\gamma_{\t^{\blam}}}{\gamma_{\t^{\blam}\sft{k}}}\cdot\bigg(\frac{\gamma_\t}{\gamma_{\m_{k}(\t)}}\bigg)^{2}=\frac{\gamma_{\t^{\blam}\sft{k}}}{\gamma_{\t^{\blam}}}
\cdot\bigg(\frac{\gamma_{\m_{k}(\t)}}{\gamma_{\t\sft{k}}}\bigg)^{2}.\end{equation}
	In particular, we have 	\begin{equation}\label{eqa00} \frac{\gamma_{\m_k(\t)}}{\gamma_{\t^\blam}}=\frac{\gamma_{\m_k(\t)\<k\>}}{\gamma_{\t^{\blam}\<k\>}}.\end{equation}
\end{lem}

\begin{proof}
	 By Definition \ref{midkt}, it is clear that $\m_{k}(\t)\trianglerighteq\t$. We keep the same notations as in Definition \ref{midkt}. Fixing reduced expressions $x_{k}=s_{i_{1}}s_{i_{2}}\cdots s_{i_{l}}$ and
$d_{k}=s_{i_{l+1}}s_{i_{l+2}}\cdots s_{i_{N}}$, then the expression
	$$d(\t)=s_{i_{1}}\cdots s_{i_{l}}s_{i_{l+1}}\cdots s_{i_{N}}$$
	is also reduced.
	
	For each $1\leq j\leq N$, set $w_{j}:=s_{i_{1}}\cdots s_{i_{j}}\in\Sym_{n}$ and $\t_{j}:=\t^{\blam}w_{j}\in\Std(\blam)$. By convention, we set $\t_{0}:=\t^{\blam}$. There is a sequence of standard $\blam$-tableaux
	\begin{equation}\label{sequenceT}\t^{\blam}=\t_{0}\triangleright \t_{1}\triangleright \t_{2}\triangleright\cdots\triangleright \t_{l}\triangleright \t_{l+1}\triangleright\cdots\triangleright\t_{N}=\t.\end{equation}
	By the definition of $\gamma_{\t}$, we can get that
	\begin{equation}\label{gammat1}
\gamma_{\t}=\gamma_{\t^{\blam}}\cdot\frac{\gamma_{\t_{1}}}{\gamma_{\t_{0}}}\cdot\frac{\gamma_{\t_{2}}}{\gamma_{\t_{1}}}\cdots\frac{\gamma_{\t_{N}}}{\gamma_{\t_{N-1}}}.\end{equation}
	
	Applying $\sft{k}$ to the above sequence (\ref{sequenceT}) and noting that $s_{i_j}\in\Sym_{(a_{k},n-a_{k})}, \forall\,1\leq j\leq l$, we can get a sequence of standard $\blam\sft{k}$-tableaux
	$$\t^{\blam}\sft{k}=\t_{0}\sft{k}\triangleright \t_{1}\sft{k}\triangleright \t_{2}\sft{k}\triangleright\cdots\triangleright \t_{l}\sft{k}.$$

By assumption, $d_{k}\in\mathcal{D}_{(a_{k},n-a_{k})}$. It follows that $i_{l+1}$ must lies in the $j$-th component $\t_{l}\sft{k}$ for some $(p-k)d<j\leq r=pd$, while $i_{l+1}+1$ must lies in the $s$-th component
$\t_{l}\sft{k}$ for some $1\leq s\leq (p-k)d$. Thus we can deduce that $\t_{l}\sft{k}\triangleleft \t_{l+1}\sft{k}$. More generally, for each $l+1\leq a\leq N$, $i_{a}$ must lies in the $j$-th component $\t_{a-1}\sft{k}$
for some
$(p-k)d<j\leq r=pd$, while $i_{a}+1$ must lies in the $s$-th component $\t_{a-1}\sft{k}$ for some $1\leq s\leq (p-k)d$. Thus we can deduce that $\t_{a-1}\sft{k}\triangleleft \t_{a}\sft{k}$. Thus we get that
\begin{equation}\label{downUp}
\t^{\blam}\sft{k}=\t_{0}\sft{k}\triangleright \t_{1}\sft{k}\triangleright \t_{2}\sft{k}\triangleright\cdots\triangleright \t_{l}\sft{k}\triangleleft
\t_{l+1}\sft{k}\triangleleft\cdots\triangleleft\t_{N}\sft{k}=\t\sft{k}.\end{equation}

Note that \begin{equation}\label{gammat2}
\gamma_{\t\sft{k}}=\gamma_{\t^{\blam}\sft{k}}\cdot\frac{\gamma_{\t_{1}\sft{k}}}{\gamma_{\t_{0}\sft{k}}}\cdots\frac{\gamma_{\t_{l}\sft{k}}}{\gamma_{\t_{l-1}\sft{k}}}\cdot
\bigg(\frac{\gamma_{\t_{l}\sft{k}}}{\gamma_{\t_{l+1}\sft{k}}}\cdots\frac{\gamma_{\t_{N-1}\sft{k}}}{\gamma_{\t_{N}\sft{k}}}\bigg)^{-1}.\end{equation}
	
Applying Lemma \ref{gammacoeff}, we see that for each $1\leq j\leq N$, 	$$
		\frac{\gamma_{\t_{j}}}{\gamma_{\t_{j-1}}}=B_{i_j}(\t_{j-1})=\frac{(q\res_{\t_{j-1}}(i_{j})-\res_{\t_{j}}(i_{j}))(\res_{\t_{j-1}}(i_{j})-q\res_{\t_{j}}(i_{j}))}{(\res_{\t_{j-1}}(i_{j})-\res_{\t_{j}}(i_{j}))^{2}}\\
$$		
Using (\ref{downUp}) and Lemma \ref{gammacoeff}, we can deduce that \begin{equation}\label{eqa11} \frac{\gamma_{\t_{j}\sft{k}}}{\gamma_{\t_{j-1}\sft{k}}}=\begin{cases}
			\frac{\gamma_{\t_{j}}}{\gamma_{\t_{j-1}}}, & \text{if $1\leq j\leq l$},\\
			\frac{\gamma_{\t_{j-1}}}{\gamma_{\t_{j}}}, & \text{if $l+1\leq j\leq N$.}
		\end{cases}
	\end{equation}

Now combining (\ref{gammat1}) with (\ref{gammat2}) together, we can deduce that \begin{equation}\label{eqa12}
\frac{\gamma_{\t}}{\gamma_{\t\sft{k}}}=\frac{\gamma_{\t^{\blam}}}{\gamma_{\t^{\blam}\sft{k}}}\bigg(\frac{\gamma_{\t_{l+1}}}{\gamma_{\t_{l}}}\bigg)^{2}\cdots\bigg(\frac{\gamma_{\t_{N}}}{\gamma_{\t_{N-1}}}\bigg)^{2}
=\frac{\gamma_{\t^{\blam}}}{\gamma_{\t^{\blam}\sft{k}}}\cdot\big(\frac{\gamma_\t}{\gamma_{\m_{k}(\t)}}\big)^{2}.\end{equation}
This proves the first equality in (\ref{eqa10}).

Replacing $\t$ with $\m_k(\t)$ in (\ref{eqa10}) and using the fact that $\m_k(\m_k\t)=\m_k(\t)$, we see that (\ref{eqa00}) follows from the first equality in (\ref{eqa10}).

Combining (\ref{eqa11}) with (\ref{eqa12}), we can get that
$$\frac{\gamma_{\t}}{\gamma_{\t\sft{k}}}=\frac{\gamma_{\t^{\blam}}}{\gamma_{\t^{\blam}\sft{k}}}\cdot\bigg(\frac{\gamma_{\t_{l}\<k\>}}{\gamma_{\t_{l+1}\<k\>}}\bigg)^{2}\cdots\bigg(\frac{\gamma_{\t_{N-1}\<k\>}}{\gamma_{\t_{N}\<k\>}}
\bigg)^{2}=\frac{\gamma_{\t^{\blam}}}{\gamma_{\t^{\blam}\sft{k}}}\cdot\bigg(\frac{\gamma_{\m_{k}(\t)\sft{k}}}{\gamma_{\t\sft{k}}}\bigg)^{2}.$$
Now applying (\ref{eqa00}), we get $\frac{\gamma_{\m_k(\t)}}{\gamma_{\t^\blam}}=\frac{\gamma_{\m_k(\t)\<k\>}}{\gamma_{\t^{\blam}\<k\>}}$, from which we prove the second equality in (\ref{eqa10}).
\end{proof}

\begin{dfn}\label{2dfns} Let $\blam\in\P_{r,n}$, $\s, \t\in\Std(\blam)$. For each integer $k\in\Z_{\geq 0}$, we define $$
r_{\t,k}:=\frac{\gamma_\t}{\gamma_{\m_{k}(\t)}}\in K^\times,\quad
R_{\s\t,k}:=\frac{\gamma_{\t^{\blam}}}{\gamma_{\t^{\blam}\sft{k}}}r_{\s,k}r_{\t,k}=\frac{\gamma_{\t^{\blam}}}{\gamma_{\t^{\blam}\sft{k}}}\frac{\gamma_{\s}\gamma_\t}{\gamma_{\m_{k}(\s)}\gamma_{\m_{k}(\t)}}\in K^\times.
$$
For simplicity, we use $r_{\t}, R_{\s\t}$ to denote $r_{\t,1}, R_{\s\t,1}$ respectively.
\end{dfn}

\begin{rem}
	By Remark \ref{rmkmid}, we know that $\m_{k}(\t)$ can be defined for $k\in\Z/p\Z$. It is straightforward to check that $r_{\t,k}=r_{\t,k'}$ and $R_{\s\t,k}=R_{\s\t,k'}$ whenever $k\equiv k'\mod p$. Then $r_{\t,k}$ and $R_{\s\t,k}$ can be defined for any $k\in\Z/p\Z$.
\end{rem}

\begin{lem}\label{snphit} Suppose that $\blam\in\P_{r,n}$, $\t\in\Std(\blam)$ and $k\in\Z_{\geq 0}$ is an integer. Then $$
	f_{\t^{\blam}\sft{k}\t^{\blam}\sft{k}}\Phi_{\t}=r_{\t,k}f_{\t^{\blam}\sft{k}\t\sft{k}},\,\,
	\Phi_{\t}^{*}f_{\t^{\blam}\sft{k}\t^{\blam}\sft{k}}=r_{\t,k}f_{\t\sft{k}\t^{\blam}\sft{k}}. $$
Moreover, $$\begin{aligned}
r_{\t,k}&=\frac{\gamma_{\m_{k}(\t)\sft{k}}}{\gamma_{\t\sft{k}}}=\frac{\gamma_{\t^\blam\<k\>}}{\gamma_{\t^\blam}}\frac{\gamma_{\m_{k}(\t)}}{\gamma_{\t\sft{k}}},\\
R_{\s\t,k}&=\frac{\gamma_{\t^{\blam}\sft{k}}}{\gamma_{\t^{\blam}}}\frac{\gamma_{\m_{k}(\s)}\gamma_{\m_{k}(\t)}}{\gamma_{\s\sft{k}}\gamma_{\t\sft{k}}}=\frac{\gamma_{\m_{k}(\s)}\gamma_{\t}}{\gamma_{\s\sft{k}}\gamma_{\m_{k}(\t)}}
=\frac{\gamma_{\s}\gamma_{\m_{k}(\t)}}{\gamma_{\m_{k}(\s)}\gamma_{\t\sft{k}}}.
\end{aligned} $$
\end{lem}

\begin{proof} We keep the same notations as in Definition \ref{midkt} and the proof of Lemma \ref{gtsft}. That is, $d(\t)=x_kd_k$, $x_{k}=s_{i_{1}}s_{i_{2}}\cdots s_{i_{l}}\in\Sym_{(a_{k},n-a_{k})}$, $d_{k}=
s_{i_{l+1}}s_{i_{l+2}}\cdots s_{i_{N}}\in\mathcal{D}_{(a_{k},n-a_{k})}$ are fixed reduced expressions, where $a_{k}=\sum_{i=1}^{k}|\blam^{[i]}|=\sum_{j=1}^{kd}|\lam^{(j)}|$.

Since $\frac{\gamma_{\m_{k}(\t)}}{\gamma_{\t^{\blam}}}=\frac{\gamma_{\m_{k}(\t)\sft{k}}}{\gamma_{\t^{\blam}\sft{k}}}$ by Lemma \ref{gtsft}, we have that $$
\frac{\gamma_{\m_{k}(\t)\sft{k}}}{\gamma_{\t\sft{k}}}=\frac{\gamma_{\t^\blam\<k\>}}{\gamma_{\t^\blam}}\frac{\gamma_{\m_{k}(\t)}}{\gamma_{\t\sft{k}}}.
$$
Now $r_{\t,k}=\frac{\gamma_{\m_{k}(\t)\sft{k}}}{\gamma_{\t\sft{k}}}$ follows from the first equality in (\ref{eqa10}), from which we also get the second equality for $R_{\s\t,k}$.

The second equality for $r_{\t,k}$ follows from (\ref{eqa00}), from which we also get the first equality for $R_{\s\t,k}$. The equality
$\frac{\gamma_{\t^{\blam}\sft{k}}}{\gamma_{\t^{\blam}}}\frac{\gamma_{\m_{k}(\s)}\gamma_{\m_{k}(\t)}}{\gamma_{\s\sft{k}}\gamma_{\t\sft{k}}}=\frac{\gamma_{\s}\gamma_{\m_{k}(\t)}}{\gamma_{\m_{k}(\s)}\gamma_{\t\sft{k}}}$
follows from the first equality in (\ref{eqa10}).

By Lemma \ref{recursiveA}, we have that
	$$\Phi_{\t}=(T_{i_{1}}-A_{i_{1}}(\t_{0}))(T_{i_{2}}-A_{i_{1}}(\t_{1}))\cdots(T_{i_{N}}-A_{i_{N}}(\t_{N-1})).$$
By Proposition \ref{tiact} and Lemma \ref{gammacoeff}, the action of $T_{i}, 1\leq i\leq n-1$ on the seminormal basis is given by
\begin{equation}\label{2casesAction} f_{\u\s}T_{i}=\begin{cases}
		A_{i}(\s)f_{\u\s}+f_{\u\t}, & \text{if $\s\rhd\t:=\s(i,i+1)\in\Std(\blam)$},\\
		A_{i}(\s)f_{\u\s}+\frac{\gamma_{\s}}{\gamma_{\t}}f_{\u\t}, & \text{if $\s\lhd\t:=\s(i,i+1)\in\Std(\blam)$,}
\end{cases}\end{equation}
	where $\blam\in\P_{r,n}$, $\u,\s\in\Std(\blam)$.
For any $1\leq i\leq n$ and $\s\in\Std(\blam)$, it is easy to check that $A_{i}(\s)=A_{i}(\s\sft{k})$ (see (\ref{Ais})). Using (\ref{downUp}) and the formulae in the last paragraph, we can get that
$$
	f_{\t^{\blam}\sft{k}\t^{\blam}\sft{k}}\Phi_{\t}=f_{\t^{\blam}\sft{k}\t\sft{k}}\cdot\frac{\gamma_{\t_{l}\sft{k}}}{\gamma_{\t_{l+1}\sft{k}}}\cdots\frac{\gamma_{\t_{N-1}\sft{k}}}{\gamma_{\t_{N}\sft{k}}}
		=f_{\t^{\blam}\sft{k}\t\sft{k}}\cdot\frac{\gamma_{\m_{k}(\t)\sft{k}}}{\gamma_{\t\sft{k}}}=r_{\t,k}f_{\t^{\blam}\sft{k}\t\sft{k}}.
$$
Applying the anti-involution $\ast$ to the above equation and using Lemma \ref{gammacoeff}, we can get that $\Phi_{\t}^{*}f_{\t^{\blam}\sft{k}\t^{\blam}\sft{k}}=r_{\t,k}f_{\t\sft{k}\t^{\blam}\sft{k}}$.
\end{proof}

\begin{cor} Suppose that $\blam\in\P_{r,n}$, $\t\in\Std(\blam)$ and $k\in\Z/p\Z$. Then $\gamma_\t\gamma_{\t\<k\>}=\gamma_{\m_{k}(\t)}\gamma_{\m_{k}(\t)\sft{k}}$.
\end{cor}

\begin{cor}\label{sigmafst}
	Suppose that $\blam\in\P_{r,n}$ and $\s, \t\in\Std(\blam)$. For any $k\in\Z_{\geq 0}$, we have that
	$$\sigma^{k}(f_{\s\t})=\frac{\gamma_{\t^{\blam}}}{\gamma_{\t^{\blam}\sft{k}}}r_{\s,k}r_{\t,k}f_{\s\sft{k}\t\sft{k}}=R_{\s\t,k}f_{\s\sft{k}\t\sft{k}}.$$
	\end{cor}

\begin{proof}
	Using Lemmas \ref{Ft}, \ref{sigmaFt}, \ref{recursiveA}, we can get that
	$$\sigma^{k}(f_{\s\t})=\sigma^{k}(\Phi_\s^* f_{\t^\blam\t^\blam}\Phi_t)=\Phi_\s^*
\sigma^{k}(f_{\t^\blam\t^\blam})\Phi_t=\frac{\gamma_{\t^{\blam}}}{\gamma_{\t^{\blam}\sft{k}}}\Phi_{\s}^{*}f_{\t^{\blam}\sft{k}\t^{\blam}\sft{k}}\Phi_{\t}.$$
	Now applying Lemma \ref{snphit}, we get that
	$$\sigma^{k}(f_{\s\t})=\frac{\gamma_{\t^{\blam}}}{\gamma_{\t^{\blam}\sft{k}}}r_{\s,k}r_{\t,k}f_{\s\sft{k}\t\sft{k}}.$$
This proves the first equality of the corollary. The second equality of the corollary follows from Definition \ref{2dfns}.
\end{proof}

Next we give some basic properties of the coefficients $R_{\s\t,k}$.

\begin{prop}\label{propRstk}
	Let $\blam\in\P_{r,n}$. Then for any $\s, \t\in\Std(\blam)$ and $k\in\Z_{\geq 0}$, we have
	\begin{itemize}
		\item[(1)] $$R_{\s\t,k}^{2}=\frac{\gamma_{\s}\gamma_{\t}}{\gamma_{\s\sft{k}}\gamma_{\t\sft{k}}};$$
		\item[(2)] For any composition $\mu=(\mu_{1},\mu_{2},\cdots,\mu_{l})$ of the integer $k$, we have
		$$\prod_{i=1}^{l}R_{\s\sft{a_{i}}\t\sft{a_{i}},a_{i+1}-a_{i}}=R_{\s\t,k},$$
		where $a_{i}:=\sum_{j=1}^{i-1}\mu_{j}, 2\leq i\leq l$ and $a_{1}:=0$;
		\item[(3)] $$\prod_{l=0}^{c-1}R_{\s\sft{lk}\t\sft{lk},k}=1,$$
		where the constant $c:=l.c.m(k,p)/k$ and $l.c.m(k,p)$ is the least common multiple of $k$ and $p$.
	\end{itemize}
\end{prop}

\begin{proof} For any standard $\blam$-tableaux $\s,\t,\v\in\Std(\blam)$, we have $f_{\s\t}f_{\t\v}=\gamma_{\t}f_{\s\v}$. Applying the automorphism $\sigma^{k}$, we can get
	$$R_{\s\t,k}R_{\t\v,k}\gamma_{\t\sft{k}}f_{\s\sft{k}\v\sft{k}}=\gamma_{\t}R_{\s\v,k}f_{\s\sft{k}\v\sft{k}}.$$
	Hence, we can get that $R_{\s\t,k}R_{\t\v,k}=\frac{\gamma_{\t}}{\gamma_{\t\sft{k}}}R_{\s\v,k}$. Applying Definition \ref{2dfns}, we can get that
	
$$\frac{\gamma_{\t^{\blam}}}{\gamma_{\t^{\blam}\sft{k}}}r_{\s,k}r_{\t,k}\cdot\frac{\gamma_{\t^{\blam}}}{\gamma_{\t^{\blam}\sft{k}}}r_{\t,k}r_{\v,k}=\frac{\gamma_{\t}}{\gamma_{\t\sft{k}}}\frac{\gamma_{\t^{\blam}}}{\gamma_{\t^{\blam}\sft{k}}}r_{\s,k}r_{\v,k}.$$
	Since $r_{\s,k}, r_{\t,k}, r_{\v,k}$ are invertible elements in $K$, we get that $$
r_{\t,k}^{2}=\frac{\gamma_{\t}\gamma_{\t^{\blam}\sft{k}}}{\gamma_{\t\sft{k}}\gamma_{\t^{\blam}}}. $$
By Definition \ref{2dfns}, we can deduce that
$$
R_{\s\t,k}^{2}=\frac{\gamma_{\s}\gamma_{\t}}{\gamma_{\s\sft{k}}\gamma_{\t\sft{k}}},$$
which proves Part (1) of the corollary.

Let $\mu=(\mu_{1},\mu_{2},\cdots,\mu_{l})$ be a composition of $k$. Since $\sigma^{k}=\sigma^{\mu_{l}}\sigma^{\mu_{l-1}}\cdots\sigma^{\mu_{1}}$, we can get that
$\sigma^{k}(f_{\s\t})=\big(\sigma^{\mu_{l}}\sigma^{\mu_{l-1}}\cdots\sigma^{\mu_{1}}\big)(f_{\s\t})$. Comparing the coefficients, we have that
$$R_{\s\t,k}=\prod_{i=1}^{l}R_{\s\sft{a_{i}}\t\sft{a_{i}},a_{i+1}-a_{i}},$$
which proves Part (2) of the corollary.
	
Now we consider Part (3) of the corollary. By Corollary \ref{sigmafst}, $\sigma^{k}(f_{\s\sft{lk}\t\sft{lk}})=R_{\s\sft{lk}\t\sft{lk},k}f_{\s\sft{(l+1)k}\t\sft{(l+1)k}}$, $\forall\, l\in\Z$. By definition of the
automorphism $\sigma$, we can get that the order of $\sigma^{k}$ is $c=l.c.m(k,p)/k$, hence we have $(\sigma^{k})^{c}(f_{\s\t})=f_{\s\t}$. Comparing the coefficients, we can get that
	$$\prod_{l=0}^{c-1}R_{\s\sft{lk}\t\sft{lk},k}=1,$$
which proves Part (3) of the corollary.
\end{proof}

Now we can give a proof of Theorem \ref{mainthm1}, which reveals some remarkable symmetric property of these constants $r_{\t,k}$. This theorem will play key role in our later construction of seminormal base for
$\HH_{r,p,n}$, see the proof of Lemma \ref{orth}.

\medskip
\noindent
{\bf Proof of Theorem \ref{mainthm1}:} By assumption, $k,l$ are two integers with $0\leq k,l\leq p_\blam$. By the definition of the integer $o_{\blam}$, we know that $\t\sft{lo_{\blam}}$ and $\t\sft{ko_{\blam}}$ are both
standard $\blam$-tableaux. Using the second statement of Proposition \ref{propRstk}, we can get that
	$$R_{\t\t^{\blam},lo_{\blam}}R_{\t\sft{lo_{\blam}}\t^{\blam}\sft{lo_{\blam}},ko_{\blam}}=R_{\t\t^\blam,(k+l)o_\blam}=R_{\t\t^{\blam},ko_{\blam}}R_{\t\sft{ko_{\blam}}\t^{\blam}\sft{ko_{\blam}},lo_{\blam}}.$$
Applying Lemma \ref{snphit}, we can deduce that
$$\frac{\gamma_{\t}\gamma_{\m_{lo_{\blam}}(\t^{\blam})}}{\gamma_{\m_{lo_{\blam}(\t)}}\gamma_{\t^{\blam}\sft{lo_{\blam}}}}\frac{\gamma_{\t\sft{lo_{\blam}}}\gamma_{\m_{ko_{\blam}}(\t^{\blam}\sft{lo_{\blam}})}}{\gamma_{\m_{ko_{\blam}(\t\sft{lo_{\blam}})}}\gamma_{\t^{\blam}\sft{(k+l)o_{\blam}}}}=\frac{\gamma_{\t}\gamma_{\m_{ko_{\blam}}(\t^{\blam})}}{\gamma_{\m_{ko_{\blam}(\t)}}\gamma_{\t^{\blam}\sft{ko_{\blam}}}}\frac{\gamma_{\t\sft{ko_{\blam}}}\gamma_{\m_{lo_{\blam}}(\t^{\blam}\sft{ko_{\blam}})}}{\gamma_{\m_{lo_{\blam}(\t\sft{ko_{\blam}})}}\gamma_{\t^{\blam}\sft{(l+k)o_{\blam}}}}.$$
	
By Definition \ref{2dfns}, the above equality implies that
	$$r_{\t,lo_{\blam}}\frac{\gamma_{\m_{lo_{\blam}}(\t^{\blam})}}{\gamma_{\t^{\blam}\sft{lo_{\blam}}}}\cdot
r_{\t\sft{lo_{\blam}},ko_{\blam}}\frac{\gamma_{\m_{ko_{\blam}}(\t^{\blam}\sft{lo_{\blam}})}}{\gamma_{\t^{\blam}\sft{(k+l)o_{\blam}}}}=r_{\t,ko_{\blam}}\frac{\gamma_{\m_{ko_{\blam}}(\t^{\blam})}}{\gamma_{\t^{\blam}\sft{ko_{\blam}}}}\cdot
r_{\t\sft{ko_{\blam}},lo_{\blam}}\frac{\gamma_{\m_{lo_{\blam}}(\t^{\blam}\sft{ko_{\blam}})}}{\gamma_{\t^{\blam}\sft{(l+k)o_{\blam}}}}.$$
	Since $\m_{lo_{\blam}}(\t^{\blam})=\t^{\blam}=\m_{ko_{\blam}}(\t^{\blam})$, we can get that
	\begin{equation}\label{rgamma}
r_{\t,lo_{\blam}}r_{\t\sft{lo_{\blam}},ko_{\blam}}\frac{\gamma_{\m_{ko_{\blam}}(\t^{\blam}\sft{lo_{\blam}})}}{\gamma_{\t^{\blam}\sft{lo_{\blam}}}}=r_{\t,ko_{\blam}}r_{\t\sft{ko_{\blam}},lo_{\blam}}\frac{\gamma_{\m_{lo_{\blam}}(\t^{\blam}\sft{ko_{\blam}})}}{\gamma_{\t^{\blam}\sft{ko_{\blam}}}}.
	\end{equation}
Now to prove the lemma, it remains to show that \begin{equation}\label{claim1}
\frac{\gamma_{\m_{ko_{\blam}}(\t^{\blam}\sft{lo_{\blam}})}}{\gamma_{\t^{\blam}\sft{lo_{\blam}}}}=\frac{\gamma_{\m_{lo_{\blam}}(\t^{\blam}\sft{ko_{\blam}})}}{\gamma_{\t^{\blam}\sft{ko_{\blam}}}}.\end{equation}
	
Set $a:=\sum_{i=1}^{o_{\blam}}|\blam^{[i]}|=\sum_{j=1}^{do_{\blam}}|\lam^{(j)}|$. Then $n=p_{\blam}a$. There are unique permutations $d(\t\sft{lo_{\blam}})=w_{la,n-la}, d(\t\sft{ko_{\blam}})=w_{ka,n-ka}$ such that
$\t\sft{lo_{\blam}}=\t^{\blam}d(\t\sft{lo_{\blam}})$, $\t\sft{ko_{\blam}}=\t^{\blam}d(\t\sft{ko_{\blam}})$ respectively.
	
Without loss of generality we can assume that $l\leq k$. If $l=k$, it is clear that the claim (\ref{claim1}) trivially holds. Henceforth we assume $l<k$. We have that $w_{ka,n-ka}=w_{la+(k-l)a,n-ka}$ and
$w_{la,n-la}=w_{la,(k-l)a+n-ka}$. By \cite[Lemma 2.8]{HM12}, we can get that
	$$w_{ka,n-ka}=w_{(k-l)a,n-ka}^{\sft{la}}w_{la,n-ka},\quad w_{la,n-la}=w_{la,(k-l)a}w_{la,n-ka}^{\sft{(k-l)a}}.$$
	It is straightforward to check that $w_{(k-l)a,n-ka}^{\sft{la}}\in\Sym_{(la,n-la)}$, $w_{la,n-ka}\in\mathcal{D}_{la,n-la}$ and $w_{la,(k-l)a}\in\Sym_{(ka,n-ka)}$, $w_{la,n-ka}^{\sft{(k-l)a}}\in\mathcal{D}_{ka,n-ka}$. By
Definition \ref{midkt}, we have that $\m_{lo_{\blam}}(\t^{\blam}\sft{ko_{\blam}})=\t^{\blam}w_{(k-l)a,n-ka}^{\sft{la}}$ and $\m_{ko_{\blam}}(\t^{\blam}\sft{lo_{\blam}})=\t^{\blam}w_{la,(k-l)a}$.
	
For simplicity, we set $\t_{0}:=\m_{lo_{\blam}}(\t^{\blam}\sft{ko_{\blam}})$. By Lemma \ref{gammacoeff}, we can get that
	$$\frac{\gamma_{\t^{\blam}\sft{ko_{\blam}}}}{\gamma_{\m_{lo_{\blam}}(\t^{\blam}\sft{ko_{\blam}})}}
	=\underbrace{B_{la}(\t_{0}s_{la})\cdots B_{1}(\t_{0}s_{la}\cdots s_{1})}\cdots\underbrace{B_{n-(k-l)a-1}(\t_{0}s_{la}\cdots s_{n-(k-l)a-1})\cdots B_{n-ka}(\t_{0}w_{la,n-ka})}.$$
	
	Consider the standard tableau $\t_{0}=\m_{lo_{\blam}}(\t^{\blam}\sft{ko_{\blam}})$. The integers $1,\cdots, la$ sit inside
	$$\bigg(\t_{0}^{[1]},\cdots,\t_{0}^{[lo_{\blam}]}\bigg)$$
	and integers $la+1,\cdots,n-(k-l)a$ sit inside
	$$\bigg(\t_{0}^{[ko_{\blam}+1]},\cdots,\t_{0}^{[p]}\bigg).$$
	For $1\leq i\leq la$ and $1\leq j\leq n-ka$, the node $\big(\t_{0}\big)^{-1}(i)$ coincides with the node $\big(\t^{\blam}\big)^{-1}(i)$, and the node $\big(\t_{0}\big)^{-1}(la+j)$ coincides with the node
$\big(\t^{\blam}\big)^{-1}(ka+j)$.
	
	Combining the definition of $w_{la,n-ka}$, we can get that
\begin{equation}\label{expressProd1}
\frac{\gamma_{\t^{\blam}\sft{ko_{\blam}}}}{\gamma_{\m_{lo_{\blam}}(\t^{\blam}\sft{ko_{\blam}})}}=\prod_{j=ka+1}^{n}\prod_{i=1}^{la}\frac{(\res_{\t^{\blam}}(j)-q\res_{\t^{\blam}}(i))(q\res_{\t^{\blam}}(j)-
\res_{\t^{\blam}}(i))}{(\res_{\t^{\blam}}(j)-\res_{\t^{\blam}}(i))^{2}}.\end{equation}
	
	Similarly, we set $\t_{1}:=\m_{ko_{\blam}}(\t^{\blam}\sft{lo_{\blam}})$. We know that $\t_{1}w_{la,n-ka}^{\sft{(k-l)a}}=\t^{\blam}\sft{lo_{\blam}}$. Consider the standard tableau $\t_{1}$. The integers $(k-l)a+1,\cdots,
ka$ sit inside
	$$\bigg(\t_{1}^{[1]},\cdots,\t_{1}^{[lo_{\blam}]}\bigg)$$
	and integers $ka+1,\cdots,n$ sit inside
	$$\bigg(\t_{1}^{[ko_{\blam}+1]},\cdots,\t_{1}^{[p]}\bigg).$$
	For $1\leq i\leq la$ and $1\leq j\leq n$, the node $\big(\t_{1}\big)^{-1}((k-l)a+i)$ coincides with the node $\big(\t^{\blam}\big)^{-1}(i)$, and the node $\big(\t_{1}\big)^{-1}(ka+j)$ coincides with the node
$\big(\t^{\blam}\big)^{-1}(ka+j)$.
	Then we can get that
\begin{equation}\label{expressProd2}
\frac{\gamma_{\t^{\blam}\sft{lo_{\blam}}}}{\gamma_{\m_{ko_{\blam}}(\t^{\blam}\sft{lo_{\blam}})}}=\prod_{j=ka+1}^{n}\prod_{i=1}^{la}\frac{(\res_{\t^{\blam}}(j)-q\res_{\t^{\blam}}(i))(q\res_{\t^{\blam}}(j)-
\res_{\t^{\blam}}(i))}{(\res_{\t^{\blam}}(j)-\res_{\t^{\blam}}(i))^{2}}\end{equation}
	and hence $\frac{\gamma_{\m_{lo_{\blam}}(\t^{\blam}\sft{ko_{\blam}})}}{\gamma_{\t^{\blam}\sft{ko_{\blam}}}}=\frac{\gamma_{\m_{ko_{\blam}}(\t^{\blam}\sft{lo_{\blam}})}}{\gamma_{\t^{\blam}\sft{lo_{\blam}}}}$.
This proves the claim (\ref{claim1}).
	
Finally, since $\frac{\gamma_{\m_{lo_{\blam}}(\t^{\blam}\sft{ko_{\blam}})}}{\gamma_{\t^{\blam}\sft{ko_{\blam}}}}=\frac{\gamma_{\m_{ko_{\blam}}(\t^{\blam}\sft{lo_{\blam}})}}{\gamma_{\t^{\blam}\sft{lo_{\blam}}}}\in K$ is
invertible, from (\ref{rgamma}) we can deduce that
	 $r_{\t,lo_{\blam}}r_{\t\sft{lo_{\blam}},ko_{\blam}}=r_{\t,ko_{\blam}}r_{\t\sft{ko_{\blam}},lo_{\blam}}$. This completes the proof of Theorem \ref{mainthm1}.	\hfill\qed
\medskip

The next proposition deals with Theorem \ref{mainthm2a}.

\begin{prop}\label{squareProp}
	Let $\blam\in\P_{r,n}$. Then there exists an invertible element $h_{\blam}\in K^{\times}$ such that
	$$\frac{\gamma_{\t^{\blam}\<o_\blam\>}}{\gamma_{\t^{\blam}}}=\bigl(h_{\blam}\bigr)^{2}.$$
We set $h_{\blam,0,1}:=h_{\blam}$ for later use.
\end{prop}

\begin{proof}
	If $o_{\blam}=p$, then $\t^{\blam}\<o_\blam\>=\t^\blam$ and we can take $h_\blam:=1$ and we are done. Henceforth we assume that $o_\blam<p$. In his case, we can write
	$$\blam=(\blam^{[1]}, \cdots, \blam^{[p]})=(\underbrace{\underbrace{\blam^{[1]}, \cdots, \blam^{[o_{\blam}]}}, \cdots, \underbrace{\blam^{[1]}, \cdots, \blam^{[o_{\blam}]}}}_{p_{\blam} copies}).$$
By assumption, $|\blam|=n$. Set $a:=\sum_{k=1}^{o_{\blam}}|\blam^{[k]}|=\sum_{i=1}^{do_{\blam}}|\lam^{(i)}|$. It follows that $n=ap_{\blam}$.
	
	Since $\t^{\blam}\sft{o_{\blam}}$ is a standard $\blam$-tableau, there exists a unique element $d(\t^{\blam}\sft{o_{\blam}})\in\Sym_{n}$ such that $\t^{\blam}\sft{o_{\blam}}=\t^{\blam}d(\t^{\blam}\sft{o_{\blam}})$.
Note also that $$
\ell(d(\t^{\blam}\sft{o_{\blam}}))=\sum_{k=0}^{p_\blam-2}\ell(w_{a,a}^{\sft{ka}})=(p_\blam-1)a^2 .
$$

We set $\t_{0}=\t^{\blam}$ and for each $1\leq k\leq p_{\blam}-1$, $$
\t_{k}=\t^{\blam}w_{a,a}^{\sft{n-2a}}w_{a,a}^{\sft{n-3a}}\cdots w_{a,a}^{\sft{n-(k+1)a}}. $$
In particular, $\t_{p_{\blam}-1}=\t^{\blam}\sft{o_{\blam}}$. Hence, we get a sequence of standard $\blam$-tableaux as follows:
	$$\t^{\blam}=\t_{0}\triangleright \t_{1}\triangleright \t_{2}\triangleright \cdots \triangleright \t_{p_{\blam}-1}=\t^{\blam}\sft{o_{\blam}}.$$
	
For any integer $k$ with $1\leq k\leq p_{\blam}-1$, by definition, we have that $$\begin{aligned}
\t_k&=\t_{k-1}w_{a,a}^{\sft{n-(k+1)a}}\\
&=\t_{k-1}\underbrace{s_{n-ka}\cdots s_{n-ka-a+2}s_{n-ka-a+1}}\underbrace{s_{n-ka+1}\cdots s_{n-ka-a+3}s_{n-ka-a+2}}\cdots\\
&\qquad \underbrace{s_{n-ka+a-1}\cdots s_{n-ka+1}s_{n-ka}}.
\end{aligned}
$$

Note that for each $1\leq j\leq a$, the node in which the integer $n-ka+j$ sits in $\t_{k-1}$ is the same as the node in which the integer $n-a+j$ sits in $\t^\blam$, while the integer $n-(k+1)a+j$ sits in the same node
both in $\t_{k-1}$ and in $\t^\blam$. Applying Lemma \ref{gammacoeff}, we can deduce that	\begin{equation}\label{gammres}\begin{aligned}
\frac{\gamma_{\t_{k}}}{\gamma_{\t_{k-1}}}&=\underbrace{B_{n-ka}(\t_{k-1}s_{n-ka})\cdots B_{n-ka-a+1}(\t_{k-1}s_{n-ka}\cdots s_{n-ka-a+1})}\cdots\\
&\qquad \underbrace{B_{n-ka+a-1}(\t_{k-1}s_{n-ka}\cdots s_{n-ka+a-1})\cdots B_{n-ka}(\t^{\blam}\sft{o_\blam})}\\
&=\prod_{i=1}^{a}\prod_{j=1}^{a}\frac{\bigl(q\res_{\t^{\blam}}(n-(k+1)a+j)-\res_{\t^{\blam}}(n-a+i)\bigr)\big(\res_{\t^{\blam}}(n-(k+1)a+j)-q\res_{\t^{\blam}}(n-a+i)\big)}{(\res_{\t^{\blam}}(n-(k+1)a+j)-
\res_{\t^{\blam}}(n-a+i))^{2}} .
	\end{aligned}\end{equation}
	
If $k$ lies in row $c$ and column $b$ of $\s^{(l+md)}$, the residue $\res_{\s}(k)=\varepsilon^{m}q^{b-c}Q_{l}$. By the definition of $\t^{\blam}$, we have that
	\begin{equation}\label{epsres}
		\begin{aligned}
		\res_{\t^{\blam}}(n-(k+1)a+j)&=\varepsilon^{(p_{\blam}-k-1)o_{\blam}}\res_{\t^{\blam}}(j),\quad \forall\,1\leq j\leq a;\\
		\res_{\t^{\blam}}(n-a+i)&=\varepsilon^{(p_{\blam}-1)o_{\blam}}\res_{\t^{\blam}}(i),\quad \forall\,1\leq i\leq a.
	\end{aligned}\end{equation}
We can get that
	\begin{equation}\label{gammatk}\begin{aligned}
\frac{\gamma_{\t_{k}}}{\gamma_{t_{k-1}}}&=\prod_{i=1}^{a}\prod_{j=1}^{a}\frac{(q\varepsilon^{(p_{\blam}-k-1)o_{\blam}}\res_{\t^{\blam}}(j)-\varepsilon^{(p_{\blam}-1)o_{\blam}}\res_{\t^{\blam}}(i))(\varepsilon^{(p_{\blam}-k-1)
o_{\blam}}\res_{\t^{\blam}}(j)-q\varepsilon^{(p_{\blam}-1)o_{\blam}}\res_{\t^{\blam}}(i))}{(\varepsilon^{(p_{\blam}-k-1)o_{\blam}}\res_{\t^{\blam}}(j)-\varepsilon^{(p_{\blam}-1)o_{\blam}}\res_{\t^{\blam}}(i))^{2}}\\
&=\prod_{i=1}^{a}\prod_{j=1}^{a}\frac{(q\res_{\t^{\blam}}(j)-\varepsilon^{ko_{\blam}}\res_{\t^{\blam}}(i))(\res_{\t^{\blam}}(j)-q\varepsilon^{ko_{\blam}}\res_{\t^{\blam}}(i))}{(\res_{\t^{\blam}}(j)
-\varepsilon^{ko_{\blam}}\res_{\t^{\blam}}(i))^{2}}.
	\end{aligned}\end{equation}
Note that \begin{equation}\label{tip1}
	\frac{\gamma_{\t^{\blam}\sft{o_{\blam}}}}{\gamma_{\t^{\blam}}}=\frac{\gamma_{\t_{1}}}{\gamma_{\t_{0}}}\cdot \frac{\gamma_{\t_{2}}}{\gamma_{\t_{1}}}\cdots
\frac{\gamma_{\t_{p_{\blam}-1}}}{\gamma_{\t_{p_{\blam}-2}}}.\end{equation}
It follows that $$
\frac{\gamma_{\t^{\blam}\sft{o_{\blam}}}}{\gamma_{\t^{\blam}}}=\prod_{k=1}^{p_{\blam}-1}\prod_{i=1}^{a}\prod_{j=1}^{a}\frac{(q\res_{\t^{\blam}}(j)-\varepsilon^{ko_{\blam}}\res_{\t^{\blam}}(i))(\res_{\t^{\blam}}(j)
-q\varepsilon^{ko_{\blam}}\res_{\t^{\blam}}(i))}{(\res_{\t^{\blam}}(j)-\varepsilon^{ko_{\blam}}\res_{\t^{\blam}}(i))^{2}}.$$
	
Since $\varepsilon$ is a primitive $p$th root of unity in $K$, $\varepsilon^{o_{\blam}}$ is a primitive $p_{\blam}$th root of unity in $K$. For any $1\leq k\leq p_{\blam}-1$, $1\leq i,j\leq a$, we can get that
	$$\begin{aligned}
		&\quad \,
\frac{(q\res_{\t^{\blam}}(j)-\varepsilon^{ko_{\blam}}\res_{\t^{\blam}}(i))(\res_{\t^{\blam}}(j)-q\varepsilon^{ko_{\blam}}\res_{\t^{\blam}}(i))}{(\res_{\t^{\blam}}(j)-\varepsilon^{ko_{\blam}}\res_{\t^{\blam}}(i))^{2}}\cdot
\frac{\varepsilon^{2(p_{\blam}-k)o_{\blam}}}{\varepsilon^{2(p_{\blam}-k)o_{\blam}}}\\
&=\frac{(q\res_{\t^{\blam}}(i)-\varepsilon^{(p_{\blam}-k)o_{\blam}}\res_{\t^{\blam}}(j))(\res_{\t^{\blam}}(i)-q\varepsilon^{(p_{\blam}-k)o_{\blam}}\res_{\t^{\blam}}(j))}{(\res_{\t^{\blam}}(i)-\varepsilon^{(p_{\blam}-k)
o_{\blam}}\res_{\t^{\blam}}(j))^{2}}.
	\end{aligned}$$
Combining this with (\ref{gammatk}), we have that
	\begin{equation}\label{epsres2}
		\frac{\gamma_{\t_{k}}}{\gamma_{\t_{k-1}}}=\frac{\gamma_{\t_{p_{\blam}-k}}}{\gamma_{\t_{p_{\blam}-k-1}}}, \quad \forall\, 1\leq k\leq p_{\blam}-1.
	\end{equation}

If $p_{\blam}$ is odd, then we have that $$
\frac{\gamma_{\t^{\blam}\sft{o_{\blam}}}}{\gamma_{\t^{\blam}}}=\prod_{k=1}^{(p_{\blam}-1)/2}\prod_{i=1}^{a}\prod_{j=1}^{a}\frac{(q\res_{\t^{\blam}}(j)-\varepsilon^{ko_{\blam}}\res_{\t^{\blam}}(i))^{2}(\res_{\t^{\blam}}(j)
-q\varepsilon^{ko_{\blam}}\res_{\t^{\blam}}(i))^{2}}{(\res_{\t^{\blam}}(j)-\varepsilon^{ko_{\blam}}\res_{\t^{\blam}}(i))^{4}}.$$
	
If $p_{\blam}$ is even, then $2|p=p_\blam o_\blam$. In particular, $\cha K$ is odd and $2\cdot 1_K$ is invertible in $K$. In this case, we can get that $\varepsilon^{p_{\blam}o_{\blam}/2}=-1$ and hence
	$$\begin{aligned}
\frac{\gamma_{\t^{\blam}\sft{o_{\blam}}}}{\gamma_{\t^{\blam}}}&=\prod_{k=1}^{p_{\blam}/2-1}\prod_{i=1}^{a}\prod_{j=1}^{a}\frac{(q\res_{\t^{\blam}}(j)-\varepsilon^{ko_{\blam}}\res_{\t^{\blam}}(i))^{2}(\res_{\t^{\blam}}(j)
-q\varepsilon^{ko_{\blam}}\res_{\t^{\blam}}(i))^{2}}{(\res_{\t^{\blam}}(j)-\varepsilon^{ko_{\blam}}\res_{\t^{\blam}}(i))^{4}}\\
		&\qquad\times  \prod_{i=1}^{a}\prod_{j=1}^{a}\frac{(q\res_{\t^{\blam}}(j)+\res_{\t^{\blam}}(i))(\res_{\t^{\blam}}(j)+q\res_{\t^{\blam}}(i))}{(\res_{\t^{\blam}}(j)+\res_{\t^{\blam}}(i))^{2}}\\
		&=\prod_{k=1}^{p_{\blam}/2-1}\prod_{i=1}^{a}\prod_{j=1}^{a}\frac{(q\res_{\t^{\blam}}(j)-\varepsilon^{ko_{\blam}}\res_{\t^{\blam}}(i))^{2}(\res_{\t^{\blam}}(j)
-q\varepsilon^{ko_{\blam}}\res_{\t^{\blam}}(i))^{2}}{(\res_{\t^{\blam}}(j)-\varepsilon^{ko_{\blam}}\res_{\t^{\blam}}(i))^{4}}\\
		&\qquad \times\frac{(q+1)^{2a}}{4^{a}}\cdot \prod_{1\leq j<i\leq a}
\frac{(q\res_{\t^{\blam}}(j)+\res_{\t^{\blam}}(i))^{2}(\res_{\t^{\blam}}(j)+q\res_{\t^{\blam}}(i))^{2}}{(\res_{\t^{\blam}}(j)+\res_{\t^{\blam}}(i))^{4}}.
	\end{aligned}$$
	
Therefore, if $p_{\blam}$ is odd, then we set \begin{equation}\label{odddefinition1}
h_{\blam}:=\prod_{k=1}^{(p_{\blam}-1)/2}\prod_{i=1}^{a}\prod_{j=1}^{a}\frac{(q\res_{\t^{\blam}}(j)-\varepsilon^{ko_{\blam}}\res_{\t^{\blam}}(i))(\res_{\t^{\blam}}(j)
-q\varepsilon^{ko_{\blam}}\res_{\t^{\blam}}(i))}{(\res_{\t^{\blam}}(j)-\varepsilon^{ko_{\blam}}\res_{\t^{\blam}}(i))^{2}};\end{equation}
while if $p_{\blam}$ is even, then we set \begin{equation}\label{evendefinition1}
\begin{aligned}
		h_{\blam}&=\prod_{k=1}^{p_{\blam}/2-1}\prod_{i=1}^{a}\prod_{j=1}^{a}\frac{(q\res_{\t^{\blam}}(j)-\varepsilon^{ko_{\blam}}\res_{\t^{\blam}}(i))(\res_{\t^{\blam}}(j)
-q\varepsilon^{ko_{\blam}}\res_{\t^{\blam}}(i))}{(\res_{\t^{\blam}}(j)-\varepsilon^{ko_{\blam}}\res_{\t^{\blam}}(i))^{2}}\\
		&\qquad\times \frac{(q+1)^{a}}{2^{a}}\cdot \prod_{1\leq j<i\leq a} \frac{(q\res_{\t^{\blam}}(j)+\res_{\t^{\blam}}(i))(\res_{\t^{\blam}}(j)+q\res_{\t^{\blam}}(i))}{(\res_{\t^{\blam}}(j)
+\res_{\t^{\blam}}(i))^{2}} .
	\end{aligned}\end{equation}
In both cases, $h_{\blam}\in K^\times$ is invertible and it satisfies that $$\frac{\gamma_{\t^{\blam}\sft{o_{\blam}}}}{\gamma_{\t^{\blam}}}=h_{\blam}^{2}.$$
This completes the proof of the proposition.
\end{proof}

\medskip\noindent
{\bf{Proof of Theorem \ref{mainthm2a}:}} This follows from Proposition \ref{squareProp}.
\qed
\medskip

\begin{dfn}\label{hlaml}
	Suppose that $\blam\in\P_{r,n}$. For any $l\in\Z_{\geq 0}$, we define
	\begin{equation}\label{hlaml1}
h_{\blam,l,0}:=1,\,\,\,h_{\blam,l,1}:=h_{\blam,0,1}\frac{\gamma_{\m_{o_{\blam}}(\tlam\sft{lo_{\blam}})}}{\gamma_{\tlam\sft{lo_{\blam}}}}\in K^{\times}.
	\end{equation}
	
	For any integers $l_{1}\in \Z_{\geq 0}$ and $l_{2}\in\Z_{\geq 1}$, we recursively define
	\begin{equation}\label{hlaml2}
		h_{\blam,l_{1},l_{2}}:=h_{\blam,l_{1},l_{2}-1}h_{\blam,0,1}\frac{\gamma_{\m_{o_{\blam}}(\tlam\sft{(l_{1}+l_{2}-1)o_{\blam}})}}{\gamma_{\tlam\sft{(l_{1}+l_{2}-1)o_{\blam}}}}\in K^{\times}.
	\end{equation}
\end{dfn}

\begin{prop}\label{prophlam}
	Suppose that $\blam\in\P_{r,n}$. Let $l,l_{1},l_{2}\in \Z_{\geq 0}$. Then we have
	\begin{itemize}
		\item[(1)] 	 $h_{\blam,l_{1},l_{2}}=h_{\blam,0,1}^{l_{2}}\prod_{k=0}^{l_{2}-1}\frac{\gamma_{\m_{o_{\blam}}(\tlam\sft{(l_{1}+k)o_{\blam}})}}{\gamma_{\tlam\sft{(l_{1}+k)o_{\blam}}}}$. In particular,
$h_{\blam,0,l}=h_{\blam,0,1}^{l}\prod_{k=1}^{l-1}\frac{\gamma_{\m_{o_{\blam}}(\t^{\blam}\sft{ko_{\blam}})}}{\gamma_{\t^{\blam}\sft{ko_{\blam}}}}$.
		\item[(2)]  $h_{\blam,l_{1},l_{2}}^{2}=\frac{\gamma_{\tlam\sft{(l_{1}+l_{2})o_{\blam}}}}{\gamma_{\tlam\sft{l_{1}o_{\blam}}}}$. In particular,
$h_{\blam,0,l}^{2}=\frac{\gamma_{\t^{\blam}\sft{lo_{\blam}}}}{\gamma_{\t^{\blam}}}$.
	\end{itemize}
\end{prop}

\begin{proof}
By Definition \ref{hlaml} and (\ref{hlaml2}), we can get that	
$$h_{\blam,l_{1},l_{2}}=h_{\blam,l_{1},l_{2}-1}h_{\blam,0,1}\frac{\gamma_{\m_{o_{\blam}}(\tlam\sft{(l_{1}+l_{2}-1)o_{\blam}})}}{\gamma_{\tlam\sft{(l_{1}+l_{2}-1)o_{\blam}}}}=h_{\blam,l_{1},1}h_{\blam,0,1}^{l_{2}-1}\prod_{k=1}^{l_{2}-1}\frac{\gamma_{\m_{o_{\blam}}(\tlam\sft{(l_{1}+k)o_{\blam}})}}{\gamma_{\tlam\sft{(l_{1}+k)o_{\blam}}}}.$$
	By Definition \ref{hlaml}, we have that $h_{\blam,l_{1},1}=h_{\blam,0,1}\frac{\gamma_{\m_{o_{\blam}}(\tlam\sft{l_{1}o_{\blam}})}}{\gamma_{\tlam\sft{l_{1}o_{\blam}}}}$ and hence
	$$\begin{aligned}
		h_{\blam,l_{1},l_{2}}&=h_{\blam,0,1}\frac{\gamma_{\m_{o_{\blam}}(\tlam\sft{l_{1}o_{\blam}})}}{\gamma_{\tlam\sft{l_{1}o_{\blam}}}}\times
h_{\blam,0,1}^{l_{2}-1}\prod_{k=1}^{l_{2}-1}\frac{\gamma_{\m_{o_{\blam}}(\tlam\sft{(l_{1}+k)o_{\blam}})}}{\gamma_{\tlam\sft{(l_{1}+k)o_{\blam}}}}\\
		&=h_{\blam,0,1}^{l_{2}}\prod_{k=0}^{l_{2}-1}\frac{\gamma_{\m_{o_{\blam}}(\tlam\sft{(l_{1}+k)o_{\blam}})}}{\gamma_{\tlam\sft{(l_{1}+k)o_{\blam}}}}.
	\end{aligned}$$
This proves the first statement.

For the second statement, if the second equality  $h_{\blam,0,l}^{2}=\frac{\gamma_{\t^{\blam}\sft{lo_{\blam}}}}{\gamma_{\t^{\blam}}}$ holds, then we can deduce that
$$
h_{\blam,l_{1},l_{2}}^{2}=h_{\blam,0,l_{1}+l_{2}}^{2}/h_{\blam,0,l_{1}}^{2}=\frac{\gamma_{\t^{\blam}\sft{(l_{1}+l_{2})o_{\blam}}}}{\gamma_{\t^{\blam}}}\frac{\gamma_{\t^{\blam}}}{\gamma_{\t^{\blam}\sft{l_{1}o_{\blam}}}}
=\frac{\gamma_{\t^{\blam}\sft{(l_{1}+l_{2})o_{\blam}}}}{\gamma_{\t^{\blam}\sft{l_{1}o_{\blam}}}},$$
as expected. Therefore, it remains to show the second equality $h_{\blam,0,l}^{2}=\frac{\gamma_{\t^{\blam}\sft{lo_{\blam}}}}{\gamma_{\t^{\blam}}}$ holds.
	
We use induction on $l$. If $l=1$, then that second equality follows from Proposition \ref{squareProp}.
	
Assume $h_{\blam,0,l-1}^{2}=\frac{\gamma_{\t^{\blam}\sft{(l-1)o_{\blam}}}}{\gamma_{\t^{\blam}}}$. By (\ref{hlaml2}), we have that
	$$h_{\blam,0,l}:=h_{\blam,0,l-1}h_{\blam,0,1}\frac{\gamma_{\m_{o_{\blam}}(\tlam\sft{(l-1)o_{\blam}})}}{\gamma_{\tlam\sft{(l-1)o_{\blam}}}}.$$
	By assumption, we can get that
	$$h_{\blam,0,l}^{2}=h_{\blam,0,l-1}^{2}
\bigg(h_{\blam,0,1}\frac{\gamma_{\m_{o_{\blam}}(\t^{\blam}\sft{(l-1)o_{\blam}})}}{\gamma_{\t^{\blam}\sft{(l-1)o_{\blam}}}}\bigg)^{2}=\frac{\gamma_{\t^{\blam}\sft{(l-1)o_{\blam}}}}{\gamma_{\t^{\blam}}}h_{\blam,0,1}^{2}\bigg(\frac{\gamma_{\m_{o_{\blam}}(\t^{\blam}\sft{(l-1)o_{\blam}})}}{\gamma_{\t^{\blam}\sft{(l-1)o_{\blam}}}}\bigg)^{2}.$$
	By Lemma \ref{gtsft} and (\ref{claim1}), we can get that
$$
\frac{\gamma_{\t^{\blam}\sft{lo_{\blam}}}}{\gamma_{\t^{\blam}\sft{o_{\blam}}}}=\frac{\gamma_{\t^{\blam}\sft{(l-1)o_{\blam}}}}{\gamma_{\t^{\blam}}}\bigg(\frac{\gamma_{\m_{(l-1)o_{\blam}}(\t^{\blam}\sft{o_{\blam}})}}{\gamma_{\t^{\blam}\sft{o_{\blam}}}}\bigg)^{2}=\frac{\gamma_{\t^{\blam}\sft{(l-1)o_{\blam}}}}{\gamma_{\t^{\blam}}}\bigg(\frac{\gamma_{\m_{o_{\blam}}(\t^{\blam}\sft{(l-1)o_{\blam}})}}{\gamma_{\t^{\blam}\sft{(l-1)o_{\blam}}}}\bigg)^{2}.$$
	
	Combining Proposition \ref{squareProp}, we can deduce that
	
$$h_{\blam,0,l}^{2}=\frac{\gamma_{\t^{\blam}\sft{(l-1)o_{\blam}}}}{\gamma_{\t^{\blam}}}\frac{\gamma_{\t^{\blam}\sft{o_{\blam}}}}{\gamma_{\t^{\blam}}}\bigg(\frac{\gamma_{\m_{(l-1)o_{\blam}}(\t^{\blam}\sft{o_{\blam}})}}{\gamma_{\t^{\blam}\sft{o_{\blam}}}}\bigg)^{2}=\frac{\gamma_{\t^{\blam}\sft{o_{\blam}}}}{\gamma_{\t^{\blam}}}\frac{\gamma_{\t^{\blam}\sft{lo_{\blam}}}}{\gamma_{\t^{\blam}\sft{o_{\blam}}}}=\frac{\gamma_{\t^{\blam}\sft{lo_{\blam}}}}{\gamma_{\t^{\blam}}}.$$
	This completes the proof of the Proposition.
\end{proof}

\begin{lem}\label{sqhlam}
	Suppose that $\blam\in\P_{r,n}$. Then we have
	$$h_{\blam,0,p_{\blam}}=1.$$
\end{lem}

\begin{proof}
	By Proposition \ref{prophlam}, we have
	$$h_{\blam,0,p_{\blam}}=h_{\blam,0,1}^{p_{\blam}}\prod_{k=1}^{p_{\blam}-1}\frac{\gamma_{\m_{o_{\blam}}(\t^{\blam}\sft{ko_{\blam}})}}{\gamma_{\t^{\blam}\sft{ko_{\blam}}}}. $$
To prove the lemma, it suffices to show \begin{equation}\label{sqhlam1}
		h_{\blam,0,1}^{p_{\blam}}=\prod_{k=1}^{p_{\blam}-1}\frac{\gamma_{\t^{\blam}\sft{ko_{\blam}}}}{\gamma_{\m_{o_{\blam}}(\t^{\blam}\sft{ko_{\blam}})}}.
	\end{equation}\

As before, we set $a:=\sum_{i=1}^{o_{\blam}}|\blam^{[i]}|$. By (\ref{expressProd1}) and noting that $\res_{\t^\blam}(ca+j)=\eps^{co_\blam}\res_{\t^\blam}(j)$ for any $1\leq j\leq a$, we have that (for any $1\leq k\leq
p_{\blam}-1$),
	$$\begin{aligned}
\frac{\gamma_{\t^{\blam}\sft{ko_{\blam}}}}{\gamma_{\m_{o_{\blam}}(\t^{\blam}\sft{ko_{\blam}})}}&=\prod_{j=ka+1}^{n}\prod_{i=1}^{a}\frac{(\res_{\t^{\blam}}(j)-q\res_{\t^{\blam}}(i))(q\res_{\t^{\blam}}(j)-
\res_{\t^{\blam}}(i))}{(\res_{\t^{\blam}}(j)-\res_{\t^{\blam}}(i))^{2}}\\
&=\prod_{c=k}^{p_{\blam}-1}\prod_{j=1}^{a}\prod_{i=1}^{a}\frac{(q\res_{\t^{\blam}}(j)-\eps^{co_{\blam}}\res_{\t^{\blam}}(i))(\res_{\t^{\blam}}(j)-q\eps^{co_{\blam}}\res_{\t^{\blam}}(i))}{(\res_{\t^{\blam}}(j)-\eps^{co_{\blam}}\res_{\t^{\blam}}(i))^{2}}
	\end{aligned}$$
	Since $\eps$ is a primitive $p$-th root of unity, then $\eps^{o_{\blam}}$ is a $p_{\blam}$-th root of unity. Then we can get that
	$$\begin{aligned}
		&\quad
\prod_{c=p_{\blam}-k+1}^{p_{\blam}-1}\prod_{j=1}^{a}\prod_{i=1}^{a}\frac{(q\res_{\t^{\blam}}(j)-\eps^{co_{\blam}}\res_{\t^{\blam}}(i))(\res_{\t^{\blam}}(j)-q\eps^{co_{\blam}}\res_{\t^{\blam}}(i))}{(\res_{\t^{\blam}}(j)-
\eps^{co_{\blam}}\res_{\t^{\blam}}(i))^{2}}\\
&=\prod_{c=1}^{k-1}\prod_{j=1}^{a}\prod_{i=1}^{a}\frac{(q\res_{\t^{\blam}}(j)-\eps^{co_{\blam}}\res_{\t^{\blam}}(i))(\res_{\t^{\blam}}(j)-q\eps^{co_{\blam}}\res_{\t^{\blam}}(i))}{(\res_{\t^{\blam}}(j)-
\eps^{co_{\blam}}\res_{\t^{\blam}}(i))^{2}},
	\end{aligned}$$
	and hence
	$$\begin{aligned}
		\frac{\gamma_{\t^{\blam}\sft{ko_{\blam}}}}{\gamma_{\m_{o_{\blam}}(\t^{\blam}\sft{ko_{\blam}})}}&=\prod_{j=ka+1}^{n}\prod_{i=1}^{a}\frac{(\res_{\t^{\blam}}(j)-q\res_{\t^{\blam}}(i))(q\res_{\t^{\blam}}(j)-
\res_{\t^{\blam}}(i))}{(\res_{\t^{\blam}}(j)-\res_{\t^{\blam}}(i))^{2}}\\		
&=\prod_{c=k}^{p_{\blam}-1}\prod_{j=1}^{a}\prod_{i=1}^{a}\frac{(q\res_{\t^{\blam}}(j)-\eps^{co_{\blam}}\res_{\t^{\blam}}(i))(\res_{\t^{\blam}}(j)-q\eps^{co_{\blam}}\res_{\t^{\blam}}(i))}{(\res_{\t^{\blam}}(j)-
\eps^{co_{\blam}}\res_{\t^{\blam}}(i))^{2}}\\
&=\prod_{c=k}^{p_{\blam}-k}\prod_{j=1}^{a}\prod_{i=1}^{a}\frac{(q\res_{\t^{\blam}}(j)-\eps^{co_{\blam}}\res_{\t^{\blam}}(i))(\res_{\t^{\blam}}(j)-q\eps^{co_{\blam}}\res_{\t^{\blam}}(i))}{(\res_{\t^{\blam}}(j)-
\eps^{co_{\blam}}\res_{\t^{\blam}}(i))^{2}}\\
&\qquad\times\prod_{c=p_{\blam}-k+1}^{p_{\blam}-1}\prod_{j=1}^{a}\prod_{i=1}^{a}\frac{(q\res_{\t^{\blam}}(j)-\eps^{co_{\blam}}\res_{\t^{\blam}}(i))(\res_{\t^{\blam}}(j)-q\eps^{co_{\blam}}\res_{\t^{\blam}}(i))}{(\res_{\t^{\blam}}(j)-
\eps^{co_{\blam}}\res_{\t^{\blam}}(i))^{2}}\\
&=\prod_{c=1}^{p_{\blam}-k}\prod_{j=1}^{a}\prod_{i=1}^{a}\frac{(q\res_{\t^{\blam}}(j)-\eps^{co_{\blam}}\res_{\t^{\blam}}(i))(\res_{\t^{\blam}}(j)-q\eps^{co_{\blam}}\res_{\t^{\blam}}(i))}{(\res_{\t^{\blam}}(j)-
\eps^{co_{\blam}}\res_{\t^{\blam}}(i))^{2}}.
	\end{aligned}$$

For each $1\leq k\leq p_{\blam}-1$, we want to compute the product
	$$\frac{\gamma_{\t^{\blam}\sft{ko_{\blam}}}}{\gamma_{\m_{o_{\blam}}(\t^{\blam}\sft{ko_{\blam}})}}
	\frac{\gamma_{\t^{\blam}\sft{p-ko_{\blam}}}}{\gamma_{\m_{o_{\blam}}(\t^{\blam}\sft{p-ko_{\blam}})}}.$$
Assume that $p_{\blam}$ is odd. In this case, without loss of generality, we can assume that $1\leq k\leq (p_{\blam}-1)/2$. By computation, we can get that
	\begin{equation}\label{sqhlam2}
	\begin{aligned}
	&\quad\, \frac{\gamma_{\t^{\blam}\sft{ko_{\blam}}}}{\gamma_{\m_{o_{\blam}}(\t^{\blam}\sft{ko_{\blam}})}}
	\frac{\gamma_{\t^{\blam}\sft{p-ko_{\blam}}}}{\gamma_{\m_{o_{\blam}}(\t^{\blam}\sft{p-ko_{\blam}})}}\\
	&=\prod_{c=1}^{p_{\blam}-k}\prod_{j=1}^{a}\prod_{i=1}^{a}\frac{(q\res_{\t^{\blam}}(j)-\eps^{co_{\blam}}\res_{\t^{\blam}}(i))(\res_{\t^{\blam}}(j)-q\eps^{co_{\blam}}\res_{\t^{\blam}}(i))}{(\res_{\t^{\blam}}(j)-
\eps^{co_{\blam}}\res_{\t^{\blam}}(i))^{2}}\\
&\qquad\times\prod_{c=p_{\blam}-k}^{p_{\blam}-1}\prod_{j=1}^{a}\prod_{i=1}^{a}\frac{(q\res_{\t^{\blam}}(j)-\eps^{co_{\blam}}\res_{\t^{\blam}}(i))(\res_{\t^{\blam}}(j)-q\eps^{co_{\blam}}\res_{\t^{\blam}}(i))}{(\res_{\t^{\blam}}(j)-
\eps^{co_{\blam}}\res_{\t^{\blam}}(i))^{2}}\\
&=\prod_{c=1}^{p_{\blam}-1}\prod_{j=1}^{a}\prod_{i=1}^{a}\frac{(q\res_{\t^{\blam}}(j)-\eps^{co_{\blam}}\res_{\t^{\blam}}(i))(\res_{\t^{\blam}}(j)-q\eps^{co_{\blam}}\res_{\t^{\blam}}(i))}{(\res_{\t^{\blam}}(j)-
\eps^{co_{\blam}}\res_{\t^{\blam}}(i))^{2}}\\
&\qquad\times\prod_{j=1}^{a}\prod_{i=1}^{a}\frac{(q\res_{\t^{\blam}}(j)-\eps^{p-ko_{\blam}}\res_{\t^{\blam}}(i))(\res_{\t^{\blam}}(j)-q\eps^{p-ko_{\blam}}\res_{\t^{\blam}}(i))}{(\res_{\t^{\blam}}(j)-
\eps^{p-ko_{\blam}}\res_{\t^{\blam}}(i))^{2}}.
\end{aligned}
	\end{equation}
	By (\ref{odddefinition1}),  if $p_{\blam}$ is odd then
	$$h_{\blam,0,1}=\prod_{k=1}^{(p_{\blam}-1)/2}\prod_{i=1}^{a}\prod_{j=1}^{a}\frac{(q\res_{\t^{\blam}}(j)-\varepsilon^{ko_{\blam}}\res_{\t^{\blam}}(i))(\res_{\t^{\blam}}(j)
		-q\varepsilon^{ko_{\blam}}\res_{\t^{\blam}}(i))}{(\res_{\t^{\blam}}(j)-\varepsilon^{ko_{\blam}}\res_{\t^{\blam}}(i))^{2}}.$$
	Then we can deduce that
	$$\frac{\gamma_{\t^{\blam}\sft{ko_{\blam}}}}{\gamma_{\m_{o_{\blam}}(\t^{\blam}\sft{ko_{\blam}})}}
	\frac{\gamma_{\t^{\blam}\sft{p-ko_{\blam}}}}{\gamma_{\m_{o_{\blam}}(\t^{\blam}\sft{p-ko_{\blam}})}}=h_{\blam,0,1}^{2}\prod_{j=1}^{a}\prod_{i=1}^{a}\frac{(q\res_{\t^{\blam}}(j)-
\eps^{p-ko_{\blam}}\res_{\t^{\blam}}(i))(\res_{\t^{\blam}}(j)-q\eps^{p-ko_{\blam}}\res_{\t^{\blam}}(i))}{(\res_{\t^{\blam}}(j)-\eps^{p-ko_{\blam}}\res_{\t^{\blam}}(i))^{2}}.$$
	Since $p_{\blam}$ is odd, we can deduce that
	$$\begin{aligned}
	&\quad \prod_{k=1}^{p_{\blam}-1}\frac{\gamma_{\t^{\blam}\sft{ko_{\blam}}}}{\gamma_{\m_{o_{\blam}}(\t^{\blam}\sft{ko_{\blam}})}}
	=\prod_{k=1}^{(p_{\blam}-1)/2}\frac{\gamma_{\t^{\blam}\sft{ko_{\blam}}}}{\gamma_{\m_{o_{\blam}}(\t^{\blam}\sft{ko_{\blam}})}}
	\frac{\gamma_{\t^{\blam}\sft{p-ko_{\blam}}}}{\gamma_{\m_{o_{\blam}}(\t^{\blam}\sft{p-ko_{\blam}})}}\\
	&=h_{\blam,0,1}^{p_{\blam}-1}\prod_{k=1}^{(p_{\blam}-1)/2}\prod_{j=1}^{a}\prod_{i=1}^{a}\frac{(q\res_{\t^{\blam}}(j)-\eps^{p-ko_{\blam}}\res_{\t^{\blam}}(i))(\res_{\t^{\blam}}(j)-
q\eps^{p-ko_{\blam}}\res_{\t^{\blam}}(i))}{(\res_{\t^{\blam}}(j)-\eps^{p-ko_{\blam}}\res_{\t^{\blam}}(i))^{2}}\\
&=h_{\blam,0,1}^{p_{\blam}-1}\prod_{k=1}^{(p_{\blam}-1)/2}\prod_{j=1}^{a}\prod_{i=1}^{a}\frac{(q\res_{\t^{\blam}}(j)-\eps^{ko_{\blam}}\res_{\t^{\blam}}(i))(\res_{\t^{\blam}}(j)-q\eps^{ko_{\blam}}\res_{\t^{\blam}}(i))}{(\res_{\t^{\blam}}(j)-\eps^{ko_{\blam}}\res_{\t^{\blam}}(i))^{2}}\\
	&=h_{\blam,0,1}^{p_{\blam}}.
	\end{aligned}$$
	This proves the lemma for odd $p_{\blam}$. By the same methods, one can prove the lemma for even $p_{\blam}$.
\end{proof}

\begin{cor}\label{congruence}
Let $\blam\in\P_{r,n}$ and $l_{1},l_{2}\in\Z_{\geq 0}$. For any $l_{1}',l_{2}'\in\Z_{\geq 0}$ such that $l_{1}\equiv l_{1}' \mod p_{\blam}$ and $l_{2}\equiv l_{2}' \mod p_{\blam}$, we have $$h_{\blam,l_{1},l_{2}}=h_{\blam,l_{1}',l_{2}'}.$$
\end{cor}

\begin{proof} We have $\tlam\sft{p_{\blam}o_{\blam}}=\tlam$ and $\tlam\sft{l_{1}o_{\blam}}=\tlam\sft{l_{1}'o_{\blam}}$ if $l_{1}\equiv l_{1}' \mod p_{\blam}$. By Definition \ref{hlaml}, it is straightforward to check that $h_{\blam,l_{1},1}=h_{\blam,l_{1}',1}$ if $l_{1}\equiv l_{1}' \mod p_{\blam}$.
		
	By the first statement of Proposition \ref{prophlam}, we have that
	$$h_{\blam,l_{1}',l_{2}'}=h_{\blam,0,1}^{l_{2}'}\prod_{k=0}^{l_{2}'-1}\frac{\gamma_{\m_{o_{\blam}}(\tlam\sft{(l_{1}'+k)o_{\blam}})}}{\gamma_{\tlam\sft{(l_{1}'+k)o_{\blam}}}}.$$
	Since $l_{1}\equiv l_{1}'\mod p_{\blam}$, we have $l_{1}+k\equiv l_{1}'+k, \forall\, 0\leq k\leq l_{2}'-1$. Then we can get that $\tlam\sft{(l_{1}+k)o_{\blam}}=\tlam\sft{(l_{1}'+k)o_{\blam}},\forall\, 0\leq k\leq l_{2}'-1$.
Without loss of generality, we assume $l'_2\geq l_2$. Hence we can get that
	$$\begin{aligned}
h_{\blam,l_{1}',l_{2}'}&=h_{\blam,0,1}^{l_{2}'}\prod_{k=0}^{l_{2}'-1}\frac{\gamma_{\m_{o_{\blam}}(\tlam\sft{(l_{1}'+k)o_{\blam}})}}{\gamma_{\tlam\sft{(l_{1}'+k)o_{\blam}}}}=h_{\blam,0,1}^{l_{2}'}\prod_{k=0}^{l_{2}'-1}\frac{\gamma_{\m_{o_{\blam}}(\tlam\sft{(l_{1}+k)o_{\blam}})}}{\gamma_{\tlam\sft{(l_{1}+k)o_{\blam}}}}\\
		&=h_{\blam,0,1}^{l_{2}}\prod_{k=0}^{l_{2}-1}\frac{\gamma_{\m_{o_{\blam}}(\tlam\sft{(l_{1}+k)o_{\blam}})}}{\gamma_{\tlam\sft{(l_{1}+k)o_{\blam}}}}\times h_{\blam,0,1}^{l_{2}'-l_{2}}\prod_{k=l_{2}}^{l_{2}'-1}\frac{\gamma_{\m_{o_{\blam}}(\tlam\sft{(l_{1}+k)o_{\blam}})}}{\gamma_{\tlam\sft{(l_{1}+k)o_{\blam}}}}\\
		&=h_{\blam,l_{1},l_{2}}\times h_{\blam,0,1}^{l_{2}'-l_{2}}\prod_{k=l_{2}}^{l_{2}'-1}\frac{\gamma_{\m_{o_{\blam}}(\tlam\sft{(l_{1}+k)o_{\blam}})}}{\gamma_{\tlam\sft{(l_{1}+k)o_{\blam}}}}.
	\end{aligned}$$
	Since $l_{2}\equiv l_{2}'\mod p_{\blam}$, we have $l_{2}'=l_{2}+xp_{\blam}$ for some $x\in\Z_{\geq 0}$. By computation, we can get that
	$$\begin{aligned}
	&\,\,\quad h_{\blam,0,1}^{l_{2}'-l_{2}}\prod_{k=l_{2}}^{l_{2}'-1}\frac{\gamma_{\m_{o_{\blam}}(\tlam\sft{(l_{1}+k)o_{\blam}})}}{\gamma_{\tlam\sft{(l_{1}+k)o_{\blam}}}}
	=h_{\blam,0,1}^{xp_{\blam}}\prod_{k=l_{2}}^{l_{2}+xp_{\blam}-1}\frac{\gamma_{\m_{o_{\blam}}(\tlam\sft{(l_{1}+k)o_{\blam}})}}{\gamma_{\tlam\sft{(l_{1}+k)o_{\blam}}}}\\
	&=h_{\blam,0,1}^{xp_{\blam}}\prod_{c=0}^{x-1}\prod_{k=l_{2}+cp_{\blam}}^{l_{2}+(c+1)p_{\blam}-1}\frac{\gamma_{\m_{o_{\blam}}(\tlam\sft{(l_{1}+k)o_{\blam}})}}{\gamma_{\tlam\sft{(l_{1}+k)o_{\blam}}}}
	=h_{\blam,0,1}^{xp_{\blam}}\prod_{c=0}^{x-1}\prod_{k=cp_{\blam}}^{(c+1)p_{\blam}-1}\frac{\gamma_{\m_{o_{\blam}}(\tlam\sft{(l_{1}+l_{2}+k)o_{\blam}})}}{\gamma_{\tlam\sft{(l_{1}+l_{2}+k)o_{\blam}}}}\\
	&=h_{\blam,0,1}^{xp_{\blam}}\bigg(\prod_{k=1}^{p_{\blam}-1}\frac{\gamma_{\m_{o_{\blam}}(\t^{\blam}\sft{ko_{\blam}})}}{\gamma_{\t^{\blam}\sft{ko_{\blam}}}}\bigg)^{x}.
	\end{aligned}$$
	By Lemma \ref{sqhlam}, we know that $h_{\blam,0,p_{\blam}}=1$, that is
	$$h_{\blam,0,p_{\blam}}=h_{\blam,0,1}^{p_{\blam}}\prod_{k=1}^{p_{\blam}-1}\frac{\gamma_{\m_{o_{\blam}}(\t^{\blam}\sft{ko_{\blam}})}}{\gamma_{\t^{\blam}\sft{ko_{\blam}}}}=1.$$
	It follows that
	$$h_{\blam,0,1}^{l_{2}'-l_{2}}\prod_{k=l_{2}}^{l_{2}'-1}\frac{\gamma_{\m_{o_{\blam}}(\tlam\sft{(l_{1}+k)o_{\blam}})}}{\gamma_{\tlam\sft{(l_{1}+k)o_{\blam}}}}=1$$
	and hence $h_{\blam,l_{1},l_{2}}=h_{\blam,l_{1}',l_{2}'}$.
\end{proof}

\begin{rem}
	By the aforementioned Corollary, the elements $h_{\blam,l_{1},l_{2}}$ can be defined for any $l_{1},l_{2}\in\Z/p_{\blam}\Z$. Henceforth we shall identify $\Z/p_{\blam}\Z$ with $\{0,1,\cdots,p_\blam-1\}$, and use the notation $l_{1},l_{2}\in\Z/p_{\blam}\Z$ to mean $l_{1},l_{2}\in\{0,1,\cdots,p_\blam-1\}$.
\end{rem}

\begin{lem}\label{hlaml1l2}
	Suppose that $\blam\in\P_{r,n}$. Let $ l_{1},l_{2}\in\Z/p_\blam\Z$. Then we have
	$$h_{\blam,0,l_{1}+l_{2}}=h_{\blam,0,l_{1}}h_{\blam,l_{1},l_{2}}.$$
\end{lem}

\begin{proof}
		If $l_{1}+l_{2}\leq p_{\blam}$, by the first statement of Proposition \ref{prophlam}, we can get that
	$$\begin{aligned}
		h_{\blam,0,l_{1}+l_{2}}&=h_{\blam,0,1}^{l_{1}+l_{2}}\prod_{k=0}^{l_{1}+l_{2}-1}\frac{\gamma_{\m_{o_{\blam}}(\tlam\sft{ko_{\blam}})}}{\gamma_{\tlam\sft{ko_{\blam}}}}\\
		&=h_{\blam,0,1}^{l_{1}}\prod_{k=1}^{l_{1}-1}\frac{\gamma_{\m_{o_{\blam}}(\t^{\blam}\sft{ko_{\blam}})}}{\gamma_{\t^{\blam}\sft{ko_{\blam}}}}\times
		h_{\blam,0,1}^{l_{2}}\prod_{k=l_{1}}^{l_{1}+l_{2}-1}\frac{\gamma_{\m_{o_{\blam}}(\t^{\blam}\sft{ko_{\blam}})}}{\gamma_{\t^{\blam}\sft{ko_{\blam}}}}\\
		&=h_{\blam,0,l_{1}}\times h_{\blam,0,1}^{l_{2}}\prod_{k=0}^{l_{2}-1}\frac{\gamma_{\m_{o_{\blam}}(\t^{\blam}\sft{{l_{1}+k}o_{\blam}})}}{\gamma_{\t^{\blam}\sft{(l_{1}+k)o_{\blam}}}}\\
		&=h_{\blam,0,l_{1}}h_{\blam,l_{1},l_{2}}.
	\end{aligned}$$
	
	If $l_{1}+l_{2}>p_{\blam}$, there is a unique $l\in\Z/p_{\blam}\Z$ such that $l_{1}+l_{2}=l+p_{\blam}$. Combining the first statement of Proposition \ref{prophlam} and Lemma \ref{sqhlam}, we can get that
	$$\begin{aligned}
		h_{\blam,0,l_{1}+l_{2}}&=h_{\blam,0,p_{\blam}}h_{\blam,0,l}=h_{\blam,0,p_{\blam}}h_{\blam,0,1}^{l}\prod_{k=0}^{l-1}\frac{\gamma_{\m_{o_{\blam}}(\tlam\sft{ko_{\blam}})}}{\gamma_{\tlam\sft{ko_{\blam}}}}\\
		&=h_{\blam,0,1}^{p_{\blam}}\prod_{k=1}^{p_{\blam}-1}\frac{\gamma_{\m_{o_{\blam}}(\t^{\blam}\sft{ko_{\blam}})}}{\gamma_{\t^{\blam}\sft{ko_{\blam}}}}\times
		h_{\blam,0,1}^{l}\prod_{k=l}^{l-1}\frac{\gamma_{\m_{o_{\blam}}(\t^{\blam}\sft{ko_{\blam}})}}{\gamma_{\t^{\blam}\sft{ko_{\blam}}}}\\
		&=h_{\blam,0,1}^{l_{1}}\prod_{k=1}^{l_{1}-1}\frac{\gamma_{\m_{o_{\blam}}(\t^{\blam}\sft{ko_{\blam}})}}{\gamma_{\t^{\blam}\sft{ko_{\blam}}}}\times h_{\blam,0,1}^{p_{\blam}-l_{1}}\prod_{k=l_{1}}^{p_{\blam}-1}\frac{\gamma_{\m_{o_{\blam}}(\t^{\blam}\sft{ko_{\blam}})}}{\gamma_{\t^{\blam}\sft{ko_{\blam}}}}\times
		h_{\blam,0,1}^{l}\prod_{k=1}^{l-1}\frac{\gamma_{\m_{o_{\blam}}(\t^{\blam}\sft{ko_{\blam}})}}{\gamma_{\t^{\blam}\sft{ko_{\blam}}}}\\
		&=h_{\blam,0,l_{1}}\times h_{\blam,0,1}^{l_{2}}\prod_{k=0}^{l_{2}-1}\frac{\gamma_{\m_{o_{\blam}}(\t^{\blam}\sft{{l_{1}+k}o_{\blam}})}}{\gamma_{\t^{\blam}\sft{(l_{1}+k)o_{\blam}}}}\quad
\text{(by the proof in last paragraph and $p_\blam-l_1+l=l_2$)}\\
		&=h_{\blam,0,l_{1}}h_{\blam,l_{1},l_{2}}.
	\end{aligned}$$
	This proves the lemma.
\end{proof}

\begin{lem}\label{hlamQuo}
	Suppose that $\blam\in\P_{r,n}$. Let $l_{1},l_{2}\in\Z/p_{\blam}\Z$. Then we have that
	$$h_{\blam,l_{1},l_{2}}/h_{\blam,0,l_{2}}=\frac{\gamma_{\m_{l_{2}o_{\blam}}(\tlam\sft{l_{1}o_{\blam}})}}{\gamma_{\tlam\sft{l_{1}o_{\blam}}}}.$$
\end{lem}

\begin{proof}
	By the first statement of Proposition \ref{prophlam}, we have that
	$$h_{\blam,l_{1},l_{2}}=h_{\blam,0,1}^{l_{2}}\prod_{k=0}^{l_{2}-1}\frac{\gamma_{\m_{o_{\blam}}(\tlam\sft{(l_{1}+k)o_{\blam}})}}{\gamma_{\tlam\sft{(l_{1}+k)o_{\blam}}}},\quad
h_{\blam,0,l_{2}}=h_{\blam,0,1}^{l_{2}}\prod_{k=1}^{l_{2}-1}\frac{\gamma_{\m_{o_{\blam}}(\t^{\blam}\sft{ko_{\blam}})}}{\gamma_{\t^{\blam}\sft{ko_{\blam}}}}.$$
	By Definition \ref{2dfns}, we can deduce that
	$$h_{\blam,l_{1},l_{2}}=h_{\blam,0,1}^{l_{2}}\prod_{k=0}^{l_{2}-1}r_{\tlam\sft{(l_{1}+k)o_{\blam}},o_{\blam}}^{-1},\quad
h_{\blam,0,l_{2}}=h_{\blam,0,1}^{l_{2}}\prod_{k=1}^{l_{2}-1}r_{\tlam\sft{ko_{\blam}},o_{\blam}}^{-1}.$$
	
	Note that $\m_{x}(\tlam)=\tlam, \forall x\in\Z/p\Z$. By the second statement of Proposition \ref{propRstk}, we have that
	$$R_{\tlam\sft{o_{\blam}}\tlam, (l_{2}-1)o_{\blam}}=\prod_{k=1}^{l_{2}-1}R_{\tlam\sft{ko_{\blam}}\tlam\sft{(k-1)o_{\blam}}, o_{\blam}}.$$
	Combining this with Lemma \ref{snphit}, we can deduce that
	$$r_{\tlam\sft{o_{\blam}},(l_{2}-1)o_{\blam}}\frac{\gamma_{\tlam}}{\gamma_{\tlam\sft{(l_{2}-1)o_{\blam}}}}=\prod_{k=1}^{l_{2}-1}r_{\tlam\sft{ko_{\blam}},
o_{\blam}}\frac{\gamma_{\m_{o_{\blam}}(\tlam\sft{(k-1)o_{\blam}})}}{\gamma_{\tlam\sft{ko_{\blam}}}}.$$
	It follows that
	\begin{equation}\label{rlam1}
		\prod_{k=1}^{l_{2}-1}r_{\tlam\sft{ko_{\blam}}, o_{\blam}}=r_{\tlam\sft{o_{\blam}},(l_{2}-1)o_{\blam}}\prod_{k=1}^{l_{2}-2}\frac{\gamma_{\tlam\sft{ko_{\blam}}}}{\gamma_{\m_{o_{\blam}}(\tlam\sft{(k-1)o_{\blam}})}}.
	\end{equation}
	
	Similarly, by the second statement of Proposition \ref{propRstk}, we have $$R_{\tlam\sft{l_{1}o_{\blam}}\tlam, l_{2}o_{\blam}}=\prod_{k=0}^{l_{2}-1}R_{\tlam\sft{(l_{1}+k)o_{\blam}}\tlam\sft{ko_{\blam}}, o_{\blam}},$$
that is
	\begin{equation}\label{rlam2}
		\prod_{k=0}^{l_{2}-1}r_{\tlam\sft{(l_{1}+k)o_{\blam}},
o_{\blam}}=r_{\tlam\sft{l_{1}o_{\blam}},l_{2}o_{\blam}}\prod_{k=1}^{l_{2}-1}\frac{\gamma_{\tlam\sft{ko_{\blam}}}}{\gamma_{\m_{o_{\blam}}(\tlam\sft{(k-1)o_{\blam}})}}.
	\end{equation}
	Henceforth, combining (\ref{rlam1}) and (\ref{rlam2}), we have that
	
$$h_{\blam,l_{1},l_{2}}/h_{\blam,0,l_{2}}=\frac{r_{\tlam\sft{o_{\blam}},(l_{2}-1)o_{\blam}}}{r_{\tlam\sft{l_{1}o_{\blam}},l_{2}o_{\blam}}}\frac{\gamma_{\m_{o_{\blam}}(\tlam\sft{(l_{2}-1)o_{\blam}})}}{\gamma_{\tlam\sft{(l_{2}-1)o_{\blam}}}}.$$
	By Definition \ref{2dfns} and (\ref{claim1}), we have that
	$$r_{\tlam\sft{o_{\blam}},(l_{2}-1)o_{\blam}}=\frac{\gamma_{\tlam\sft{o_{\blam}}}}{\gamma_{\m_{(l_{2}-1)o_{\blam}}(\tlam\sft{o_{\blam}})}}
	=\frac{\gamma_{\tlam\sft{(l_{2}-1)o_{\blam}}}}{\gamma_{\m_{o_{\blam}}(\tlam\sft{(l_{2}-1)o_{\blam}})}}, \quad
r_{\tlam\sft{l_{1}o_{\blam}},l_{2}o_{\blam}}=\frac{\gamma_{\tlam\sft{l_{1}o_{\blam}}}}{\gamma_{\m_{l_{2}o_{\blam}}(\tlam\sft{l_{1}o_{\blam}})}}.$$
	Then we can get that
	$$h_{\blam,l_{1},l_{2}}/h_{\blam,0,l_{2}}=\frac{\gamma_{\m_{l_{2}o_{\blam}}(\tlam\sft{l_{1}o_{\blam}})}}{\gamma_{\tlam\sft{l_{1}o_{\blam}}}}.$$
This proves the lemma.
\end{proof}

Recall that $\sim_{\sigma}$ is the equivalence relation on $\P_{r,n}$, which is defined in Section 2. $\P_{r,n}^{\sigma}$ is the set of equivalence classes. Then $\blam\in\P_{r,n}^{\sigma}$ means that $\blam\in\P_{r,n}$ is
a fixed representative of an equivalence class. By Remark \ref{rmkmid}, for any standard tableau $\t$, we can define $\m_{k}(\t)$ for any integers $k\in\Z_{\geq 0}$.

\begin{dfn}\label{sqrtt}
Suppose that $\blam\in\P_{r,n}^{\sigma}$ and $\t\in\Std(\blam)$ such that $\t^{-1}(1)\in\big(\blam^{[1]},\cdots,\blam^{[o_{\blam}]}\big)$. For any integers $l,l_{1},l_{2}\in\Z_{\geq 0}$, we define
$$h_{\t}^{\sft{l}}:=h_{\blam,0,l}\frac{\gamma_{\m_{lo_{\blam}}(\t)}}{\gamma_{\t}}\in K^{\times},\quad
h_{\t\sft{l_{1}o_{\blam}}}^{\sft{l_{2}}}:=h_{\t}^{\sft{l_{2}}}\frac{\gamma_{\t}\gamma_{\m_{l_{1}o_{\blam}}(\t\sft{l_{2}o_{\blam}})}}{\gamma_{\m_{l_{1}o_{\blam}}(\t)}\gamma_{\t\sft{l_{2}o_{\blam}}}}\in K^{\times}.$$
\end{dfn}
In other words, we have give the definition of $h_\t^{\<l\>}$ for any $\t\in\Std(\blam)$, $\blam\in\P_{r,n}^{\sigma}$ and integer $l\in\Z_{\geq 0}$, without requiring $\t^{-1}(1)\in\big(\blam^{[1]},\cdots,\blam^{[o_{\blam}]}\big)$.

\begin{rem}
	By Remark \ref{rmkmid}, if non-negative integers $k\equiv k'\mod p$ then $\m_{k}(\t)=\m_{k'}(\t)$ for any standard tableau $\t$. Henceforth, if $\t\in\Std(\blam)$ and $l\equiv l'\mod p_{\blam}$, then $\m_{lo_{\blam}}(\t)=\m_{l'o_{\blam}}(\t)$. It is straightforward to check that $h_{\t}^{\sft{l}}=h_{\t}^{\sft{l'}}$. Similarly, it is easy to check that $h_{\t\sft{l_{1}'o_{\blam}}}^{\sft{l_{2}'}}=h_{\t\sft{l_{1}o_{\blam}}}^{\sft{l_{2}}}$ if $l_{1}\equiv l_{1}'\mod p_{\blam}$ and $l_{2}\equiv l_{2}'\mod p_{\blam}$. Hence, the elements $h_{\t\sft{l_{1}o_{\blam}}}^{\sft{l_{2}}}$ can be defined for any $l_{1},l_{2}\in\Z/p_{\blam}\Z$.
\end{rem}

\begin{cor}\label{plamht}
	Suppose that $\blam\in\P_{r,n}^{\sigma}$ and $\t\in\Std(\blam)$ such that $\t^{-1}(1)\in\big(\blam^{[1]},\cdots,\blam^{[o_{\blam}]}\big)$. For any $l\in\Z_{\geq 0}$, we have
	$$h_{\t\sft{lo_{\blam}}}^{p_{\blam}}=1.$$
\end{cor}

\begin{proof}
	By Lemma \ref{sqhlam}, we know that $h_{\blam,0,p_{\blam}}=1$. Note that $\m_{p_{\blam}o_{\blam}}(\t)=\m_{p}(\t)=\t$. We can deduce that
	$$h_{\t}^{\sft{p_{\blam}}}=h_{\blam,0,p_{\blam}}\frac{\gamma_{\m_{p_{\blam}o_{\blam}}(\t)}}{\gamma_{\t}}=1.$$
	Similarly, we can get that
	$$h_{\t\sft{lo_{\blam}}}^{\sft{p_{\blam}}}:=h_{\t}^{\sft{p_{\blam}}}\frac{\gamma_{\t}\gamma_{\m_{lo_{\blam}}(\t\sft{p_{\blam}o_{\blam}})}}{\gamma_{\m_{lo_{\blam}}(\t)}\gamma_{\t\sft{p_{\blam}o_{\blam}}}}=1.$$
This proves the corollary.
\end{proof}

\begin{lem}\label{squareht}
	Suppose that $\blam\in\P_{r,n}^{\sigma}$ and $\t\in\Std(\blam)$ with $1\in\big(\t^{[1]},\cdots,\t^{[o_{\blam}]}\big)$. For any $l, l_{1}, l_{2}\in\Z/p_{\blam}\Z$, we have
	$$(h_{\t}^{\sft{l}})^{2}=\frac{\gamma_{\t\sft{lo_{\blam}}}}{\gamma_{\t}},\quad (h_{\t\sft{l_{1}o_{\blam}}}^{\sft{l_{2}}})^{2}=\frac{\gamma_{\t\sft{(l_{1}+l_{2})o_{\blam}}}}{\gamma_{\t\sft{l_{1}o_{\blam}}}}.$$
\end{lem}

\begin{proof}
	By Lemma \ref{gtsft} and Definition \ref{sqrtt}, it easy to check that
	$$(h_{\t}^{\sft{l}})^{2}=h_{\blam,0,l}^{2}\bigg(\frac{\gamma_{\m_{lo_{\blam}}(\t)}}{\gamma_{\t}}\bigg)^{2}=\frac{\gamma_{\t\sft{lo_{\blam}}}}{\gamma_{\t}}.$$
	For the second equality, by Theorem \ref{mainthm1}, we have that
$$h_{\t\sft{l_{1}o_{\blam}}}^{\sft{l_{2}}}=h_{\t}^{\sft{l_{2}}}\frac{\gamma_{\t}\gamma_{\m_{l_{2}o_{\blam}}(\t\sft{l_{1}o_{\blam}})}}{\gamma_{\m_{l_{2}o_{\blam}}(\t)}\gamma_{\t\sft{(l_{1}+l_{2})o_{\blam}}}}\frac{\gamma_{\t\sft{(l_{1}+l_{2})o_{\blam}}}}{\gamma_{\t\sft{l_{1}o_{\blam}}}}$$
	and hence
	
$$(h_{\t\sft{l_{1}o_{\blam}}}^{\sft{l_{2}}})^{2}=(h_{\t}^{\sft{l_{2}}})^{2}\bigg(\frac{\gamma_{\t}\gamma_{\m_{l_{2}o_{\blam}}(\t\sft{l_{1}o_{\blam}})}}{\gamma_{\m_{l_{2}o_{\blam}}(\t)}\gamma_{\t\sft{(l_{1}+l_{2})o_{\blam}}}}\bigg)^{2}\bigg(\frac{\gamma_{\t\sft{(l_{1}+l_{2})o_{\blam}}}}{\gamma_{\t\sft{l_{1}o_{\blam}}}}\bigg)^{2}.$$
	Then we can get that
	$$\begin{aligned}
		(h_{\t\sft{l_{1}o_{\blam}}}^{\sft{l_{2}}})^{2}&=\frac{\gamma_{\t\sft{l_{2}o_{\blam}}}}{\gamma_{\t}}R_{\t\t\sft{l_{1}o_{\blam}},
l_{2}o_{\blam}}^{2}\bigg(\frac{\gamma_{\t\sft{(l_{1}+l_{2})o_{\blam}}}}{\gamma_{\t\sft{l_{1}o_{\blam}}}}\bigg)^{2}\\
&=\frac{\gamma_{\t\sft{l_{2}o_{\blam}}}}{\gamma_{\t}}\frac{\gamma_{\t}\gamma_{\t\sft{l_{1}o_{\blam}}}}{\gamma_{\t\sft{l_{2}o_{\blam}}}\gamma_{\t\sft{(l_{1}+l_{2})o_{\blam}}}}\bigg(\frac{\gamma_{\t\sft{(l_{1}+l_{2})o_{\blam}}}}{\gamma_{\t\sft{l_{1}o_{\blam}}}}\bigg)^{2}\\
		&=\frac{\gamma_{\t\sft{(l_{1}+l_{2})o_{\blam}}}}{\gamma_{\t\sft{l_{1}o_{\blam}}}}.
	\end{aligned}$$
This proves the lemma.
\end{proof}


\begin{lem}\label{htkl}
	Suppose that $\blam\in\P_{r,n}^{\sigma}$ and $\t\in\Std(\blam)$ with $1\in\big(\t^{[1]},\cdots,\t^{[o_{\blam}]}\big)$. Let $l_{1},l_{2}\in\Z/p_{\blam}\Z$. Then we have
	$$h_{\t}^{\sft{l_{1}+l_{2}}}=h_{\t}^{\sft{l_{1}}}h_{\t\sft{l_{1}o_{\blam}}}^{\sft{l_{2}}}.$$
\end{lem}

\begin{proof}
	By Definition \ref{sqrtt}, we have that
	$$h_{\t}^{\sft{l_{1}+l_{2}}}=h_{\blam,0,l_{1}+l_{2}}\frac{\gamma_{\m_{(l_{1}+l_{2})o_{\blam}}(\t)}}{\gamma_{\t}},\quad h_{\t}^{\sft{l_{1}}}=h_{\blam,0,l_{1}}\frac{\gamma_{\m_{l_{1}o_{\blam}}(\t)}}{\gamma_{\t}},$$
	and that
$$h_{\t\sft{l_{1}o_{\blam}}}^{\sft{l_{2}}}=h_{\t}^{\sft{l_{2}}}\frac{\gamma_{\t}\gamma_{\m_{l_{1}o_{\blam}}(\t\sft{l_{2}o_{\blam}})}}{\gamma_{\m_{l_{1}o_{\blam}}(\t)}\gamma_{\t\sft{l_{2}o_{\blam}}}}=h_{\blam,0,l_{2}}\frac{\gamma_{\m_{l_{2}o_{\blam}}(\t)}}{\gamma_{\t}}\frac{\gamma_{\t}\gamma_{\m_{l_{1}o_{\blam}}(\t\sft{l_{2}o_{\blam}})}}{\gamma_{\m_{l_{1}o_{\blam}}(\t)}\gamma_{\t\sft{l_{2}o_{\blam}}}}=h_{\blam,0,l_{2}}\frac{\gamma_{\m_{l_{2}o_{\blam}}(\t)}\gamma_{\m_{l_{1}o_{\blam}}(\t\sft{l_{2}o_{\blam}})}}{\gamma_{\m_{l_{1}o_{\blam}}(\t)}\gamma_{\t\sft{l_{2}o_{\blam}}}}.$$
	By the second statement of Proposition \ref{prophlam} and Lemma \ref{hlamQuo}, we have
	$$h_{\blam,0,l_{1}+l_{2}}=h_{\blam,0,l_{1}}h_{\blam,l_{1},l_{2}}=h_{\blam,0,l_{1}}h_{\blam,0,l_{2}}\frac{\gamma_{\m_{l_{2}o_{\blam}}(\tlam\sft{l_{1}o_{\blam}})}}{\gamma_{\tlam\sft{l_{1}o_{\blam}}}}.$$
	Then we can deduce that $h_{\t}^{\sft{l_{1}+l_{2}}}=h_{\t}^{\sft{l_{1}}}h_{\t\sft{l_{1}o_{\blam}}}^{\sft{l_{2}}}$ if and only if
	$$\frac{\gamma_{\m_{l_{2}o_{\blam}}(\tlam\sft{l_{1}o_{\blam}})}}{\gamma_{\tlam\sft{l_{1}o_{\blam}}}}\frac{\gamma_{\m_{(l_{1}+l_{2})o_{\blam}}(\t)}}{\gamma_{\t}}
	=\frac{\gamma_{\m_{l_{1}o_{\blam}}(\t)}}{\gamma_{\t}}\frac{\gamma_{\m_{l_{2}o_{\blam}}(\t)}\gamma_{\m_{l_{1}o_{\blam}}(\t\sft{l_{2}o_{\blam}})}}{\gamma_{\m_{l_{1}o_{\blam}}(\t)}\gamma_{\t\sft{l_{2}o_{\blam}}}}.$$
	By Definition \ref{2dfns}, it is equivalent to
	$$\frac{\gamma_{\m_{l_{2}o_{\blam}}(\tlam\sft{l_{1}o_{\blam}})}}{\gamma_{\tlam\sft{l_{1}o_{\blam}}}}r_{\t,(l_{1}+l_{2})o_{\blam}}^{-1}=r_{\t,l_{2}o_{\blam}}^{-1}r_{\t\sft{l_{2}o_{\blam}},l_{1}o_{\blam}}^{-1}.$$
	
	Note that $\m_{x}(\tlam)=\tlam, \forall x\in \Z/p\Z$. By the second statement of Proposition \ref{propRstk}, we have that
$R_{\t\tlam,(l_{1}+l_{2})o_{\blam}}=R_{\t\tlam,l_{2}o_{\blam}}R_{\t\sft{l_{2}o_{\blam}}\tlam\sft{l_{2}o_{\blam}},l_{1}o_{\blam}}$ and hence
	$$r_{\t,(l_{1}+l_{2})o_{\blam}}=r_{\t,l_{2}o_{\blam}}r_{\t\sft{l_{2}o_{\blam}},l_{1}o_{\blam}}\frac{\gamma_{\m_{l_{1}o_{\blam}}(\tlam\sft{l_{2}o_{\blam}})}}{\gamma_{\tlam\sft{l_{2}o_{\blam}}}}.$$
	By (\ref{claim1}), we have that
	$$\frac{\gamma_{\m_{l_{1}o_{\blam}}(\tlam\sft{l_{2}o_{\blam}})}}{\gamma_{\tlam\sft{l_{2}o_{\blam}}}}=\frac{\gamma_{\m_{l_{2}o_{\blam}}(\tlam\sft{l_{1}o_{\blam}})}}{\gamma_{\tlam\sft{l_{1}o_{\blam}}}}.$$
This proves the lemma.
\end{proof}

In Definition \ref{sqrtt}, the elements $h_{\t}^{\sft{l}}$ are only defined for certain shapes $\blam\in\P_{r,n}^{\sigma}$. We want to generalize these elements to tableaux of arbitrary shape. The following Lemma makes sure
that our definition will be well-defined.

\begin{lem}
	Suppose that $\blam\in\P_{r,n}^{\sigma}$ and $\t\in\Std(\blam)$ such that $1\in\big(\t^{[1]},\cdots,\t^{[o_{\blam}]}\big)$. Let $l,a\in\Z/p_{\blam}\Z$, and $x\in\Z/p\Z$. Then we have that
\begin{equation}\label{CompatibleEnsure}
h_{\t}^{\sft{l}}\frac{\gamma_{\t}\gamma_{\m_{ao_{\blam}+x}(\t\sft{lo_{\blam}})}}{\gamma_{\m_{ao_{\blam}+x}(\t)}\gamma_{\t\sft{lo_{\blam}}}}=
h_{\t\sft{ao_{\blam}}}^{\sft{l}}\frac{\gamma_{\t\sft{ao_{\blam}}}\gamma_{\m_{x}(\t\sft{(a+l)o_{\blam}})}}{\gamma_{\m_{x}(\t\sft{ao_{\blam}})}\gamma_{\t\sft{(a+l)o_{\blam}}}}.\end{equation}
\end{lem}

\begin{proof}
	By Definition \ref{sqrtt}, we have that
	$$h_{\t\sft{ao_{\blam}}}^{\sft{l}}=h_{\t}^{\sft{l}}\frac{\gamma_{\t}\gamma_{\m_{ao_{\blam}}(\t\sft{lo_{\blam}})}}{\gamma_{\m_{ao_{\blam}}(\t)}\gamma_{\t\sft{lo_{\blam}}}}.$$
	By Definition \ref{2dfns}, the left hand side of (\ref{CompatibleEnsure}) is
	$$h_{\t}^{\sft{l}}\frac{\gamma_{\t}\gamma_{\m_{ao_{\blam}+x}(\t\sft{lo_{\blam}})}}{\gamma_{\m_{ao_{\blam}+x}(\t)}\gamma_{\t\sft{lo_{\blam}}}}=h_{\t}^{\sft{l}}\frac{r_{\t,
ao_{\blam}+x}}{r_{\t\sft{lo_{\blam}},ao_{\blam}+x}}$$
	and the right hand side of (\ref{CompatibleEnsure}) is
$$h_{\t}^{\sft{l}}\frac{\gamma_{\t}\gamma_{\m_{ao_{\blam}}(\t\sft{lo_{\blam}})}}{\gamma_{\m_{ao_{\blam}}(\t)}\gamma_{\t\sft{lo_{\blam}}}}\frac{\gamma_{\t\sft{ao_{\blam}}}
\gamma_{\m_{x}(\t\sft{(a+l)o_{\blam}})}}{\gamma_{\m_{x}(\t\sft{ao_{\blam}})}\gamma_{\t\sft{(a+l)o_{\blam}}}}=
h_{\t}^{\sft{l}}\frac{r_{\t,ao_{\blam}}r_{\t\sft{ao_{\blam}},x}}{r_{\t\sft{lo_{\blam}},ao_{\blam}}r_{\t\sft{(a+l)o_{\blam}},x}}.$$
	
	Note that $\m_{k}(\tlam)=\tlam, \forall\, k\in \Z/p\Z$. By the second statement of Proposition \ref{propRstk}, we have $R_{\t\tlam,ao_{\blam}+x}=R_{\t\tlam,ao_{\blam}}R_{\t\sft{ao_{\blam}}\tlam\sft{ao_{\blam}},x}$.
Applying Lemma \ref{snphit}, we can get that
	
$$r_{\t,ao_{\blam}+x}\frac{\gamma_{\tlam}}{\gamma_{\tlam\sft{ao_{\blam}+x}}}=r_{\t,ao_{\blam}}\frac{\gamma_{\tlam}}{\gamma_{\tlam\sft{ao_{\blam}}}}r_{\t\sft{ao_{\blam}},x}
\frac{\gamma_{\m_{x}(\tlam\sft{ao_{\blam}})}}{\gamma_{\tlam\sft{ao_{\blam}+x}}},$$
	and hence
	$$r_{\t,ao_{\blam}+x}=r_{\t,ao_{\blam}}r_{\t\sft{ao_{\blam}},x}\frac{\gamma_{\m_{x}(\tlam\sft{ao_{\blam}})}}{\gamma_{\tlam\sft{ao_{\blam}}}}.$$
	Similarly, considering $R_{\t\sft{lo_{\blam}}\tlam,ao_{\blam}+x}=R_{\t\sft{lo_{\blam}}\tlam,ao_{\blam}}R_{\t\sft{(l+a)o_{\blam}}\tlam\sft{ao_{\blam}},x}$, we can get that
	$$r_{\t\sft{lo_{\blam}},ao_{\blam}+x}=r_{\t\sft{lo_{\blam}},ao_{\blam}}r_{\t\sft{(l+a)o_{\blam}},x}\frac{\gamma_{\m_{x}(\tlam\sft{ao_{\blam}})}}{\gamma_{\tlam\sft{ao_{\blam}}}}.$$
	Henceforth, the right hand side of the equality is
	$$h_{\t}^{\sft{l}}\frac{r_{\t,ao_{\blam}}r_{\t\sft{ao_{\blam}},x}}{r_{\t\sft{lo_{\blam}},ao_{\blam}}r_{\t\sft{(a+l)o_{\blam}},x}}=h_{\t}^{\sft{l}}\frac{r_{\t, ao_{\blam}+x}}{r_{\t\sft{lo_{\blam}},ao_{\blam}+x}}.$$
This proves the lemma.
\end{proof}

By the aforementioned lemma, we can generalize Definition \ref{sqrtt} a litter bit.

\begin{dfn}\label{sqrttx}
	Suppose that $\blam\in\P_{r,n}$ and $\t\in\Std(\blam)$ such that $1\in\big(\t^{[1]},\cdots,\t^{[o_{\blam}]}\big)$. Let $l\in\Z/p_{\blam}\Z$. For any $x\in\Z/p\Z$, we define
	$$h_{\t,x}^{\sft{l}}:=h_{\t}^{\sft{l}}\frac{\gamma_{\t}\gamma_{\m_{x}(\t\sft{lo_{\blam}})}}{\gamma_{\m_{x}(\t)}\gamma_{\t\sft{lo_{\blam}}}}.$$
\end{dfn}

\begin{rem}
	In the aforementioned Definition, if $x=ko_{\blam}$ for some $0\leq k<p_{\blam}$, the element $h_{\t,x}^{\sft{l}}$ is exactly $h_{\t\sft{ko_{\blam}}}^{\sft{l}}$ defined in Definition \ref{sqrtt}.
\end{rem}

\begin{lem}\label{SquareRoots2}
	Suppose that $\blam\in\P_{r,n}$ and $\t\in\Std(\blam)$ such that $1\in\big(\t^{[1]},\cdots,\t^{[o_{\blam}]}\big)$. Let $l\in\Z/p_{\blam}\Z$. For any $x\in\Z/p\Z$, we have that
	$$(h_{\t,x}^{\sft{l}})^{2}=\frac{\gamma_{\t\sft{x+lo_{\blam}}}}{\gamma_{\t\sft{x}}}.$$
\end{lem}

\begin{proof}
	By Definition \ref{sqrttx}, we have that
$$h_{\t,x}^{\sft{l}}=h_{\t}^{\sft{l}}\frac{\gamma_{\t}\gamma_{\m_{x}(\t\sft{lo_{\blam}})}}{\gamma_{\m_{x}(\t)}\gamma_{\t\sft{lo_{\blam}}}}=h_{\t}^{\sft{l}}\frac{\gamma_{\t}\gamma_{\m_{x}(\t\sft{lo_{\blam}})}}{\gamma_{\m_{x}(\t)}\gamma_{\t\sft{lo_{\blam}+x}}}\frac{\gamma_{\t\sft{lo_{\blam}+x}}}{\gamma_{\t\sft{lo_{\blam}}}}.$$
	By Lemma \ref{snphit} and Lemma \ref{squareht}, we can deduce that
	$$\begin{aligned}
(h_{\t,x}^{\sft{l}})^{2}&=(h_{\t}^{\sft{l}})^{2}\bigg(\frac{\gamma_{\t}\gamma_{\m_{x}(\t\sft{lo_{\blam}})}}{\gamma_{\m_{x}(\t)}\gamma_{\t\sft{lo_{\blam}+x}}}\bigg)^{2}\bigg(\frac{\gamma_{\t\sft{lo_{\blam}+x}}}{\gamma_{\t\sft{lo_{\blam}}}}\bigg)^{2}\\
	&=\frac{\gamma_{\t\sft{lo_{\blam}}}}{\gamma_{\t}}R_{\t\t\sft{lo_{\blam}},x}^{2}\bigg(\frac{\gamma_{\t\sft{lo_{\blam}+x}}}{\gamma_{\t\sft{lo_{\blam}}}}\bigg)^{2}\\
&=\frac{\gamma_{\t\sft{lo_{\blam}}}}{\gamma_{\t}}\frac{\gamma_{\t}\gamma_{\t\sft{lo_{\blam}}}}{\gamma_{\t\sft{x}}\gamma_{\t\sft{lo_{\blam}+x}}}\bigg(\frac{\gamma_{\t\sft{lo_{\blam}+x}}}{\gamma_{\t\sft{lo_{\blam}}}}\bigg)^{2}\\
	&=\frac{\gamma_{\t\sft{lo_{\blam}+x}}}{\gamma_{\t\sft{x}}}.
	\end{aligned}$$
This proves the Lemma.
\end{proof}

\medskip\noindent
{\bf{Proof of Theorem \ref{mainthm2b}:}} This follows from Lemmas \ref{squareht} and \ref{SquareRoots2}.
\qed
\medskip

\begin{lem}\label{htklx}
	Suppose that $\blam\in\P_{r,n}^{\sigma}$ and $\t\in\Std(\blam)$ with $1\in\big(\t^{[1]},\cdots,\t^{[o_{\blam}]}\big)$. Let $l_{1},l_2\in\Z/p_{\blam}\Z$, and $x\in\Z/p\Z$. Then we have
	$$h_{\t,x}^{\sft{l_{1}+l_{2}}}=h_{\t,x}^{\sft{l_{1}}}h_{\t,x+l_{1}o_{\blam}}^{\sft{l_{2}}}.$$
\end{lem}

\begin{proof}
	By Definition \ref{sqrttx}, we have that
	$$h_{\t,x}^{\sft{l_{1}}}=h_{\t}^{\sft{l_{1}}}\frac{\gamma_{\t}\gamma_{\m_{x}(\t\sft{l_{1}o_{\blam}})}}{\gamma_{\m_{x}(\t)}\gamma_{\t\sft{l_{1}o_{\blam}}}},\quad
	h_{\t,x+l_{1}o_{\blam}}^{\sft{l_{2}}}=h_{\t}^{\sft{l_{2}}}\frac{\gamma_{\t}\gamma_{\m_{x+l_{1}o_{\blam}}(\t\sft{l_{2}o_{\blam}})}}{\gamma_{\m_{x+l_{1}o_{\blam}}(\t)}\gamma_{\t\sft{l_{2}o_{\blam}}}}.$$
	Combining Lemma \ref{htkl}, we can get that
	$$
	\begin{aligned}
h_{\t,x}^{\sft{l_{1}+l_{2}}}&=h_{\t}^{\sft{l_{1}+l_{2}}}\frac{\gamma_{\t}\gamma_{\m_{x}(\t\sft{(l_{1}+l_{2})o_{\blam}})}}{\gamma_{\m_{x}(\t)}\gamma_{\t\sft{(l_{1}+l_{2})o_{\blam}}}}=h_{\t}^{\sft{l_{1}}}h_{\t,l_{1}o_{\blam}}^{\sft{l_{2}}}\frac{\gamma_{\t}\gamma_{\m_{x}(\t\sft{(l_{1}+l_{2})o_{\blam}})}}{\gamma_{\m_{x}(\t)}\gamma_{\t\sft{(l_{1}+l_{2})o_{\blam}}}}\\
&=h_{\t}^{\sft{l_{1}}}h_{\t}^{\sft{l_{2}}}\frac{\gamma_{\t}\gamma_{\m_{l_{1}o_{\blam}}(\t\sft{l_{2}o_{\blam}})}}{\gamma_{\m_{l_{1}o_{\blam}}(\t)}\gamma_{\t\sft{l_{2}o_{\blam}}}}\frac{\gamma_{\t}\gamma_{\m_{x}(\t\sft{(l_{1}+l_{2})o_{\blam}})}}{\gamma_{\m_{x}(\t)}\gamma_{\t\sft{(l_{1}+l_{2})o_{\blam}}}}.
	\end{aligned}$$
	By Definition \ref{2dfns}, the aforementioned equalities can be rewritten as
	$$\begin{aligned}
		&h_{\t,x}^{\sft{l_{1}}}=h_{\t}^{\sft{l_{1}}}r_{\t,x}r_{\t\sft{l_{1}o_{\blam}},x}^{-1},\quad
h_{\t,x+l_{1}o_{\blam}}^{\sft{l_{2}}}=h_{\t}^{\sft{l_{2}}}r_{\t,x+l_{1}o_{\blam}}r_{\t\sft{l_{2}o_{\blam}},x+l_{1}o_{\blam}}^{-1},\\
&h_{\t,x}^{\sft{l_{1}+l_{2}}}=h_{\t}^{\sft{l_{1}}}h_{\t}^{\sft{l_{2}}}r_{\t,x}r_{\t,l_{1}o_{\blam}}r_{\t\sft{l_{2}o_{\blam}},l_{1}o_{\blam}}^{-1}r_{\t\sft{(l_{1}+l_{2})o_{\blam}},x}^{-1}.
	\end{aligned}$$
	By computation, we have that $h_{\t,x}^{\sft{l_{1}+l_{2}}}=h_{\t,x}^{\sft{l_{1}}}h_{\t,x+l_{1}o_{\blam}}^{\sft{l_{2}}}$ if and only if
	$$r_{\t,l_{1}o_{\blam}}r_{\t\sft{l_{1}o_{\blam}},x}r_{\t,x+l_{1}o_{\blam}}^{-1}=r_{\t\sft{l_{2}o_{\blam}},l_{1}o_{\blam}}r_{\t\sft{(l_{1}+l_{2})o_{\blam}},x}r_{\t\sft{l_{2}o_{\blam}},x+l_{1}o_{\blam}}^{-1}.$$
	
	By the second statement of Proposition \ref{propRstk}, we have that $$R_{\t\tlam, x+l_{1}o_{\blam}}=R_{\t\tlam,l_{1}o_{\blam}}R_{\t\sft{l_{1}o_{\blam}}\tlam\sft{l_{1}o_{\blam}},x}.$$ Note that $\m_{k}(\tlam)=\tlam,
\forall\,0\leq k\leq p$. Then we can get that
	
$$r_{\t,x+l_{1}o_{\blam}}\frac{\gamma_{\tlam}}{\gamma_{\tlam\sft{x+l_{1}o_{\blam}}}}=r_{\t,l_{1}o_{\blam}}\frac{\gamma_{\tlam}}{\gamma_{\tlam\sft{l_{1}o_{\blam}}}}r_{\t\sft{l_{1}o_{\blam}},x}\frac{\gamma_{\m_{x}(\tlam\sft{l_{1}o_{\blam}})}}{\gamma_{\tlam\sft{x+l_{1}o_{\blam}}}}$$
	and hence
	$$r_{\t,l_{1}o_{\blam}}r_{\t\sft{l_{1}o_{\blam}},x}r_{\t,x+l_{1}o_{\blam}}^{-1}=\frac{\gamma_{\tlam\sft{l_{1}o_{\blam}}}}{\gamma_{\m_{x}(\tlam\sft{l_{1}o_{\blam}})}}.$$
	Similarly, by the second statement of Proposition \ref{propRstk}, we have $$R_{\t\sft{l_{2}o_{\blam}}\tlam,
x+l_{1}o_{\blam}}=R_{\t\sft{l_{2}o_{\blam}}\tlam,l_{1}o_{\blam}}R_{\t\sft{(l_{1}+l_{2})o_{\blam}}\tlam\sft{l_{1}o_{\blam}},x}.$$
	Then we can deduce that
	$$r_{\t\sft{l_{2}o_{\blam}},l_{1}o_{\blam}}r_{\t\sft{(l_{1}+l_{2})o_{\blam}},x}r_{\t\sft{l_{2}o_{\blam}},x+l_{1}o_{\blam}}^{-1}=\frac{\gamma_{\tlam\sft{l_{1}o_{\blam}}}}{\gamma_{\m_{x}(\tlam\sft{l_{1}o_{\blam}})}}.$$
	This completes the proof of the lemma.
\end{proof}

\bigskip

\section{Seminormal bases for $\HH_{r,p,n}$}

In this section, we shall use the main results Theorems \ref{mainthm1}, \ref{mainthm2a} and \ref{mainthm2b} proved in Section 3 to construct seminormal bases for $\HH_{r,p,n}$. In particular, we shall give the proof of the
third and the fourth main results Theorem \ref{mainthm3} and \ref{mainthm4} of this paper.

\begin{dfn}\label{fstk}
	Suppose that (\ref{ssrpn}) holds. Let $\blam\in\P_{r,n}^{\sigma}$ and $\s,\t\in\Std(\blam)$. Let $i,k\in\Z/p_\blam\Z$ and $j\in\Z/p\Z$. We define $$\begin{aligned}
		A_{i,j}^{\s\t}&:=(h_\s^{\<i\>})^{-1}\frac{\gamma_{\t^{\blam}}}{\gamma_{\t^\blam\sft{j}}}\frac{\gamma_{\s\<io_\blam\>}}{\gamma_{\m_j(\s\<io_\blam\>)}}\frac{\gamma_{\t}}{\gamma_{\m_j(\t)}}\in K^\times,\\
		f_{\s\t}^{[k]}&:=\sum_{i\in\Z/p_{\blam}\Z}\sum_{j\in\Z/p\Z}(\varepsilon^{o_{\blam}})^{ki}A_{i,j}^{\s\t}f_{\s\sft{io_{\blam}+j}\t\sft{j}}.\end{aligned} $$
\end{dfn}

\begin{lem}\label{Astij}
	Suppose that (\ref{ssrpn}) holds. Let $\blam\in\P_{r,n}^{\sigma}$ and $\s,\t\in\Std(\blam)$. Let $i\in\Z/p_\blam\Z$ and $j\in\Z/p\Z$.  We have that
	$$
		A_{i,j}^{\s\t}=(h_\s^{\<i\>})^{-1}\frac{\gamma_{\t^{\blam}}}{\gamma_{\t^\blam\sft{j}}}\frac{\gamma_{\s\<io_\blam\>}\gamma_{\t}}{\gamma_{\m_j(\s\<io_\blam\>)}\gamma_{\m_j(\t)}}=
		(h_\s^{\<i\>})^{-1}\frac{\gamma_{\t^{\blam}\sft{j}}}{\gamma_{\t^{\blam}}}\frac{\gamma_{\m_{j}(\s\sft{io_{\blam}})}\gamma_{\m_{j}(\t)}}{\gamma_{\s\sft{io_{\blam}+j}}\gamma_{\t\sft{j}}}=(h_\s^{\<i\>})^{-1}R_{\s\sft{io_{\blam}}\t,j}. $$
If furthermore, $\s^{-1}(1)\in (\blam^{[1]},\cdots,\blam^{[o_\blam]})$, then $$
A_{i,j}^{\s\t}=(h_{\s,j}^{\sft{i}})^{-1}\frac{\gamma_{\t^{\blam}}}{\gamma_{\t^\blam\sft{j}}}\frac{\gamma_{\s}\gamma_{\t}}{\gamma_{\m_{j}(\s)}\gamma_{\m_{j}(\t)}}=
		(h_{\s,j}^{\sft{i}})^{-1}\frac{\gamma_{\t^{\blam}\sft{j}}}{\gamma_{\t^{\blam}}}\frac{\gamma_{\m_{j}(\s)}\gamma_{\m_{j}(\t)}}{\gamma_{\s\sft{j}}\gamma_{\t\sft{j}}}=(h_{\s,j}^{\sft{i}})^{-1}R_{\s\t,j}.
$$
Moreover, $$
	(A_{i,j}^{\s\t})^2=\frac{\gamma_\s\gamma_\t}{\gamma_{\s\<io_\blam+j\>}\gamma_{\t\<j\>}}. $$
\end{lem}

\begin{proof} By Lemma \ref{snphit}, we have that \begin{equation}\label{Aijeqa1} A_{i,j}^{\s\t}=(h_\s^{\<i\>})^{-1}R_{\s\sft{io_{\blam}}\t,j}=(h_\s^{\<i\>})^{-1}\frac{\gamma_{\t^{\blam}\sft{j}}}{\gamma_{\t^{\blam}}}\frac{\gamma_{\m_{j}(\s\sft{io_{\blam}})}
\gamma_{\m_{j}(\t)}}{\gamma_{\s\sft{io_{\blam}+j}}\gamma_{\t\sft{j}}}.\end{equation}
Combining Proposition \ref{propRstk} with (\ref{Aijeqa1}) and Lemma \ref{squareht}, we can deduce that
	$$(A_{i,j}^{\s\t})^2=\frac{\gamma_\s\gamma_\t}{\gamma_{\s\<io_\blam+j\>}\gamma_{\t\<j\>}}.$$

If furthermore, $\s^{-1}(1)\in (\blam^{[1]},\cdots,\blam^{[o_\blam]})$, then by Definition \ref{sqrttx}, we know that
	$$h_{\s,j}^{\sft{i}}:=h_{\s}^{\sft{i}}\cdot \frac{\gamma_{\s}\gamma_{\m_{j}(\s\sft{io_{\blam}})}}{\gamma_{\m_{j}(\s)}\gamma_{\s\sft{io_{\blam}}}}.$$
As a result, we can get that
	$$A_{i,j}^{\s\t}=(h_{\s,j}^{\sft{i}})^{-1}\frac{\gamma_{\t^{\blam}}}{\gamma_{\t^\blam\sft{j}}}\frac{\gamma_{\s}\gamma_{\t}}{\gamma_{\m_{j}(\s)}\gamma_{\m_{j}(\t)}}=(h_{\s,j}^{\sft{i}})^{-1}R_{\s\t,j}=
	(h_{\s,j}^{\sft{i}})^{-1}\frac{\gamma_{\t^{\blam}\sft{j}}}{\gamma_{\t^{\blam}}}\frac{\gamma_{\m_{j}(\s)}\gamma_{\m_{j}(\t)}}{\gamma_{\s\sft{j}}\gamma_{\t\sft{j}}}.$$
This proves the lemma.
\end{proof}

\begin{lem} Suppose that (\ref{ssrpn}) holds. Let $\blam\in\P_{r,n}^{\sigma}$, $\s, \t\in\Std(\blam)$ and $k\in\Z/p_{\blam}\Z$. Then $\sigma(f_{\s\t}^{[k]})=f_{\s\t}^{[k]}$. In particular, $f_{\s\t}^{[k]}\in\HH_{r,p,n}$.
\end{lem}

\begin{proof} For any $i\in\Z/p_\blam\Z$ and $j\in\Z/p\Z$, by Definition \ref{fstk},  $$
A_{i,j}^{\s\t}=(h_\s^{\<i\>})^{-1}\frac{\gamma_{\t^{\blam}}}{\gamma_{\t^\blam\sft{j}}}\frac{\gamma_{\s\<io_\blam\>}}{\gamma_{\m_j(\s\<io_\blam\>)}}\frac{\gamma_{\t}}{\gamma_{\m_j(\t)}}.$$	
	
We claim that $A_{i,j+1}^{\s\t}=A_{i,j}^{\s\t}R_{\s\sft{io_{\blam}+j}\t\sft{j}}$. In fact, by Lemma \ref{snphit}, we have
$$R_{\s\sft{io_{\blam}}\t,j}=\frac{\gamma_{\t^{\blam}}}{\gamma_{\t^{\blam}\sft{j}}}r_{\s\sft{io_{\blam}},j}r_{\t,j}=\frac{\gamma_{\t^{\blam}\sft{j}}}{\gamma_{\t^{\blam}}}\frac{\gamma_{\m_{j}(\s\sft{io_{\blam}})}\gamma_{\m_{j}(\t)}}{\gamma_{\s\sft{io_{\blam}+j}}\gamma_{\t\sft{j}}}.$$
Applying Lemma \ref{Astij}, we get that \begin{equation}\label{eqla0}
A_{i,j}^{\s\t}=(h_\s^{\<i\>})^{-1}R_{\s\sft{io_{\blam}}\t,j}. \end{equation}

Consider the composition $\mu=(j,1)$ of the integer $j+1$. By the second statement of Proposition \ref{propRstk}, we know that
\begin{equation}\label{eqla1}
R_{\s\sft{io_{\blam}}\t,j}R_{\s\sft{io_{\blam}+j}\t\sft{j}}=R_{\s\sft{io_{\blam}}\t,j+1}.\end{equation}
Now our claim follows from the definition of $A_{i,j}^{\s\t}$, (\ref{eqla0}) and (\ref{eqla1}).

Finally, applying Corollary \ref{sigmafst} and the above claim, we can get that
$$
\sigma(A_{i,j}^{\s\t}f_{\s\sft{io_{\blam}+j}\t\sft{j}})=A_{i,j}^{\s\t}R_{\s\sft{io_{\blam}+j}\t\sft{j}}f_{\s\sft{io_{\blam}+j+1}\t\sft{j+1}}=
A_{i,j+1}^{\s\t}f_{\s\sft{io_{\blam}+j+1}\t\sft{j+1}},$$
from which the lemma follows.
\end{proof}

\begin{lem}\label{orth}
	Let $\blam, \bmu\in\P_{r,n}^{\sigma}$, $0\leq k<p_{\blam}$, $0\leq l<p_{\bmu}$. Let $\s, \t\in\Std(\blam)$ and $\u, \v\in\Std(\bmu)$ be standard tableaux such that
$\s^{-1}(1),\t^{-1}(1)\in(\blam^{[1]},\cdots,\blam^{[o_{\blam}]})$ and $\u^{-1}(1),\v^{-1}(1)\in(\bmu^{[1]},\cdots,\bmu^{[o_{\bmu}]})$. Then
	$$f_{\s\t}^{[k]}f_{\u\v}^{[l]}=\delta_{\t, \u}\delta_{k,l}p_{\blam}\gamma_{\t}f_{\s\v}^{[k]}.$$
\end{lem}

\begin{proof} By definition, we have that
	$$	f_{\s\t}^{[k]}f_{\u\v}^{[l]}=\bigg(\sum_{i\in\Z/p_{\blam}\Z}\sum_{j\in\Z/p\Z}(\varepsilon^{o_{\blam}})^{ki}A_{i,j}^{\s\t}f_{\s\sft{io_{\blam}+j}\t\sft{j}}\bigg)
	\cdot\bigg(\sum_{i\in\Z/p_{\bmu}\Z}\sum_{j\in\Z/p\Z}(\varepsilon^{o_{\bmu}})^{li}A_{i,j}^{\u\v}f_{\u\sft{io_{\bmu}+j}\v\sft{j}}\bigg). $$
	If $\t\neq\u$, then our assumption that $\t^{-1}(1)\in(\blam^{[1]},\cdots,\blam^{[o_{\blam}]}), \u^{-1}(1)\in(\bmu^{[1]},\cdots,\bmu^{[o_{\bmu}]})$ implies that $$
	\t\<j\>\neq\u\<io_{\blam}+j\>,\,\,\forall\,i\in\Z/p_\bmu\Z, j\in\Z/p\Z .
	$$
	Thus $f_{\s\sft{io_{\blam}+j}\t\sft{j}}f_{\u\sft{io_{\blam}+j}\v\sft{j}}=0$, and hence $f_{\s\t}^{[k]}f_{\u\v}^{[l]}=0$.
	
	Henceforth, we assume $\t=\u$ and hence $\blam=\bmu$. Then $$
	f_{\s\t}^{[k]}f_{\u\v}^{[l]}=\bigg(\sum_{i\in\Z/p_{\blam}\Z}\sum_{j\in\Z/p\Z}(\varepsilon^{o_{\blam}})^{ki}A_{i,j}^{\s\t}f_{\s\sft{io_{\blam}+j}\t\sft{j}}\bigg)
	\cdot\bigg(\sum_{a\in\Z/p_{\blam}\Z}\sum_{b\in\Z/p\Z}(\varepsilon^{o_{\blam}})^{la}A_{a,b}^{\t\v}f_{\t\sft{ao_{\blam}+b}\v\sft{b}}\bigg).
	$$
	For each $i,a\in\Z/p_{\blam}\Z$ and $j,b\in\Z/p\Z$, we have that
	\begin{equation}\label{Astuv}
		(\varepsilon^{o_{\blam}})^{ki}A_{i,j}^{\s\t}f_{\s\sft{io_{\blam}+j}\t\sft{j}}(\varepsilon^{o_{\blam}})^{la}A_{a,b}^{\t\v}f_{\t\sft{ao_{\blam}+b}\v\sft{b}}
		=(\varepsilon^{o_{\blam}})^{ki+la}\delta_{j,ao_{\blam}+b}A_{i,j}^{\s\t}A_{a,b}^{\t\v}\gamma_{\t\sft{j}}f_{\s\sft{io_{\blam}+j}\v\sft{b}}.
	\end{equation}
	Then (\ref{Astuv}) is nonzero only if $j=ao_{\blam}+b$. In this case, applying Lemma \ref{Astij}, we have that
	$$\begin{aligned}
		A_{i,j}^{\s\t}A_{a,b}^{\t\v}\gamma_{\t\sft{j}}
		&=(h_\s^{\<i\>})^{-1}\frac{\gamma_{\s\<io_\blam\>}\gamma_{\m_j(\t)}}{\gamma_{\m_j(\s\<io_\blam\>)}\gamma_{\t\sft{j}}}
		(h_\t^{\<a\>})^{-1}\frac{\gamma_{\t\<ao_\blam\>}\gamma_{\m_b(\v)}}{\gamma_{\m_b(\t\<ao_\blam\>)}\gamma_{\v\sft{b}}}\gamma_{\t\sft{j}}\\
		&=(h_\s^{\<i\>})^{-1}(h_\t^{\<a\>})^{-1}\frac{\gamma_{\s\<io_\blam\>}\gamma_{\m_j(\t)}}{\gamma_{\m_j(\s\<io_\blam\>)}}
		\frac{\gamma_{\t\<ao_\blam\>}\gamma_{\m_b(\v)}}{\gamma_{\m_b(\t\<ao_\blam\>)}\gamma_{\v\sft{b}}}.
	\end{aligned}$$
	By definition, $h_{\t}^{\sft{a}}=h_{\blam,0,a}\frac{\gamma_{\m_{ao_{\blam}(\t)}}}{\gamma_{\t}}$. It follows that
	$$A_{i,j}^{\s\t}A_{a,b}^{\t\v}\gamma_{\t\sft{j}}=(h_\s^{\<i\>})^{-1}h_{\blam,0,a}^{-1}\frac{\gamma_{\s\<io_\blam\>}\gamma_{\m_b(\v)}}{\gamma_{\m_j(\s\<io_\blam\>)}\gamma_{\v\sft{b}}}
	\frac{\gamma_{\t}}{\gamma_{\m_{ao_{\blam}(\t)}}}\frac{\gamma_{\m_j(\t)}\gamma_{\t\<ao_\blam\>}}{\gamma_{\m_b(\t\<ao_\blam\>)}}.$$
	By Lemma \ref{htkl}, we have that
	$$
	h_{\s}^{\sft{i+a}}=h_{\s}^{\sft{i}}h_{\s\sft{io_{\blam}}}^{\sft{a}}=h_{\s}^{\sft{i}}h_{\s}^{\sft{a}}\frac{\gamma_{\s}\gamma_{\m_{io_{\blam}(\s\sft{ao_{\blam}})}}}{\gamma_{\m_{io_{\blam}(\s)}}
		\gamma_{\s\sft{ao_{\blam}}}}
	=h_{\s}^{\sft{i}}h_{\blam,0,a}\frac{\gamma_{\m_{ao_{\blam}(\s)}}\gamma_{\m_{io_{\blam}(\s\sft{ao_{\blam}})}}}{\gamma_{\m_{io_{\blam}(\s)}}\gamma_{\s\sft{ao_{\blam}}}}.
	$$
	Then we can get that
	
$$A_{i,j}^{\s\t}A_{a,b}^{\t\v}\gamma_{\t\sft{j}}=(h_{\s}^{\sft{i+a}})^{-1}\frac{\gamma_{\m_{ao_{\blam}(\s)}}\gamma_{\m_{io_{\blam}(\s\sft{ao_{\blam}})}}}{\gamma_{\m_{io_{\blam}(\s)}}\gamma_{\s\sft{ao_{\blam}}}}\frac{\gamma_{\s\<io_\blam\>}\gamma_{\m_b(\v)}}{\gamma_{\m_j(\s\<io_\blam\>)}\gamma_{\v\sft{b}}}
	\frac{\gamma_{\t}}{\gamma_{\m_{ao_{\blam}(\t)}}}\frac{\gamma_{\m_j(\t)}\gamma_{\t\<ao_\blam\>}}{\gamma_{\m_b(\t\<ao_\blam\>)}}.$$
	We now claim that \begin{equation}\label{claimf}\frac{\gamma_{\m_{ao_{\blam}(\s)}}\gamma_{\m_{io_{\blam}(\s\sft{ao_{\blam}})}}\gamma_{\s\<io_\blam\>}}{\gamma_{\m_{io_{\blam}(\s)}}
			\gamma_{\s\sft{ao_{\blam}}}\gamma_{\m_j(\s\<io_\blam\>)}}=\frac{\gamma_{\s\sft{(i+a)o_{\blam}}}}{\gamma_{\m_{b}(\s\sft{(i+a)o_{\blam}})}}
		\frac{\gamma_{\m_{b}(\t^{\blam}\sft{ao_{\blam}})}}{\gamma_{\t^{\blam}\sft{ao_{\blam}}}}.\end{equation}
	
	To prove the claim, we first recall that $\m_{k}(\t^{\blam})=\t^{\blam}, \forall\, k\in\Z/p\Z$. By the second statement of Proposition \ref{propRstk}, we have that
	$$R_{\s\t^{\blam},io_{\blam}}R_{\s\sft{io_{\blam}}\t^{\blam}\sft{io_{\blam}},ao_{\blam}+b}=R_{\s\t,(a+i)o_{\blam}+b}=R_{\s\t^{\blam},ao_{\blam}}R_{\s\sft{ao_{\blam}}\t^{\blam}\sft{ao_{\blam}},io_{\blam}}
	R_{\s\sft{(i+a)o_{\blam}}\t^{\blam}\sft{(i+a)o_{\blam}},b}.$$
	Applying Lemma \ref{snphit}, we get that
	$$r_{\s,io_{\blam}}\frac{\gamma_{\t^{\blam}}}{\gamma_{\t^{\blam}\sft{io_{\blam}}}}
	r_{\s\sft{io_{\blam}},ao_{\blam}+b}\frac{\gamma_{\m_{ao_{\blam}+b}(\t^{\blam}\sft{io_{\blam}})}}{\gamma_{\t^{\blam}\sft{(i+a)o_{\blam}+b}}}
	=r_{\s,ao_{\blam}}\frac{\gamma_{\t^{\blam}}}{\gamma_{\t^{\blam}\sft{ao_{\blam}}}}
	r_{\s\sft{ao_{\blam}},io_{\blam}}\frac{\gamma_{\m_{io_{\blam}}(\t^{\blam}\sft{ao_{\blam}})}}{\gamma_{\t^{\blam}\sft{(i+a)o_{\blam}}}}
	r_{\s\sft{(i+a)o_{\blam}},b}\frac{\gamma_{\m_{b}(\t^{\blam}\sft{(i+a)o_{\blam}})}}{\gamma_{\t^{\blam}\sft{(i+a)o_{\blam}+b}}}.$$
	Then we can get that
	
$$r_{\s\sft{(i+a)o_{\blam}},b}=\frac{r_{\s,io_{\blam}}r_{\s\sft{io_{\blam}},ao_{\blam}+b}}{r_{\s,ao_{\blam}}r_{\s\sft{ao_{\blam}},io_{\blam}}}\frac{\gamma_{\m_{ao_{\blam}+b}(\t^{\blam}\sft{io_{\blam}})}\gamma_{\t^{\blam}\sft{ao_{\blam}}}\gamma_{\t^{\blam}\sft{(i+a)o_{\blam}}}}{\gamma_{\t^{\blam}\sft{io_{\blam}}}\gamma_{\m_{io_{\blam}}(\t^{\blam}\sft{ao_{\blam}})}\gamma_{\m_{b}(\t^{\blam}\sft{(i+a)o_{\blam}})}}.$$
	By (\ref{rgamma}), we know that $\frac{\gamma_{\t^{\blam}\sft{ao_{\blam}}}}{\gamma_{\m_{io_{\blam}}(\t^{\blam}\sft{ao_{\blam}})}}=
	\frac{\gamma_{\t^{\blam}\sft{io_{\blam}}}}{\gamma_{\m_{ao_{\blam}}(\t^{\blam}\sft{io_{\blam}})}}$. Then we can get that
	$$\begin{aligned}
		r_{\s\sft{(i+a)o_{\blam}},b}&=\frac{r_{\s,io_{\blam}}r_{\s\sft{io_{\blam}},ao_{\blam}+b}}{r_{\s,ao_{\blam}}r_{\s\sft{ao_{\blam}},io_{\blam}}}
		\frac{\gamma_{\m_{ao_{\blam}+b}(\t^{\blam}\sft{io_{\blam}})}\gamma_{\t^{\blam}\sft{(i+a)o_{\blam}}}}{\gamma_{\m_{ao_{\blam}}(\t^{\blam}\sft{io_{\blam}})}
			\gamma_{\m_{b}(\t^{\blam}\sft{(i+a)o_{\blam}})}}\\
		&=\frac{r_{\s,io_{\blam}}r_{\s\sft{io_{\blam}},ao_{\blam}+b}}{r_{\s,ao_{\blam}}r_{\s\sft{ao_{\blam}},io_{\blam}}}
		\frac{r_{\t^{\blam}\sft{io_{\blam}},ao_{\blam}+b}}{r_{\t^{\blam}\sft{io_{\blam}},ao_{\blam}}r_{\t^{\blam}\sft{(i+a)o_{\blam}},b}}.
	\end{aligned}
	$$
	By Definition \ref{2dfns}, we have that
	
$$\frac{\gamma_{\m_{ao_{\blam}+b}(\t^{\blam}\sft{io_{\blam}})}\gamma_{\t^{\blam}\sft{(i+a)o_{\blam}}}}{\gamma_{\m_{ao_{\blam}}(\t^{\blam}\sft{io_{\blam}})}\gamma_{\m_{b}(\t^{\blam}\sft{(i+a)o_{\blam}})}}=\frac{r_{\t^{\blam}\sft{io_{\blam}},ao_{\blam}}r_{\t^{\blam}\sft{(i+a)o_{\blam}},b}}{r_{\t^{\blam}\sft{io_{\blam}},ao_{\blam}+b}}.$$
	Similarly, by the second statement of Proposition \ref{propRstk}, we have that
	$$R_{\t^{\blam}\sft{io_{\blam}}\t^{\blam},ao_{\blam}+b}=R_{\t^{\blam}\sft{io_{\blam}}\t^{\blam},ao_{\blam}}R_{\t^{\blam}\sft{(i+a)o_{\blam}}\t^{\blam}\sft{ao_{\blam}},b}.$$
	Applying Lemma \ref{snphit}, we can deduce that
	
$$\frac{r_{\t^{\blam}\sft{io_{\blam}},ao_{\blam}}r_{\t^{\blam}\sft{(i+a)o_{\blam}},b}}{r_{\t^{\blam}\sft{io_{\blam}},ao_{\blam}+b}}=\frac{\gamma_{\t^{\blam}\sft{ao_{\blam}}}}{\gamma_{\m_{b}(\t^{\blam}\sft{ao_{\blam}})}}.$$
	
	Using Definition \ref{2dfns} again, it is easy to check that
	$$\frac{\gamma_{\m_{ao_{\blam}(\s)}}\gamma_{\m_{io_{\blam}(\s\sft{ao_{\blam}})}}\gamma_{\s\<io_\blam\>}}{\gamma_{\m_{io_{\blam}(\s)}}\gamma_{\s\sft{ao_{\blam}}}\gamma_{\m_j(\s\<io_\blam\>)}}
	=\frac{r_{\s,io_{\blam}}r_{\s\sft{io_{\blam}},ao_{\blam}+b}}{r_{\s,ao_{\blam}}r_{\s\sft{ao_{\blam}},io_{\blam}}},\quad
\frac{\gamma_{\s\sft{(i+a)o_{\blam}}}}{\gamma_{\m_{b}(\s\sft{(i+a)o_{\blam}})}}=r_{\s\sft{(i+a)o_{\blam}},b}.$$
	Then we can get that
	$$\frac{\gamma_{\m_{ao_{\blam}(\s)}}\gamma_{\m_{io_{\blam}(\s\sft{ao_{\blam}})}}\gamma_{\s\<io_\blam\>}}{\gamma_{\m_{io_{\blam}(\s)}}\gamma_{\s\sft{ao_{\blam}}}\gamma_{\m_j(\s\<io_\blam\>)}}=
	\frac{\gamma_{\s\sft{(i+a)o_{\blam}}}}{\gamma_{\m_{b}(\s\sft{(i+a)o_{\blam}})}}\frac{\gamma_{\m_{b}(\t^{\blam}\sft{ao_{\blam}})}}{\gamma_{\t^{\blam}\sft{ao_{\blam}}}}. $$
	This proves our claim (\ref{claimf}).
	
	As a result, we have
	$$
	A_{i,j}^{\s\t}A_{a,b}^{\u\v}\gamma_{\t\sft{j}}=(h_{\s}^{\sft{i+a}})^{-1}\frac{\gamma_{\s\<(i+a)o_\blam\>}\gamma_{\m_b(\v)}}{\gamma_{\m_b(\s\<(i+a)o_\blam\>)}\gamma_{\v\sft{b}}}
	\times\gamma_{\t}\frac{\gamma_{\m_{b}(\t^{\blam}\sft{ao_{\blam}})}}{\gamma_{\t^{\blam}\sft{ao_{\blam}}}}\frac{\gamma_{\m_{j}(\t)}\gamma_{\t\<ao_\blam\>}}{\gamma_{\m_{ao_{\blam}(\t)}}\gamma_{\m_b(\t\<ao_\blam\>)}}.
	$$
	
	Note that we have assumed that $j=ao_\blam+b$. Now we will compute the quotient $\frac{\gamma_{\m_{ao_{\blam}+b}(\t)}\gamma_{\t\<ao_\blam\>}}{\gamma_{\m_{ao_{\blam}(\t)}}\gamma_{\m_b(\t\<ao_\blam\>)}}$. By Definition
\ref{2dfns}, we have that
	$$\frac{\gamma_{\m_{ao_{\blam}+b}(\t)}\gamma_{\t\<ao_\blam\>}}{\gamma_{\m_{ao_{\blam}(\t)}}\gamma_{\m_b(\t\<ao_\blam\>)}}=\frac{r_{\t,ao_{\blam}}r_{\t\sft{ao_{\blam}},b}}{r_{\t,ao_{\blam}+b}}.$$
	Applying the second statement of Proposition \ref{propRstk}, we can get that
	$R_{\t\t^{\blam},ao_{\blam}+b}=R_{\t\t^{\blam},ao_{\blam}}R_{\t\sft{ao_{\blam}}\t^{\blam}\sft{ao_{\blam}},b}$.
	Applying Lemma \ref{snphit}, we can deduce that
	$$\frac{r_{\t,ao_{\blam}}r_{\t\sft{ao_{\blam}},b}}{r_{\t,ao_{\blam}+b}}=\frac{\gamma_{\t^{\blam}\sft{ao_{\blam}}}}{\gamma_{\m_{b}(\t^{\blam}\sft{ao_{\blam}})}}.$$
	
	In conclusion, if $\t=\u$ and $j=ao_{\blam}+b$, we can deduce that
	$$A_{i,j}^{\s\t}A_{a,b}^{\u\v}\gamma_{\t\sft{j}}=\gamma_{\t}(h_{\s}^{\sft{i+a}})^{-1}\frac{\gamma_{\s\<(i+a)o_\blam\>}\gamma_{\m_b(\v)}}{\gamma_{\m_b(\s\<(i+a)o_\blam\>)}\gamma_{\v\sft{b}}}=\gamma_{\t}A_{i+a,b}^{s,v}.$$
	
	In the product $f_{\s\t}^{[k]}f_{\u\v}^{[l]}$, each $A_{i,j}^{\s\t}f_{\s\sft{io_{\blam}+j}\t\sft{j}}$ appears $p_{\blam}$ times. Then we can deduce that
	$$\begin{aligned}
		f_{\s\t}^{[k]}f_{\t\v}^{[l]}
		&=\bigg(\sum_{i\in\Z/p_{\blam}Z}\sum_{j\in\Z/p\Z}(\varepsilon^{o_{\blam}})^{ki}A_{i,j}^{\s\t}f_{\s\sft{io_{\blam}+j}\t\sft{j}}\bigg)\\
		&\qquad\times\bigg(\sum_{a\in\Z/p_{\blam}\Z}\sum_{b\in\Z/p\Z}(\varepsilon^{o_{\blam}})^{la}A_{a,b}^{\u\v}f_{\u\sft{ao_{\blam}+b}\v\sft{b}}\bigg)\\
		&=\gamma_{\t}\sum_{i,a\in\Z/p_{\blam}\Z}\sum_{b\in\Z/p\Z}(\varepsilon^{o_{\blam}})^{ki+la}A_{i+a,b}^{s,v}f_{\s\sft{(i+a)o_{\blam}+b}\v\sft{b}}\\
		&=\gamma_{\t}\sum_{i\in\Z/p_{\blam}\Z}\sum_{b\in\Z/p\Z}\bigg(\sum_{m=0}^{p_{\blam}-1}(\varepsilon^{o_{\blam}})^{k(p_{\blam}-m)+l(i+m)}\bigg)A_{i,j}^{s,v}f_{\s\sft{io_{\blam}+j}\v\sft{j}}.
	\end{aligned}$$
	
	If $l=k$, the coefficient
	$$\sum_{m=0}^{p_{\blam}-1}(\varepsilon^{o_{\blam}})^{k(p_{\blam}-m)+l(i+m)}=p_{\blam}(\varepsilon^{o_{\blam}})^{li}.$$
	In particular, if $o_{\blam}=o_{\bmu}=p$, that is $p_{\blam}=1$, then $k=l=0$ and hence the Lemma holds. Then we assume that $p_{\blam}>1$.
	
	If $l\neq k$, the coefficient
	$$\sum_{m=0}^{p_{\blam}-1}(\varepsilon^{o_{\blam}})^{k(p_{\blam}-m)+l(i+m)}=(\varepsilon^{o_{\blam}})^{li}\cdot \sum_{m=0}^{p_{\blam}-1}(\varepsilon^{o_{\blam}})^{m(l-k)}.$$
	
	Since $\varepsilon$ is a primitive $p$th root of unity and $p=p_{\blam}o_{\blam}$, $\varepsilon^{o_{\blam}}$ is a primitive $p_{\blam}$th root of unity. Then $(\varepsilon^{o_{\blam}})^{l-k}\neq 1$ is also a
$p_{\blam}$th root of unity because of $0\leq l\neq k\leq p_{\blam}-1$. Hence, $\sum_{m=0}^{p_{\blam}-1}(\varepsilon^{o_{\blam}})^{m(l-k)}=0$.
	This proves the lemma.
\end{proof}

Recall that $\P_{r,n}^{\sigma}$ denotes the set of $\sim_{\sigma}$-equivalence classes in $\P_{r,n}$. Let $[\blam]\in\P_{r,n}^{\sigma}$, where $\blam$ is a multipartition. Define the subsets
$$
\mathscr{B}_{\blam}:=\bigl\{f_{\s\t}^{[k]}\bigm|\text{$\s,\t\in\Std(\blam), \s^{-1}(1),\t^{-1}(1)\in(\blam^{[1]},\cdots,\blam^{[o_{\blam}]}), 0\leq k<p_{\blam}$}\bigr\},$$
and
$$\mathscr{B}_{r,p,n}:=\bigsqcup_{[\blam]\in\P_{r,n}^{\sigma}}\mathscr{B}_{\blam}.$$
Note that $p_\blam |\#\Std(\blam)$. Applying Theorem \ref{Classification}, we can deduce that $$
\#\mathscr{B}_{\blam}=p_{\blam}\biggl(\frac{\#\Std(\blam)}{p_{\blam}}\biggr)^{2}=\sum_{j=1}^{p_{\blam}}(\dim S_j^{\blam})^{2}.$$
Then we have that
$$\begin{aligned}
\#\mathscr{B}_{r,p,n}&=\sum_{[\blam]\in\P_{r,n}^{\sigma}}\sum_{j=1}^{p_{\blam}}\bigl(\dim S_j^{\blam}\bigr)^{2}
=\sum_{[\blam]\in\P_{r,n}^{\sigma}}\bigl(\dim S(\blam)\bigr)^2/p_\blam\\
&=\dim\HH_{r,p,n}=\frac{1}{p}\dim\HH_{r,n}=\frac{1}{p}\sum_{\blam\in\P_{r,n}}\bigl(\dim S(\blam)\bigr)^2.
\end{aligned}
$$

\begin{lem}\label{linInde}
	The elements in $\mathscr{B}_{r,p,n}:=\bigsqcup_{\blam\in\P_{r,n}^{\sigma}}\mathscr{B}_{\blam}\subset \HH_{r,p,n}$ are $K$-linearly independent.
\end{lem}

\begin{proof} Since $p1_K\in K^\times$, we can deduce that $p_\blam1_K\in K^\times$ for each $\blam$ because $p_\blam|p$. The lemma follows directly from Lemma \ref{orth} because the elements in $\mathscr{B}_{r,p,n}$ are
pairwise orthogonal.
\end{proof}

\medskip
\noindent
{\textbf{Proof of Theorems \ref{mainthm3}:} The second part of the theorem follows from Lemma \ref{orth}. Note that Lemma \ref{linInde} implies that the elements in $\mathscr{B}_{r,p,n}$ are $K$-linearly independent.
It remains to show that $\#\mathscr{B}_{r,p,n}=\dim_K\HH_{r,p,n}$.

Using Theorem \ref{Classification}, it is easy to see that for each $\blam\in\P_{r,n}^{\sigma}$, $$
\#\bigl\{\t\in\Std(\blam)\bigm|\t^{-1}(1)\in(\blam^{[1]},\cdots,\blam^{[o_{\blam}]})\bigr\}=\dim_K S^\blam/p_\blam .
$$
Combining this with the Wedderburn Decomposition Theorem for $\HH_{r,p,n}$ and Theorem \ref{Classification}, we can deduce that $\#\mathscr{B}_{r,p,n}=\dim_K\HH_{r,p,n}$. Hence $\mathscr{B}_{r,p,n}$ must be a $K$-basis of
$\HH_{r,p,n}$.
This proves the first part of Theorem \ref{mainthm3}.
\hfill\qed
\medskip

\medskip
\noindent
{\textbf{Proof of Theorems \ref{mainthm4}:} This follows from Theorems \ref{mainthm3} and Lemma \ref{orth}.
\hfill\qed
\medskip

\begin{dfn} Suppose that (\ref{ssrpn}) holds. We call the set $\mathscr{B}_{r,p,n}$ a seminormal basis of $\HH_{r,p,n}$.
\end{dfn}

\begin{cor}\label{maincor}
Suppose that (\ref{ssrpn}) holds.  Let $\blam\in \P_{r,n}^{\sigma}$ and $\t\in\Std(\blam)$. For any $0\leq k<p_{\blam}$, we define
$$F_{\t}^{[k]}:=f_{\t\t}^{[k]}/(p_{\blam}\gamma_{\t}).$$
Set $F_{\blam}^{[k]}:=\sum\limits_{\substack{\t\in\Std(\blam)\\ \t^{-1}(1)\in(\blam^{[1]},\cdots,\blam^{[o_{\blam}]})}}F_{\t}^{[k]}$. Then
$\bigl\{F_{\blam}^{[k]}\bigm|\blam\in \P_{r,n}^{\sigma}\bigr\}$ gives a complete set of central primitive idempotents of $\HH_{r,p,n}$.
\end{cor}

\bigskip\bigskip
\section{The $\sigma$-twisted $k$-centers of $\HH_{r,n}$}

In this section we shall study the center for the cyclotomic Hecke algebra $\HH_{r,p,n}$ in the general case. We shall give the proof of our third main result Theorem \ref{mainthm3} of this paper. Throughout this section,
unless otherwise stated, we {\it{do not}} assume (\ref{ssrpn}) holds. In other words, the cyclotomic Hecke algebra $\HH_{r,p,n}$ may be non-semisimple over $R$ even when $R$ is a field.

Geck and Pfeiffer \cite{GP1} have determined integral bases for the Iwahori-Hecke algebras associated to any finite Weyl groups. In particular, they showed that the dimensions of the centers of those Iwahori-Hecke algebras
are equal to
the number of conjugacy classes of the corresponding finite Weyl groups, which are independent of the characteristic of the ground field. In \cite{HS} the first author of this paper and Lei Shi have generalized this
statement for the
dimension of the center to the cyclotomic Hecke algebra of type $G(r,1,n)$. Therefore, it is natural to ask whether this is true for the center $Z(\HH_{r,p,n})$ of the cyclotomic Hecke algebra $\HH_{r,p,n}$ of type
$G(r,p,n)$. In this
section, we shall give an attempt to answer this question.

Let $R$ be a commutative domain. Let $H$ be an associative symmetric $R$-algebra with a symmetrizing form $\Tr: H\rightarrow R$. We assume that $H$ is finitely generated and free over $R$ and the algebra $H$ equipped with
an automorphism $\tau$. Let $p\in\Z^{\geq 1}$ and assume that $R$ contains a primitive $p$-th root of unity $\veps$. Suppose that $\tau^p$ is an inner automorphism of $H$. That is, there is an invertible element $h_0\in H$
such that $$
\tau^p(h)=h_0^{-1}hh_0,\,\,\,\forall\,h\in H .
$$
By definition, $\tau$ restricts to an automorphism of the center $Z(H)$ of $H$. Moreover, $\tau^p$ restricts to the identity map on the center $Z(H)$ of $H$.

\begin{lem}\label{zHdec1} With the notations as above, we have that $$
Z(H)=\bigoplus_{k\in\Z/p\Z}Z(H)_k,\,\,\,\,Z(H)_k:=\bigl\{z\in Z(H)\bigm|\tau(z)=\veps^k z\bigr\},\,\forall\,k\in\Z/p\Z .
$$
\end{lem}

Suppose that the algebra $H$ is equipped with an automorphism $\sigma$ of order $p$.

\begin{dfn}\label{twist1} Let $k\in\Z/p\Z$. For any $g,h\in H$, we define the $\sigma$-twisted $k$-commutator of $g,h$ to be: $$
[g,h]_k:=gh-h\sigma^k(g). $$
Let $[H,H]^{(k)}$ be the $R$-submodule generated by all the $\sigma$-twisted $k$-commutators $[g,h]_k, g,h\in H$. We define $$
Z(H)^{(k)}:=\bigl\{z\in H\bigm|zh=\sigma^k(h)z,\,\forall\,h\in H\bigr\} .
$$
We call the subspace $Z(H)^{(k)}$ the $\sigma$-twisted $k$-center of the algebra $H$.
\end{dfn}

If $k=0$ then $[H,H]^{(k)}$ is nothing but the usual commutator submodule of $H$, while $Z(H)^{(0)}$ is nothing but the usual center $Z(H)$ of $H$. Since $H$ is a symmetric algebra over $R$. By \cite{GP}, the symmetrizing
form $\Tr$ induces a natural $R$-module isomorphism: $\varphi: H\cong H^*:=\Hom_R(H,R), h\mapsto\varphi(h): g\mapsto\Tr(gh),\,\forall\,g\in H$.
Moreover, $\varphi$ induces an $R$-module isomorphism: \begin{equation}\label{isocoCenter1}
Z(H)\cong\biggl(\frac{H}{[H,H]}\biggr)^* .
\end{equation}

\begin{lem}\label{dual2} Let $k\in\Z/p\Z$ and $h\in H$. Then $\varphi(h)\in\bigl(H/[H,H]^{(k)}\bigr)^{*}$ if and only if $h\in Z(H)^{(k)}$. In particular, there is an $R$-module isomorphism:
\begin{equation}\label{isocoCenter2}
Z(H)^{(k)}\cong\biggl(\frac{H}{[H,H]^{(k)}}\biggr)^* .
\end{equation}
\end{lem}

\begin{proof} For any $z\in H$, we can get that
	$$\begin{aligned}
z\in Z(H)^{(k)} &\Leftrightarrow zh=\sigma^k(h)z, \forall\,h\in H, \\
		&\Leftrightarrow \Tr((zh-\sigma^k(h)z)h')=0, \forall\,h, h'\in H \\
		&\Leftrightarrow\Tr(z(hh'-h'\sigma^k(h)))=0, \forall\,h, h'\in H\Leftrightarrow \varphi(z)\in\bigg(\frac{H}{[H,H]^{(k)}}\bigg)^{*}.
	\end{aligned}$$
	This proves the lemma.
\end{proof}

Now let us return to the cyclotomic Hecke algebras $\HH_{r,n}$ and $\HH_{r,p,n}$. Recall that the algebra $\HH_{r,n}$ is equipped with an automorphisms $\sigma$ of order $p$. By definition, $$
\sigma(T_0)=\veps T_0,\,\,\,\sigma(T_j)=T_j,\,\,\forall\,1\leq j<n .
$$
The subalgebra $\HH_{r,r,n}$ coincides with the set of $\sigma$-fixed points in $\HH_{r,n}$. The inner automorphism $\tau$ induced by $T_0$ (which maps $u$ to $T_0^{-1} uT_0$) restricts to an automorphism $\tau$ of
$\HH_{r,p,n}$ such that $\tau^p$ is an inner automorphism of $\HH_{r,p,n}$.

By definition, for any $k\in\Z/p\Z$, we have $$
Z(\HH_{r,n})^{(k)}=\bigl\{z\in\HH_{r,n}\bigm|zT_{0}=\varepsilon^{k}T_{0}z, zT_{i}=T_{i}z, \forall\, 1\leq i<n\bigr\}.
$$
In particular, $Z(\HH_{r,n})^{(k)}=Z(\HH_{r,n})$. We are mostly interested in the center $Z(\HH_{r,p,n})$ of $\HH_{r,p,n}$. Recall that the automorphism $\tau$ stabilizes the center $Z(\HH_{r,p,n})$ and $\tau^p$ acts as the
identity map on the center $Z(\HH_{r,p,n})$.

\begin{dfn} For each $k\in\Z/p\Z$, we define $$
Z(\HH_{r,p,n})_k:=\bigl\{z\in Z(\HH_{r,p,n})\bigm|\tau(z)=\veps^k z\bigr\}=\bigl\{z\in Z(\HH_{r,p,n})\bigm|zT_0=\veps^k T_0z\bigr\}.
$$
\end{dfn}

By Lemma \ref{zHdec1}, we have \begin{equation}\label{Zdecompositi0}
Z(\HH_{r,p,n})=\bigoplus_{k\in\Z/p\Z}Z(\HH_{r,p,n})_k .
\end{equation}

\begin{dfn} For each $k\in\Z/p\Z$, we define the $\tau$-twisted $k$-center $Z^{\tau,k}(\HH_{r,p,n})$ of $\HH_{r,p,n}$ as $$
Z^{\tau,k}(\HH_{r,p,n}):=\bigl\{ z\in\HH_{r,p,n}\bigm|\tau^{p}(z)=z, zh=\tau^{k}(h)z, \forall\,h\in \HH_{r,p,n}\bigr\}. $$\end{dfn}

In particular, if $k=0$ then $Z^{\tau,0}(\HH_{r,p,n})=Z(\HH_{r,p,n})$. For each $k\in\Z/p\Z$, it is easy to check that $\tau\bigl(Z^{\tau,k}(\HH_{r,p,n})\bigr)\subseteq Z^{\tau,k}(\HH_{r,p,n})$ and $\tau^p$ acts as the
identity map on
$Z^{\tau,k}(\HH_{r,p,n})$. We have the following decomposition:  \begin{equation}\label{Zdecomp1}
Z^{\tau,k}(\HH_{r,p,n})=\bigoplus_{l\in\Z/p\Z}Z^{\tau,k}(\HH_{r,p,n})_{l}, \end{equation}
where $$
Z^{\tau,k}(\HH_{r,p,n})_{l}:=\bigl\{ z\in Z^{\tau,k}(\HH_{r,p,n})\bigm|\tau(z)=\veps^l z\bigr\}=\bigl\{ z\in Z^{\tau,k}(\HH_{r,p,n})\bigm| zT_{0}=\varepsilon^{l}T_{0}z\bigr\},\,\,\,l\in\Z/p\Z,
$$
are $R$-submodules of $Z^{\tau,k}(\HH_{r,p,n})$. In particular, $Z^{\tau,0}(\HH_{r,p,n})_{l}=Z(\HH_{r,p,n})_l$.

\begin{lem}\label{weightdecomp} For each $k\in\Z/p\Z$, there is a $K$-subspaces decomposition: $$
Z(\HH_{r,n})^{(k)}=\bigoplus_{l\in\Z/p\Z}Z^{\tau,l}(\HH_{r,p,n})_{k}T_{0}^{l}. $$
\end{lem}

\begin{proof} Note that $\HH_{r,n}$ is a free left $\HH_{r,p,n}$-module with basis $\{T_0^j|j\in\Z/p\Z\}$. Thus as a left $\HH_{r,p,n}$-module, $\HH_{r,n}=\bigoplus_{k=0}^{r-1}\HH_{r,p,n}T_{0}^{k}$. It straightforward to
check
that
	$$Z(\HH_{r,n})^{(k)}=\bigoplus_{l=0}^{r-1}\big(Z(\HH_{r,n})^{(k)}\cap \HH_{r,p,n}T_{0}^{l}\big).$$
	
For each $l\in\Z/p\Z$, we can get that
	$$\begin{aligned}
		Z(\HH_{r,n})^{(k)}\cap \HH_{r,p,n}T_{0}^{l}	&=\bigl\{uT_{0}^{l}\bigm| u\in\HH_{r,p,n}, (uT_{0}^{l})h=\sigma^{k}(h)(uT_{0}^{l}),\,\forall\,h\in\HH_{r,n}\bigr\}\\
		&=\bigl\{uT_{0}^{l}\bigm| u\in\HH_{r,p,n}, uT_{0}=\veps^{k}T_{0}u, uT_{i}=\tau^{k}(T_{i})u, \forall\,1\leq i<n\bigr\}\\
		&=Z^{\tau,l}(\HH_{r,p,n})_{k}T_{0}^{l}.
	\end{aligned}$$
This proves the lemma.
\end{proof}

Suppose that $R=K$ is a field and (\ref{ssrpn}) holds. Then the cyclotomic Hecke algebras $\HH_{r,n}$ and $\HH_{r,p,n}$ are both semisimple over $K$. We now give the proof of Theorem \ref{mainthm5}, which explicitly
describes the $\sigma$-twisted $k$-center of $\HH_{r,n}$.

\medskip
\noindent
{\textbf{Proof of Theorems \ref{mainthm5}:} Without loss of generality, we assume that $k$ is an integer with $0\leq k<p$. For any $z\in Z(\HH_{r,n})^{(k)}$, by Definition \ref{twist1}, we have that $$
zL_{m}=\varepsilon^{k}L_{m}z,\quad \forall\, 1\leq m\leq n. $$
We can write $z=\sum_{\u,\v}r_{\u\v}f_{\u\v}$ as a $K$-linear combination of seminormal bases, where $r_{\u\v}\in K$ for each pair $(\u,\v)$. Then we can get that
$$
L_{m}z=\sum_{\u,\v}r_{\u\v}\res_{\u}(m)f_{\u\v}=\sum_{\u,\v} r_{\u\v}\varepsilon^{-k}\res_{\v}(m)f_{\u\v}=\varepsilon^{-k}zL_{m}, \quad \forall\,1\leq m\leq n.$$
It follows that $ r_{\u\v}\neq 0$ only if $\res_{\u}(m)=\varepsilon^{-k}\res_{\v}(m), \forall\, 1\leq m\leq n$. Applying Lemma \ref{dist} we can deduce that $r_{\u\v}\neq 0$ only if $\u=\v\sft{k}$. Hence we can rewrite
$z=\sum_{\v}r_{\v}f_{\v\sft{k}\v}$, where $r_\v\in K$ for each $\v$.
	
Since $z\in Z(\HH_{r,n})^{(k)}$ and $\sigma(T_i)=T_i$ for any $1\leq i<n$, we have that
	\begin{equation}\label{Tiz}
		T_{i}z=T_{i}\big(\sum_{\v} r_{\v}f_{\v\sft{k}\v}\big)=\big(\sum_{\v} r_{\v}f_{\v\sft{k}\v}\big)T_{i}=zT_{i}, \quad \forall\, 1\leq i< n.
	\end{equation}
Let $1\leq i<n$ and $\v$ be a standard tableau such that $r_\v\neq 0$. If the tableau $\v s_{i}$ is not standard, then either $i,i+1$ both appear in the same row of $\v$ (and $\v\<k\>$ for any $k$), or $i,i+1$ both appear
in the same column of $\v$ (and $\v\<k\>$ for any $k$), in this case, by Proposition \ref{tiact}, both $T_{i}(r_{\v}f_{\v\sft{k}\v})$ and $(r_{\v}f_{\v\sft{k}\v})T_{i}$ are equal to the same scalar multiple of
$f_{\v\sft{k}\v}$.
	
Now assume that the tableau $\t:=\v s_{i}$ is standard. From (\ref{Tiz}), we can deduce that \begin{equation}\label{2TilR}
T_{i}\big(r_{\v}f_{\v\sft{k}\v}+r_{\t}f_{\t\sft{k}\t}\big)=\big(r_{\v}f_{\v\sft{k}\v}+r_{\t}f_{\t\sft{k}\t}\big)T_{i}.\end{equation}
	By computation, we can get that
	\begin{equation}\label{rvAB}
		\begin{aligned}
		&\quad\, \, r_{\v}\bigg( A_{i}(\v\sft{k})f_{\v\sft{k}\v}+B_{i}(\v\sft{k})f_{\t\sft{k}\v}\bigg)
		+r_{\t}\bigg( A_{i}(\t\sft{k})f_{\t\sft{k}\t}+B_{i}(\t\sft{k})f_{\v\sft{k}\t}\bigg)\\
		&=r_{\v}\bigg( A_{i}(\v)f_{\v\sft{k}\v}+B_{i}(\v)f_{\v\sft{k}\t}\bigg)
		+r_{\t}\bigg( A_{i}(\t)f_{\t\sft{k}\t}+B_{i}(\t)f_{\t\sft{k}\v}\bigg),
	\end{aligned}
	\end{equation}
where the coefficients $A_{i}(\v),B_{i}(\v)$ are defined in Proposition \ref{tiact}. Using the fact $\res(\v\sft{k})(m)=\varepsilon^{-k}\res_{\v}(m), \forall\, 1\leq m\leq n$, it is easy to check that
\begin{equation}\label{ALR}
A_{i}(\v\sft{k})=A_{i}(\v).\end{equation}
	
	Comparing the coefficients of (\ref{rvAB}), we can deduce that $r_\t\neq 0$ and \begin{equation}\label{rLR}
r_{\t}=r_{\v}\cdot \frac{B_{i}(\v\sft{k})}{B_{i}(\t)}=r_{\v}\cdot \frac{B_{i}(\v)}{B_{i}(\t\sft{k})}\in K^\times,\quad A_{i}(\t\sft{k})=A_{i}(\t)\in K^\times. \end{equation}
Note that (\ref{2TilR}) is actually equivalent to (\ref{ALR}) and (\ref{rLR}) in this case..
	
	Let $\Shape(\v)=\blam\in\P_{r,n}$ and $a_{k}:=\sum_{i=1}^{k}|\blam^{(i)}|$. Then $\blam\<k\>=\blam$ and hence $o_{\blam}$ divides $k$. We can write $k=lo_{\blam}$, where $0\leq l\leq k$. Let $d(\v)\in\Sym_{n}$ such that
$\t^{\blam}d(\v)=\v$. Then we have that $d(\v)=xd$ and $\ell(d(\v))=\ell(x)+\ell(d)$, where $x\in \Sym_{(a_{k},n-a_{k})}$ and $d\in\mathcal{D}_{(a_{k},n-a_{k}})$. We fix reduced expressions $x=s_{i_{1}}\cdots s_{i_{m}}$ and
$d=s_{i_{m+1}}\cdots s_{i_{N}}$. Then $d(\v)=s_{i_{1}}\cdots s_{i_{m}}s_{i_{m+1}}\cdots s_{i_{N}}$ is also a reduced expression. Set $\t_{j}:=\t^{\blam}s_{i_{1}}\cdots s_{i_{j}}, 0\leq j\leq N$. In particular,
$\t_{0}=\t^{\blam}$ and $\t_{N}=\v$. There is a sequence of standard $\blam$-tableaux
	$$\t^{\blam}=\t_{0}\triangleright \t_{1}\triangleright \cdots\triangleright \t_{m}\triangleright \t_{m+1}\triangleright\cdots\triangleright\t_{N}=\v.$$
	
	Applying $\sft{k}$ to the above sequence, we can get a sequence of standard $\blam\sft{k}$-tableaux
	$$\t^{\blam}\sft{k}=\t_{0}\sft{k}\triangleright \t_{1}\sft{k}\triangleright\cdots\triangleright \t_{m}\sft{k}\triangleleft \t_{m+1}\sft{k}\triangleleft\cdots\triangleleft\t_{N}\sft{k}=\v\sft{k}.$$
Applying (\ref{rLR}), we get that
	$$r_{\v}=r_{\t^{\blam}}\cdot \frac{B_{i_{1}}(\t_{0}\sft{k})}{B_{i_{1}}(\t_{1})}\cdots \frac{B_{i_{m}}(\t_{m-1}\sft{k})}{B_{i_{m}}(\t_{m})}\cdot \frac{B_{i_{m+1}}(\t_{m}\sft{k})}{B_{i_{m+1}}(\t_{m+1})}\cdots
\frac{B_{i_{N}}(\t_{N-1}\sft{k})}{B_{i_{N}}(\t_{N})}.$$

Using the definition of $B_{i}(\s)$ in Proposition \ref{tiact}, it is easy to see that $B_{i_{j}}(\t_{j-1}\sft{k})=1$ for any $1\leq j\leq m$. Now let $m<j\leq N$, then we have $\t_{j-1}\<k\>\triangleleft\t_j\<k\>$ and
$\t_j\triangleright\t_{j+1}$. In this case it follows again from the definition of $B_{i}(\s)$ in Proposition \ref{tiact} that $B_{i_{j}}(\t_{j-1}\sft{k})=B_{i_{j}}(\t_{j})$.
Therefore, we can get that
	$$r_{\v}=r_{\t^{\blam}}\cdot \prod_{j=1}^{m}(B_{i_{j}}(\t_{j}))^{-1}=r_{\t^{\blam}}\frac{\gamma_{\t^{\blam}}}{\gamma_{\m_{k}(\v)}},$$
	where the standard tableau $\m_{k}(\v)$ is defined in Definition \ref{midkt}.
	
	Set $r_{\t^{\blam}}:=\gamma_{\t^{\blam}}^{-1}$. Then we can get $r_{\v}=\gamma_{\m_{k}(\v)}^{-1}$. We define $$
F_{\blam,k}:=\sum_{\v\in\Std(\blam)}\gamma_{\m_{k}(\v)}^{-1}f_{\v\sft{k}\v}. $$
	
By the above discussions, it is straightforward to check that $F_{\blam,k}\in Z(\HH_{r,n})^{(k)}$ and $Z(\HH_{r,n})^{(k)}$is $K$-spanned by $\{F_{\blam,k} | \blam\in\P_{r,n}, \blam\sft{k}=\blam \}$.
It is easy to see that $\{ F_{\blam,k} | \blam\in\P_{r,n}, \blam\sft{k}=\blam \}$ are $K$-linearly independent. This proves the set $\{ F_{\blam,k}|\blam\in\P_{r,n}, \blam\sft{k}=\blam \}$ forms a $K$-basis of the
$\sigma$-twisted $k$-center $Z(\HH_{r,n})^{(k)}$. Since $\m_{0}(\v)=\v$, we see that $F_{\blam,0}$ coincides with the central primitive idempotent $F_{\blam}$ of $\HH_{r,n}$.
\hfill\qed
\medskip

\begin{cor}
	Suppose that $R=K$ is a field and (\ref{ssrpn}) holds. Let $k\in\Z/p\Z$. Then
	$$
\dim_{K}Z(\HH_{r,n})^{(k)}=\#\bigl\{\blam\in\P_{r,n}\bigm|o_{\blam}\,\, divides \,\, k\bigr\}.$$
\end{cor}

Let $K$ be a field and $1\neq\xi\in K^\times$. Suppose that $K$ contains a primitive $p$-th root of unity $\veps$. In particular, this implies that $p1_{K}\in K^\times$ is invertible. Let $Q_1,\cdots,Q_d\in K^\times$.
Let $e\in\Z^{\geq 0}$ be the minimal positive integer $k$ such that $1+\xi+\xi^2+\cdots+\xi^{k-1}=0$ in $K$; or $0$ if no such positive integer $k$ exists. Let $x$ be an indeterminate over $K$. For each $1\leq m\leq r=pd$,
we define $$
\hat{Q}_m:=(x^{jn}+Q_j)\veps^i,\,\,\text{if $m=(i-1)d+j, 1\leq i\leq p, 1\leq j\leq d$.}
$$
Set $\O:=K[x]_{(x)}$. Let $\HH_{r,n}(K)$ be the cyclotomic Hecke algebra of type $G(r,1,n)$ with Hecke parameter $\xi$ and cyclotomic parameters $$
\bigl(\underbrace{Q_1\veps,\cdots,Q_d\veps},\underbrace{Q_1\veps^2,\cdots,Q_d\veps^2},\cdots,\underbrace{Q_1\veps^p,\cdots,Q_d\veps^p}\bigr) .
$$
Let $\K$ be the fraction field of $\O$. For $R\in\{\O,\K\}$, we use $\HH_{r,n}(R)$ to denote the cyclotomic Hecke algebra of type $G(r,1,n)$, which is defined over $R$, and with Hecke parameter $\hat{\xi}:=x+\xi$ and
cyclotomic parameters $$
\bigl(\underbrace{\hat{Q}_1,\cdots,\hat{Q}_d},\underbrace{\hat{Q}_{d+1},\cdots,\hat{Q}_{2d}},\cdots,\underbrace{\hat{Q}_{r-d+1},\cdots,\hat{Q}_r}\bigr) .
$$
Let $\HH_{r,p,n}(K)$ be the subalgebra of  $\HH_{r,n}(K)$ generated by $$T_{0}^{p}, T_{u}:=T_{0}^{-1}T_{1}T_{0}, T_{1}, T_{2}, \cdots, T_{n-1}.$$
Similarly, let $\HH_{r,p,n}(R)$ be the subalgebra of  $\HH_{r,n}(R)$ generated by $$T_{0}^{p}, T_{u}:=T_{0}^{-1}T_{1}T_{0}, T_{1}, T_{2}, \cdots, T_{n-1}.$$ There are natural isomorphisms: $$\begin{aligned}
& K\otimes_{\O}\HH_{r,n}(\O)\cong\HH_{r,n}(K),\quad K\otimes_{\O}\HH_{r,p,n}(\O)\cong\HH_{r,p,n}(K),\\
& \K\otimes_{\O}\HH_{r,n}(\O)\cong\HH_{r,n}(\K),\quad \K\otimes_{\O}\HH_{r,p,n}(\O)\cong\HH_{r,p,n}(\K) .
\end{aligned}
$$
Note that the automorphisms $\sigma, \tau$ of $\HH_{r,n}(K)$ can also be defined on the $R$-algebra $\HH_{r,n}(R)$ for $R\in\{\O,\K\}$. We denote them by $\sigma_R, \tau_R$ respectively.
Applying (\ref{ssrpn}), we can deduce that both $\HH_{r,n}(\K)$ and $\HH_{r,p,n}(\K)$ are split semisimple.

\begin{lem} We have $$
\dim_K Z(\HH_{r,p,n}(K))\geq\dim_{\K}Z(\HH_{r,p,n}(\K))=\rank_{\O}Z(\HH_{r,p,n}(\O)),
$$
and for any $k\in\Z/p\Z$, $$
\dim_K Z(\HH_{r,n}(K))^{(k)}\geq\dim_{\K}Z(\HH_{r,n}(\K))^{(k)}=\rank_{\O}Z(\HH_{r,p,n}(\O))^{(k)}.
$$
\end{lem}

\begin{proof} By the definition of $Z(\HH_{r,p,n})$, $\dim_{\K}Z(\HH_{r,p,n}(\K))$ (resp., $\dim_K Z(\HH_{r,p,n}(K))$) are equal to the $\K$-dimension (resp., the $K$-dimension) of a solution space of a homogeneous linear
system
of equations whose coefficient matrix is defined over $\O$. By general theory of linear algebras, we see that  $$
\dim_K Z(\HH_{r,p,n}(K))\geq\dim_{\K}Z(\HH_{r,p,n}(\K))=\rank_{\O}Z(\HH_{r,p,n}(\O)).
$$
This proves the first part of the lemma. The second part of the lemma can be proved in a similar way.
\end{proof}

\begin{prop}\label{KeyProp} Suppose that for any $k\in\Z/p\Z$, we have $$ \dim_K Z(\HH_{r,n}(K))^{(k)}=\dim_{\K}Z(\HH_{r,n}(\K))^{(k)}=\rank_{\O}Z(\HH_{r,p,n}(\O))^{(k)}.$$
Then we have that $$
\dim_K Z(\HH_{r,p,n}(K))=\dim_{\K}Z(\HH_{r,p,n}(\K))=\rank_{\O}Z(\HH_{r,p,n}(\O)) .
$$
\end{prop}

\begin{proof} By Lemma \ref{weightdecomp}, there is a $K$-subspaces decomposition: $$
Z(\HH_{r,n}(K))^{(k)}=\bigoplus_{l\in\Z/p\Z}Z^{\tau,l}(\HH_{r,p,n}(K))_{k}T_{0}^{l},\,\,
Z(\HH_{r,n}(\O))^{(k)}=\bigoplus_{l\in\Z/p\Z}Z^{\tau,l}(\HH_{r,p,n}(\O))_{k}T_{0}^{l}.
$$
Since $Z(\HH_{r,n}(\O))^{(k)}$ is a pure $\O$-submodule of $\HH_{r,n}(\O)$, it follows that the canonical maps $$
\psi_1: K\otimes_{\O}Z(\HH_{r,n}(\O))^{(k)}\rightarrow Z(\HH_{r,n}(K))^{(k)},\,\, \psi_2: K\otimes_{\O}Z^{\tau,l}(\HH_{r,p,n}(\O))_{k}T_{0}^{l}\rightarrow Z^{\tau,l}(\HH_{r,p,n}(K))_{k}T_{0}^{l}
$$
are both embedding.

By assumption, $\dim_K Z(\HH_{r,n}(K))^{(k)}=\dim_{\K}Z(\HH_{r,n}(\K))^{(k)}=\rank_{\O}Z(\HH_{r,p,n}(\O))^{(k)}$. It follows that $\psi_1$ is an isomorphism. Applying Lemma \ref{weightdecomp}, we have $$
Z(\HH_{r,n}(\O))^{(k)}=\bigoplus_{l\in\Z/p\Z}Z^{\tau,l}(\HH_{r,p,n}(\O))_{k}T_{0}^{l},\,\, Z(\HH_{r,n}(K))^{(k)}=\bigoplus_{l\in\Z/p\Z}Z^{\tau,l}(\HH_{r,p,n}(K))_{k}T_{0}^{l}.
$$
It follows that the injection $\psi_2$ is an isomorphism for each $k\in\Z/p\Z$. Taking $l=0$, we can deduce that $$
\dim_K Z(\HH_{r,p,n}(K))_k=\rank_\O Z(\HH_{r,p,n}(\O))_k,\,\,\forall\,k\in\Z/p\Z .
$$
Finally, using (\ref{Zdecompositi0}), we prove the proposition.
\end{proof}

\begin{conj}\label{conj1} For any $k\in\Z/p\Z$, we have $$
\dim_K Z(\HH_{r,n}(K))^{(k)}=\dim_{\K}Z(\HH_{r,n}(\K))^{(k)},\,\,  \dim_K Z(\HH_{r,p,n}(K))=\dim_{\K}Z(\HH_{r,p,n}(\K)).
$$
\end{conj}

\begin{rem}\label{finalrem} We remark that Proposition \ref{KeyProp} gives rise to an approach to verify Conjecture \ref{conj1}. In fact, using Lemma \ref{dual2} and Theorem \ref{mainthm5}, it suffices to show that
$\HH_{r,n}/[\HH_{r,n},\HH_{r,n}]^{(k)}$ has dimension less than or equal to $\#\bigl\{\blam\in\P_{r,n}\bigm|\text{$o_{\blam}$ divides $k$}\bigr\}$ for each $k\in\Z/p\Z$. Equivalently, it suffices to construct a $K$-linearly
spanning set with cardinality equal to $\#\bigl\{\blam\in\P_{r,n}\bigm|\text{$o_{\blam}$ divides $k$}\bigr\}$ for the quotient space $\HH_{r,n}/[\HH_{r,n},\HH_{r,n}]^{(k)}$ for each $k\in\Z/p\Z$. It seems very likely that
one can use a similar argument used in \cite{HS} to prove this claim. We shall leave this to a future study.
\end{rem}

\end{document}